\documentclass[aap]{imsart}

\RequirePackage{amsthm,amsmath,amsfonts,amssymb}
\RequirePackage[numbers]{natbib}
\RequirePackage[colorlinks,citecolor=blue,urlcolor=blue]{hyperref}
\RequirePackage{graphicx}
\usepackage[utf8]{inputenc}
\usepackage{amsfonts}
\usepackage{amsthm}
\usepackage{amsmath}
\usepackage{amssymb}
\usepackage{bbm}
\usepackage{enumerate}
\usepackage{fancyhdr} 
\usepackage[pdftex,usenames,dvipsnames]{color}
\usepackage{url}
\usepackage{fancyhdr}
\usepackage[pdftex]{color}
\usepackage{dsfont}
\usepackage{mathtools}
\usepackage{verbatim}

\DeclarePairedDelimiter\floor{\lfloor}{\rfloor}
\usepackage{xcolor}

\usepackage{mathrsfs}
\newcommand{\nt}{\lfloor nt \rfloor}
\newcommand{\ns}{\lfloor ns \rfloor}

\newcommand{\N}{\mathbb{N} }

\newcommand{\Sp}{\mathbb{S}_p}

\newcommand{\Prob}{\mathbb{P}}

\providecommand{\abs}[1]{\lvert #1\rvert}
\providecommand{\babs}[1]{\bigl\lvert #1\bigr\rvert}

\providecommand{\norm}[1]{\lVert #1\rVert}

\providecommand{\Enorm}[1]{\lVert #1\rVert_2}

\DeclareMathOperator{\Cov}{Cov}

\DeclareMathOperator{\Var}{Var}
\newcommand{\E}{\mathbb{E}}
\newcommand{\R}{\mathbb{R}}
\newcommand{\D}{\mathcal{D}}

\newcommand{\B}{\mathcal{B}}
\newcommand{\F}{\mathcal{F}}
\renewcommand{\P}{\mathbb{P}}
\renewcommand{\phi}{\varphi}

\newcommand{\bdm}{\begin{displaymath}}
\newcommand{\edm}{\end{displaymath}}

\usepackage{hyperref}
\hypersetup{
     colorlinks   = true,
     citecolor    = blue
}
\allowdisplaybreaks

\DeclareMathOperator*{\argmax}{argmax}

\numberwithin{equation}{section}

\newcommand{\1}{\mathds{1}}
\startlocaldefs

\newtheorem{theorem}{Theorem}[section]
\newtheorem{lemma}[theorem]{Lemma}
\theoremstyle{remark}
\newtheorem{remark}[theorem]{Remark}
\newtheorem{corollary}[theorem]{Corollary}
\newtheorem{proposition}[theorem]{Proposition}

\endlocaldefs

\begin{document}

\begin{frontmatter}
\title{Functional Convergence of Sequential $U$-Processes \\ with Size-Dependent Kernels}
\runtitle{Functional Convergence of Sequential $U$-Processes}

\begin{aug}
\author[A]{\fnms{Christian} \snm{D\"obler}\ead[label=e1]{christian.doebler@hhu.de}},
\author[B]{\fnms{Miko{\l}aj J.} \snm{Kasprzak}\ead[label=e2,mark]{mikolaj.kasprzak@uni.lu}}
\and
\author[B]{\fnms{Giovanni} \snm{Peccati}\ead[label=e3,mark]{giovanni.peccati@uni.lu}}
\address[A]{Mathematisches Institut, Heinrich-Heine-Universit\"at D\"usseldorf,
\printead{e1}}

\address[B]{Department of Mathematics, University of Luxembourg,
\printead{e2,e3}}
\end{aug}

\begin{abstract}
We consider sequences of $U$-processes based on symmetric kernels of a fixed order, that possibly depend on the sample size. Our main contribution is the derivation of a set of analytic sufficient conditions, under which the aforementioned $U$-processes weakly converge to a linear combination of time-changed independent Brownian motions. In view of the underlying symmetric structure, the involved time-changes and weights remarkably depend only on the order of the $U$-statistic, and have consequently a universal nature. Checking these sufficient conditions requires calculations that have roughly the same complexity as those involved in the computation of fourth moments and cumulants. As such, when applied to the degenerate case, our findings are infinite-dimensional extensions of the central limit theorems (CLTs) proved in de Jong (1990) and D\"obler and Peccati (2017). As important tools in our analysis, we exploit the multidimensional central limit theorems established in D\"obler and Peccati (2019)  together with upper bounds on absolute moments of degenerate $U$-statistics by Ibragimov and Sharakhmetov (2002), and also prove some novel multiplication formulae for degenerate symmetric $U$-statistics --- allowing for different sample sizes --- that are of independent interest. We provide applications to random geometric graphs and to a class of $U$-statistics of order two, whose Gaussian fluctuations have been recently studied by Robins {\it et al.} (2016), in connection with quadratic estimators in non-parametric models. In particular, our application to random graphs yields a class of new functional central limit theorems for subgraph counting statistics, extending previous findings in the literature. Finally, some connections with invariance principles in changepoint analysis are established.
\end{abstract}



\end{frontmatter}

\section{Introduction}

\subsection{Overview}

Consider a sequence $\{X_i : i=1,2,...\}$ of independent and identically distributed (i.i.d.) random variables with values in some space $(E,\mathcal{E})$. The aim of this paper is to prove a class of Gaussian functional central limit theorems (FCLTs) involving general sequences of {\bf sequential $U$-processes} with symmetric kernels, that is, of c\`adl\`ag processes on the time interval $[0,1]$, obtained by progressively revealing the argument of a symmetric $U$-statistic of order $p\geq 1$ based on the sample $(X_1,...,X_n)$, for $n\geq 1$. The adjective \textit{sequential} is meant to distinguish the main objects studied in the present paper from $U$-processes indexed by abstract classes of kernels -- see e.g. \cite{DP87, AG93}; for the sake of simplicity, for the rest of the paper we will use indifferently the denomination ``$U$-process'' to indicate sequential $U$-processes and function-indexed $U$-processes of the type described above. 

{}

The type of weak convergence we deal with is in the large sample limit $n\to\infty$, and holds in the sense of the Skorohod space $D[0,1]$ of c\`adl\`ag mappings on $[0,1]$, endowed with Skorohod's $J_1$ topology (see e.g. \cite[p. 123]{Bil}). The specific difficulty tackled in our work --- marking a difference with previous contributions (see e.g. \cite{Neu, H79, MT, DP87, AG93}) --- is that we allow the kernels of the considered $U$-statistics to explicitly depend on the sample size $n$, and we do not assume a priori any form of Hoeffding degeneracy.

{}

Despite the generality of the above setup, the limit processes displayed in our results always have the form
\begin{equation}\label{e:limsum}
Z(t) = \sum_{k=1}^p \alpha_{k,p} Z_{k,p} (t), \quad t\in [0,1],
\end{equation}
where each $\alpha_{k,p}\in [0,\infty)$ is a constant depending on the sequence of $U$-statistics under study, and $\{Z_{k,p}(t) : t\in[0,1], \, 1\leq k\leq p\}$ denotes a class of independent centered Gaussian processes obtained as follows: first consider a sequence $\{B_{k}(t) : t\in [0,1],\, 1\leq k\leq p \}$ of independent standard Brownian motions on $[0,1]$, and then set
\begin{equation}\label{e:limpro}
Z_{k,p}(t) := t^{p-k}B_{k} (t^k), \quad t\in[0,1],
\end{equation}
in such a way that 
\begin{equation}\label{e:limcov}
\Gamma_{k,p}(s,t) := \E[Z_{k,p}(t)Z_{k,p}(s)] = (s\wedge t)^p  (s\vee t)^{p-k}, \quad s,t\in [0,1],
\end{equation}
and consequently
\begin{equation}\label{e:limcov2}
\Gamma(s,t) := \E[Z(t)Z(s)] =\sum_{k=1}^p \alpha_{k,p}^2 \, (s\wedge t)^p  (s\vee t)^{p-k}, \quad s,t\in [0,1].
\end{equation}
Note that, in particular, the processes $(Z(t))_{t\in[0,1]}$ appearing in \eqref{e:limsum} all have continuous paths.
Such a rigid asymptotic structure originates from the fact that we exclusively focus on {\it symmetric} $U$-statistics and i.i.d. samples: these strong assumptions yield then the emergence of the `universal' time-changes $t\mapsto t^k$ and time-dependent weights $t\mapsto t^{p-k}$ from purely combinatorial considerations. One should compare this situation with the reference \cite{B94}, where a Gaussian FCLT is proved for sequences of non-symmetric homogenous sums, displaying as possible weak limits arbitrarily time-changed Brownian motions. {} 

The sufficient conditions for weak convergence discussed above are stated in the forthcoming Theorem \ref{maintheo} and Theorem \ref{maintheo2}, and are expressed in terms of the {\bf contraction kernels} canonically associated with a given $U$-statistic (see Section \ref{ss:contractions} for definitions). We will see in Section \ref{s:proofs} that the conditions derived in our paper are a slight strengthening of the sufficient conditions for one-dimensional CLTs derived in \cite[Section 3]{DP18} --- see Remark \ref{r:aa} below for a full discussion of this point.  Some of the additional requirements with respect to \cite{DP18} (in particular, Assumption (a) in Theorem \ref{maintheo} and Assumption (a') in Theorem \ref{maintheo2}) are necessary and sufficient for the pointwise convergence of the covariance functions of the degenerate $U$-processes associated with a generic $U$-statistic via its Hoeffding decompositon: such a technical assumption (that can in principle be relaxed at the cost of more technical statements --- see Remark \ref{r:aa}-(ii)) is unavoidable in the case of degenerate $U$-statistics, and is automatically verified in the applications developed in Section \ref{applis}. 
{}

As discussed in Section \ref{s:proofs} --- and similarly to the main findings of \cite{DP18} --- when applied to Hoeffding degenerate $U$-processes (see Section \ref{setting}), the conditions expressed in Theorems \ref{maintheo} and \ref{maintheo2} are roughly equivalent to the requirement that the joint cumulants of order $\leq 4$ associated with the finite-dimensional restrictions of the $U$-processes under consideration converge to those of an appropriate Gaussian limit, and that such a convergence takes place at a rate of the type $O(1/n^{\alpha})$, where $\alpha >0$.  As such, our findings can be regarded as functional versions of the well-known {\bf de Jong CLT} for degenerate $U$-statistics, first proved in \cite{dej} and then substantially extended in \cite{dobler_peccati} (to which we refer for an overview of the relevant literature). To the best of our knowledge, apart from the reference \cite{B94} (that only deals with homogeneous sums), ours is the first functional version of de Jong's CLT proved in the literature.

{}

{The main findings of our paper also contain an invariance principle by Miller and Sen \cite[Theorem 1]{MS72}, which can be obtained from our results by considering symmetric kernels of order $p$ that do not depend on the sample size and are not degenerate in the sense of Hoeffding, {see Remark \ref{r:whattocheck1} (v).} In this case, when $p\geq 2$ the limiting process \eqref{e:limsum} is such that $\alpha_{k,p} = 0$ for $k=2,...,p$, and $Z$ reduces to a multiple of $t^{p-1} B(t)$, where $B$ is a standard Brownian motion, thus yielding in particular a full version of Donsker's Theorem for sums of i.i.d. random variables in the case $p=1$ (see \cite[p. 90]{Bil}). It is important to notice that this result from \cite{MS72} only requires the square-integrability of the involved $U$-statistics, {whereas our approach is originally best suitable to deal with $U$-statistics whose kernels satisfy some higher moment assumptions.
However, in this particular example, by means of a truncation argument it is possible to reduce the assumptions to the minimal conditions of non-degeneracy and square-integrability of the kernel. 
The stricter moment requirements emerging elsewhere in our work are a direct emanation of the tools we chose to adopt, namely Stein's method and $L^2$ estimates on contraction kernels, which cannot be dropped without replacement, in general.} It is plausible that alternate statements requiring less stringent integrability assumptions could be obtained by following a different route --- but this would demand a significant amount of novel ideas. The reader is referred to \cite{J18} for generalizations of Miller and Sen's results and for a discussion of the relevant literature. See also Remark \ref{r:donsker}.

}

{}

In the last four decades, numerous FCLTs for $U$-processes have been derived by several groups of authors; yet --- to the best of our knowledge and discounting Miller and Sen's result --- none of them has a nature that is directly comparable to our findings. Among the large set of contributions in this domain, we refer the reader to the following relevant sample. References \cite{Neu, H79, MT} contain functional limit theorems for sequences of degenerate $U$-statistics with a kernel independent on the sample size $n$: in such a framework, consistently with the known one-dimensional results (see e.g. \cite{DM}), the limit process lives in a Wiener chaos of order $>1$ and it is therefore non-Gaussian. The already mentioned paper \cite{B94} proves Gaussian FCLTs (in a spirit close to \cite{dej}) for sequences of homogeneous sums: in this case, there is no overlap with our work since symmetric homogeneous sums (that are in principle contemplated in our setting) are necessarily multiples of degenerate $U$-statistics with a kernel not depending on the sample size $n$, and their asymptotic behaviour is consequently non-Gaussian (by virtue of \cite{Neu, H79}). References \cite{DP87, AG93} are influential general contributions to the theory of $U$-processes, containing uniform FCLTs for sequences of $U$-processes indexed by function classes not depending on the sample size, both in the non-degenerate and degenerate case. The recent contribution \cite{CK2017} deals with suprema of $U$-processes indexed by non-degenerate symmetric function classes possibly depending on the sample size and not necessarily verifying a FCLT, and also contains a detailed review of further relevant literature. See also \cite{GM07, chen}, as well as \cite{Bor_book, DlPG} for general references.

{}

As discussed above, our main results are expressed in terms of explicit analytical quantities (e.g. norms of contraction kernels), and they are therefore particularly well-adapted to applications. As a demonstration of this fact, in Section \ref{applis} we deduce two new classes of FCLTs, respectively related to subgraph counting in geometric random graphs (retrieving novel functional versions of one-dimensional CLTs from \cite{JJ, Penrose, BhaGo92, DP18})), and to quadratic $U$-statistics emerging e.g. in the non-parametric estimation of quadratic functionals of compactly supported densities (see e.g. \cite{BR88, LM00, van_der_vaart}). In Section \ref{ss:cgpt}, we also illustrate some connections with invariance principles related to changepoint analysis, see e.g. \cite{CH88, CH_book, Gombay2004}.

{}

We eventually mention the challenging problem of deriving explicit rates of convergence for the FCLTs derived in this paper. While some promising partial results seem to be obtainable by adapting the infinite-dimensional `generator approach' to Stein's method developed in \cite{barbour, barbour2009, kasprzak2018}, we prefer to consider this point as a separate issue, and leave it open for subsequent research.

\subsection{Notation and tightness criteria}\label{ss:nt}

From now on, every random object considered in the paper is assumed to be defined on a common probability space $(\Omega, \mathcal{F}, \P)$, with $\E$ denoting expectation with respect to $\P$. Given a collection of stochastic processes $\{X, \, X_n : n\geq 1\}$ with values in $D[0,1]$, we write $X_n \Longrightarrow X$ to indicate that $X_n$ {\bf weakly converges} to $X$, meaning that $\E[\varphi(X_n)]\to \E[\varphi(X)]$, as $n\to\infty$, for every bounded mapping $\varphi : D[0,1] \to \R$ which is continuous with respect to the Skorohod topology. Given two positive sequences $\{a_n,  b_n\}$, we write $a_n \sim b_n$ whenever $a_n/b_n\to 1$, as $n\to \infty$. We will also use the notation $a_n \lesssim b_n$ to indicate that there exists an absolute finite constant $C$ such that $a_n \leq C \,  b_n$ for every $n$.

{}

In several places of the present paper,  tightness in the space $D[0,1]$ is established by using the following criterion. The argument in the proof reproduces the strategy adopted in \cite[Lemma 3.1]{NouNu}, and is reported for the sake of completeness.

\begin{lemma}\label{lemma_tightness}
Let $X_n = \{X_n(t) : t\in [0,1]\}$, $n\in \N$, be a sequence of stochastic processes with paths a.s. in $D[0,1]$. 
Suppose that there is a stochastic process $X=\{X(t) : t\in [0,1]\}$ such that the finite-dimensional distributions of $X_n$, $n\in\N$, converge to those of $X$, as $n\to\infty$. Then, the paths of $X$ are a.s. continuous and we have $X_n\Longrightarrow X$,
if there are constants $C>0, \, \beta >0$ and $\alpha>0$ such that, for all $n\in\N$  sufficiently large and for all $0\leq s<t\leq1$, 
\begin{align}\label{Bil2.1}
 \E\babs{X_n(t)-X_n(s)}^{\beta}&\leq  C\biggl(\frac{\floor{nt}-\floor{ns}}{n}\biggr)^{1+\alpha}\,.
\end{align}
In particular, in this case the sequence $(X_n)_{n\in\N}$ is tight in $D[0,1]$.
\end{lemma}
\begin{proof}
We are going to use the following well-known criterion from \cite{Bil}: let $X_n = \{X_n(t) : t\in [0,1]\}$, $n\in \N$, be a sequence of stochastic processes with paths a.s. in $D[0,1]$ such that there is a stochastic process $X=\{X(t) : t\in [0,1]\}$ whose paths are a.s. continuous and such that the finite-dimensional distributions of $X_n$, $n\in\N$, converge to those of $X$, as $n\to\infty$.
Then, the sequence $(X_n)_{n\in\N}$, converges in distribution with respect to the Skorohod topology to $X$, if there are finite and strictly positive constants $C_1,\, \alpha$ and $\gamma$ such that, for all $n\in\N$  sufficiently large and for all $0\leq r\leq s\leq t\leq1$, 
\begin{align}\label{Bil}
 \E\Bigl[\babs{X_n(t)-X_n(s)}^\gamma \babs{X_n(s)-X_n(r)}^\gamma\Bigr] &\leq C_1 (t-r)^{1+\alpha} \,.
\end{align}
Note that \eqref{Bil} is in fact a more specialized instance of formula (13.14) in \cite{Bil}. Now assume \eqref{Bil2.1} and observe that
\begin{align}\label{tight3}
 &\E\Bigl[\babs{X_n(t)-X_n(s)}^{\beta/2} \babs{X_n(s)-X_n(r)}^{\beta/2}\Bigr]\notag \\
 &\leq \sqrt{\E\Bigl[ \babs{X_n(t)-X_n(s)}^{\beta}\Bigr]} \sqrt{\E\Bigl[ \babs{X_n(s)-X_n(r)}^{\beta}\Bigr]}\notag\\
 &\leq\sqrt{C \biggl(\frac{\floor{nt}-\floor{ns}}{n}\biggr)^{1+\alpha}}\sqrt{C \biggl(\frac{\floor{ns}-\floor{nr}}{n}\biggr)^{1+\alpha}}\notag\\
  &\leq 3^{1+\alpha}C(t-r)^{1+\alpha}\,,
\end{align}
where the last inequality follows from an argument used in the proof of \cite[Lemma 3.1]{NouNu}. Hence, we conclude that \eqref{Bil} holds true with $\gamma=\beta/2$ and with $C_1=3^{1+\alpha}C$. 
It remains to show that \eqref{Bil2.1} implies that $X$ has a continuous modification. Indeed, from the convergence of finite-dimensional distributions, the continuous mapping theorem and since $\E|Y|^r\leq\liminf_{n\to\infty} \E|Y_n|^r$
for all $r>0$, whenever $Y_n$, $n\in\N$, converges in distribution to $Y$, we conclude from \eqref{Bil2.1} that for all $0\leq s<t\leq1$ we have that
\begin{align*}
\E\babs{X(t)-X(s)}^{\beta}&\leq \liminf_{n\to\infty}\E\babs{X_n(t)-X_n(s)}^{\beta}
\leq C(t-s)^{1+\alpha}\,,
\end{align*}
where $C,\alpha$ and $\beta$ are as in \eqref{Bil2.1}. Thus, the continuity of $X$ follows from the Kolmogorov-Chentsov theorem.
\end{proof}

\subsection{Plan}
Section \ref{setting} contains some general information about U-statistics and U-processes and several useful estimates for contraction operators. The main results of the paper, which give sufficient conditions for functional convergence of U-processes, are presented in Section \ref{ss:main};  some connections to changepoint analysis are described in Section \ref{ss:cgpt}, whereas multidimensional extensions are discussed in Section \ref{mult}. Section \ref{applis} deals with applications of our main results to subgraph counting in random geometric graphs and U-statistics of order 2 with a dominant diagonal component. Finally, Section \ref{s:proofs} contains some further ancillary results, as well as the proofs of the main results.

\section{General setup}\label{setting}

\subsection{Symmetric $U$-statistcs and $U$-processes}

As before, we assume that $X_1,X_2,\dotsc$ is a sequence of i.i.d. random variables taking values in a measurable space $(E,\mathcal{E})$ (that we fix for the rest of this paper), and denote by $\mu$ their common distribution. 
For a fixed $p\in\N$, let 
\[\psi:\bigl(E^p,\mathcal{E}^{\otimes p}\bigr)\rightarrow\bigl(\R,\B(\R)\bigr)\] 
be a symmetric and measurable kernel of order $p$. By ``symmetric'' we mean that, for all $x=(x_1,\dotsc,x_p)\in E^p$ and each $\sigma\in\Sp$ (the symmetric group acting on $\{1,\dotsc,p\}$), one has that 
\[\psi(x_1,\dotsc,x_p)=\psi(x_{\sigma(1)},\dotsc,x_{\sigma(p)})\,.\]
In general, the kernel $\psi = \psi^{(n)}$ can (and will most of the time) depend on an additional parameter $n\in\N$ (the size of the sample in the argument of the associated $U$-statistic), but we will often suppress such a dependence, in order to simplify the notation. We use the symbol $\mu^p$ to denote the $p$-th power of $\mu$ (which is a measure on $(E^p,\mathcal{E}^{\otimes p})$). In what follows, we will write $X:=\{ X_i : i\in\N\}$ and, for $p, \psi, X$ as above and $n\in\N$, we define
\begin{align}\label{e:sus}
 J_p^{(n)}(\psi)&:=J_{p,X}^{(n)}(\psi):= \sum_{J\in\D_p(n)}\psi(X_j,j\in J)=\sum_{1\leq i_1<\dotsc<i_p\leq n}\psi(X_{i_1},\dotsc,X_{i_p})\,.
\end{align}
where $\D_p(n)$ indicates the subsets of size $p$ of $\{1,...,n\}$. We say that the random variable $ J_p^{(n)}(\psi)$ is the {\bf $U$-statistic of order $p$}, {\bf based on $X_1,\dotsc,X_n$ and generated by the kernel $\psi$}. For $p=0$ and a constant $c\in\R$ we further let $J_0(c):=c$. 

{}

Let $p\geq1$, and let $\psi\in L^1(\mu^{p})$ be symmetric. The kernel $\psi$ is called \textbf{(completely) degenerate} or \textbf{canonical} with respect to $\mu$, if 
\begin{equation*}
 \int_E\psi(x_1,x_2,\dotsc,x_p)d\mu(x_1)=0\quad\text{for } \mu^{p-1}\text{-a.a. }(x_2,\dotsc,x_p)\in E^{p-1}\,,
\end{equation*}
or, equivalently, if
\begin{equation*}
 \E\bigl[\psi(X_1,\dotsc,X_p)\,\bigl|\,X_1,\dotsc,X_{p-1}\bigr]=0,\quad\Prob\text{-a.s..}
\end{equation*}
Now assume that $\psi\in L^1(\mu^{p})$ is a symmetric but not necessarily degenerate kernel.
In this case, the random variable $J_p^{(n)}(\psi)$ can be written as the sum of its expectation and a linear combination of symmetric $U$-statistics with degenerate kernels of respective orders $1,\dotsc,p$. More precisely, one has the following \textbf{Hoeffding decomposition} of $J_p^{(n)}(\psi)$:
\begin{align}\label{HDsym}
 J_p(\psi)&=\E\bigl[J_p(\psi)\bigr]+\sum_{k=1}^p \binom{n-k}{p-k} J_k(\psi_k)=\sum_{k=0}^p \binom{n-k}{p-k} J_k(\psi_k)\,,
\end{align}
where 
\begin{align}\label{defpsis}
\psi_k(x_1,\dotsc,x_k)&=\sum_{l=0}^k(-1)^{k-l}\sum_{1\leq i_1<\dotsc<i_l\leq k} g_l(x_{i_1},\dotsc,x_{i_l})
 \end{align}
and the symmetric functions $g_l:E^l\rightarrow\R$ are defined by 
\begin{equation}\label{gk}
 g_l(y_1,\dotsc,y_l):=\E\bigl[\psi(y_1,\dotsc,y_l,X_1,\dotsc,X_{p-l})\bigr]\,,
\end{equation}
in such a way that, for $1\leq k\leq p$, $\psi_k$ is symmetric and degenerate of order $k$. 
In particular, one has $g_0\equiv\psi_0\equiv \E\bigl[\psi(X_1,\dotsc,X_p)\bigr]$ and $g_p=\psi$. See e.g. \cite{serfling, Vit92} for general references on Hoeffding decompositions. 

{}

Similarly to \eqref{e:sus}, we can naturally define a \textbf{$U$-process} 
\begin{equation}\label{e:ut0} 
U = \{ U(t) : {t\in[0,1]}\}
\end{equation}
as follows. For $t\in[0,1]$, write $ U(t):=U_n(t):=J_p^{(\floor{nt})}(\psi)$. Then, for every $t$ one has the Hoeffding decomposition
\begin{equation}\label{e:ut}
 U(t):=\E[U(t)]+ \sum_{k=1}^p \binom{\floor{nt}-k}{p-k} J_k^{(\floor{nt})}(\psi_k).
\end{equation}
Whenever $\psi$ is a symmetric element of $L^2(\mu^p)$, we will also make use of the notation 
\begin{equation}\label{genW}
 W(t):=W_n(t)=\frac{U(t)-\E[U(t)]}{\sigma_n}\,, \quad t\in [0,1],
\end{equation}
where 
\begin{align}
 \sigma_n^2&:=\Var(U_n(1))=\Var(J_p^{(n)}(\psi))=\sum_{k=1}^p\binom{n-k}{p-k}^2\binom{n}{k}\|\psi_k\|_{L^2(\mu^{ k})}^2\label{sigma}\\
 &=\binom{n}{p}\sum_{k=1}^p\binom{p}{k}\binom{n-p}{p-k}\Var\bigl(g_k(X_1,\dotsc,X_k)\bigr)\label{sigma2} \,.
\end{align}
Setting
\begin{equation*}
 \phi^{(k)}:=\phi^{(n,k)}:=\frac{\binom{n-k}{p-k} \psi_k}{\sigma_n}\,, \quad 1\leq k\leq p,
\end{equation*}
one has that each $\phi^{(k)}$ is a degenerate kernel and, using the notation $V_k(t):= J_k^{(\floor{nt})}(\phi^{(k)})$, one infers the following useful representation of $W$
\begin{equation*}
 W(t)=\sum_{k=1}^p \frac{\binom{\floor{nt}-k}{p-k}}{ \binom{n-k}{p-k}}V_k(t)\,, \quad t\in [0,1].
\end{equation*}

{}

It is immediately verified that, for each $n\in\N$, both $ W_n := \{W_n(t): {t\in[0,1]}\}$ and $ U_n := \{ U_n(t) : t\in[0,1]\}$ are random elements with values in $D[0,1]$. As anticipated, the aim of this paper is to deduce verifiable analytical conditions, under which the sequence $\{W_n : n\in\N\}$ converges in distribution to some continuous Gaussian process $Z = \{ Z(t) : {t\in[0,1]}\}$ of the form \eqref{e:limsum}.

\subsection{Contractions}\label{ss:contractions}

{We will now define ``contraction kernels'' obtained from pairs of square-integrable mappings}. These are one of the principal analytical tools exploited in our paper. Given integers $p,q\geq1$, $0\leq l\leq r\leq p\wedge q$ and two symmetric kernels $\psi\in L^2(\mu^{ p})$ and $\phi\in L^2(\mu^{ q})$, we define the \textbf{contraction kernel} $\psi\star_r^l \phi$ on $E^{p+q-r-l}$ by {the relation}
\begin{align}
 &(\psi\star_r^l \phi)(y_1,\dotsc,y_{r-l},t_1,\dotsc,t_{p-r},s_1,\dotsc,s_{q-r})\notag\\
&:=\int_{E^l}\Bigl(\psi\bigl(x_1,\dotsc,x_l, y_1,\dotsc,y_{r-l},t_1,\dotsc,t_{p-r}\bigr)\notag\\
&\hspace{3cm}\cdot\phi\bigl(x_1,\dotsc,x_l, y_1,\dotsc,y_{r-l},s_1,\dotsc,s_{q-r}\bigr)\Bigr)d\mu^{ l}(x_1,\dotsc,x_l)\label{defcontr1}\\
&=\E\Bigl[\psi\bigl(X_1,\dotsc,X_l, y_1,\dotsc,y_{r-l},t_1,\dotsc,t_{p-r}\bigr)\notag\\
&\hspace{3cm}\cdot\phi\bigl(X_1,\dotsc,X_l, y_1,\dotsc,y_{r-l},s_1,\dotsc,s_{q-r}\bigr)\Bigr]\label{defcontr2}\, ,
\end{align}
{for every $(y_1,\dotsc,y_{r-l},t_1,\dotsc,t_{p-r},s_1,\dotsc,s_{q-r})$ belonging to the set $A_0 \subset E^{p+q-r-l}$ such that the right-hand side of the previous equation is well-defined and finite, and we conventionally set it equal to zero otherwise}. Given $\psi, \varphi, r, l$ as above, we write that the kernel $\psi\star_r^l \phi$ is {\bf well-defined} if $\mu^{ p+q-r-l}(A^c_0) =0$ (where $A_0$ is the set defined above). Note that, in general, it is not clear that $\psi\star_r^l \phi$ is well-defined in the sense specified above, or that the obtained contraction is square-integrable. 

{}

If $l=0$, then \eqref{defcontr1} has to be understood as follows: 
\begin{align*}
 &(\psi\star_r^0 \phi)(y_1,\dotsc,y_{r},t_1,\dotsc,t_{p-r},s_1,\dotsc,s_{q-r})\notag\\
 &=\psi(y_1,\dotsc,y_{r},t_1,\dotsc,t_{p-r})\phi( y_1,\dotsc,y_{r},s_1,\dotsc,s_{q-r})\,.
\end{align*}
In particular, if $l=r=0$, then $\psi\star_r^l \phi$ boils down to the \textbf{tensor product}
$\psi\otimes \phi:E^{p+q}\rightarrow\R$ of $\psi$ and $\phi$, given by 
\begin{align*}
 (\psi \otimes\phi)(x_1,\dotsc,x_{p+q})&=\psi(x_1,\dotsc,x_p)\cdot\phi(x_{p+1},\dotsc,x_{p+q})\,.
\end{align*}
We observe that $\psi\star_p^0\psi=\psi^2$ is square-integrable if and only if $\psi\in L^4(\mu^{ p})$. As a consequence, $\psi\star_r^l\phi$ might not be in $L^2(\mu^{ p+q-r-l})$ even though $\psi\in L^2(\mu^{ p})$ and $\phi\in L^2(\mu^{ q})$. Moreover, if $l=r=p$, then $\psi\star_p^p\psi=\|\psi\|_{L^2(\mu^{ p})}^2$ is constant. By Lemma 2.4 of \cite{DP18} the involved contraction kernels are always well-defined. Moreover, the same lemma gives various bounds on norms of contractions that are useful for the present paper.

In what follows, for $p\in\N$ and a function $f:E^p\rightarrow\R$ we write $\tilde{f}$ for its \textbf{symmetrization}, i.e. 
\begin{equation*}
 \tilde{f}(x_1,\dotsc,x_p):=\frac{1}{p!}\sum_{\sigma\in\mathbb{S}_p}f\bigl(x_{\sigma(1)},\dotsc,x_{\sigma(p)}\bigr)\,,\quad (x_1,\dotsc,x_p)\in E^p\,,
\end{equation*}
where $\mathbb{S}_p$ denotes the symmetric group acting on $\{1,\dotsc,p\}$. Note that, if $f\in L^q(\mu^p)$, then $\|\tilde{f}\|_{L^q(\mu^p)}\leq \|{f}\|_{L^q(\mu^p)}$, by the triangle inequality. 

\section{Weak convergence of $U$-processes with symmetric kernels}\label{ss:theo}

\subsection{Main results}\label{ss:main}

For the rest of this section, we let $p\geq 1$ be an integer. Moreover, for positive integers $1\leq r, i, k \leq p$ and $0\leq l\leq p$ such that $0\leq l\leq r\leq i\wedge k$, we let 
$Q(i,k,r,l)$ be the set of quadruples $(j,m,a,b)$ of nonnegative integers such that the following hold:
\begin{enumerate}
\item[ ] {\bf 1.} $j\leq i$ and $m\leq k$, \quad {\bf 2.} $b\leq a\leq r$, \quad {\bf 3.} $b\leq l$, \quad {\bf 4.} $a-b\leq r-l$, \quad 
\item[ ] {\bf 5.} $j+m-a-b\leq i+k-r-l\leq i+k-1$, \quad {\bf 6.} $a\leq j\wedge m$, 
\item[ ] {\bf 7. } if $j=m=p$, then $b=l$ and $a=r\geq1$.
\end{enumerate}
The next statement is the main result of the paper. 


\begin{theorem}[\bf Functional convergence of general $U$-statistics, I]\label{maintheo}
Let the assumptions and notation of Section \ref{setting} prevail, define the sequence $$W_n  = \{ W_n(t) : {t\in[0,1]}\} \, , \quad n\in\N,$$ according to \eqref{genW}, and assume that the following three conditions (expressed by means of the notation \eqref{gk}) are satisfied: 
\begin{enumerate}[{\normalfont (a)}]
 \item for all $1\leq k\leq p$, the real limit 
 \[b_k^2:=\lim_{n\to\infty}\frac{n^{2p-k}}{\sigma_n^2}\bigl(\norm{g_{k}}^2_{L^2(\mu^{k})}-(\E[\psi(X_1,\dotsc,X_p)])^2\bigr)=\lim_{n\to\infty}\frac{n^{2p-k}}{\sigma_n^2}\Var\bigl(g_k(X_1,\dotsc,X_k)\bigr)\] 
 exists; 
 \item for all $1\leq v\leq u\leq p$ and all pairs $(l,r)$ and quadruples $(j,m,a,b)$ of integers such that 
$1\leq r\leq v$, $0\leq l\leq r\wedge(u+v-r-1)$ and $(j,m,a,b)\in Q(v,u,r,l)$
 one has that
 \[\lim_{n\to\infty}\frac{n^{2p-\frac{u+v+r-l}{2}}}{\sigma_n^2}\,\norm{g_{j}\star_a^{b}g_{m}}_{L^2(\mu^{j+m-a-b})}=0\,;\]
 \item there exists some $\epsilon>0$, such that, for all $1\leq r\leq p$, all $0\leq l\leq r-1$ and all quadruples $(j,m,a,b)\in Q(r,r,r,l)$, the sequence
 \[\frac{n^{2p-r-\frac{r-l}{2}+\epsilon}}{\sigma_n^2} \,\norm{g_{j}\star_a^{b}g_{m}}_{L^2(\mu^{j+m-a-b})}\]
 is bounded.
\end{enumerate}
Then, as $n\to\infty$, one has that $W_n \Longrightarrow Z$, where $Z$ is the centered Gaussian process defined in \eqref{e:limsum}, for
$$
\alpha^2_{k,p} = \frac{b^2_k}{k!(p-k)!^2}\,, \quad 1\leq k\leq p. 
$$
%
%
\end{theorem}

\begin{remark}{\rm The contractions appearing at Point (c) of Theorem \ref{maintheo} also appear at Point (b) of the same statement. In particular the requirement at Point (c) can be rephrased by saying that, for all $(j,m,a,b)\in Q(r,r,r,l)$, the sequence
 \[\frac{n^{2p-r-\frac{r-l}{2}}}{\sigma_n^2} \,\norm{g_{j}\star_a^{b}g_{m}}_{L^2(\mu^{j+m-a-b})}, \quad n\geq 1,\]
converges to zero as $O(1/n^{\epsilon})$, for some $\epsilon >0$. A similar remark applies to Point (b') and Point (c') of Theorem \ref{maintheo2}.
}
\end{remark}

\begin{remark}\label{r:whattocheck}{\rm In the case $p=2$, after removing redundant terms, verifying conditions (b) and (c) of Theorem \ref{maintheo} boils down to checking that the following quantities converge to zero, as $n\to \infty$:
\begin{enumerate}
\item[] {\bf 1.} $\frac{n^2}{\sigma_n^2}\|g_2\|_{L^2(\mu^2)}\left|\mathbb{E}g_2(X_1,X_2)\right|$,\quad  {\bf 2.} $\frac{n^{2}}{\sigma_n^2}\|g_1\star_1^0 g_2\|_{L^2(\mu^2)}$, \quad {\bf 3.} $\frac{n^{5/2}}{\sigma_n^2}\|g_1\star_1^1 g_2\|_{L^2(\mu)}$,
\item[] {\bf 4.} $\frac{n^{3/2}}{\sigma_n^2}\|g_2\star_1^0 g_2\|_{L^2(\mu^3)}$, \quad {\bf 5.} $\frac{n^{2}}{\sigma_n^2}\|g_2\star_1^1 g_2\|_{L^2(\mu^2)}$, 
\item[] {\bf 6.} $\frac{n^{3/2}}{\sigma_n^{2}}\|g_2\star_0^0 g_1\|_{L^2(\mu^3)}=\frac{n^{3/2}}{\sigma_n^{2}}\|g_1\|_{L^2(\mu)}\|g_2\|_{L^2(\mu^2)}$,
\end{enumerate}
and that the following sequences are bounded for some $\epsilon>0$:
\begin{enumerate}
\item[] {\bf i)} $n\mapsto \frac{n^{5/2+\epsilon}}{\sigma_n^2}\left(\mathbb{E}g_2(X_1,X_2)\right)^2$, \quad {\bf ii)} $n\mapsto\frac{n^{5/2+\epsilon}}{\sigma_n^{2}}\|g_1\|_{L^2(\mu)}\left|\mathbb{E}g_2(X_1,X_2)\right|$,
\item[] {\bf iii)} $n\mapsto\frac{n^{5/2+\epsilon}}{\sigma_n^2}\|g_1\|_{L^4(\mu)}^2$,\quad {\bf iv)} $n\mapsto\frac{n^{1+\epsilon}}{\sigma_n^2}\|g_2\|_{L^4(\mu^2)}^2$.
\end{enumerate}
Recall also that in this case we have that $g_2=\psi$.}
\end{remark}
{}

When the Hoeffding decomposition of a given $U$-statistic is directly provided, it is more convenient to work with the kernels $\{\psi_k\}$ defined in formula \eqref{defpsis}, rather than with the class $\{g_k\}$. The next statement allows one to obtain the same conclusion as in Theorem \ref{maintheo} by uniquely checking conditions related to the family $\{\psi_k\}$ defined in \eqref{defpsis}.

\begin{theorem}[\bf Functional convergence of general $U$-statistics, II]\label{maintheo2}
The conclusion of Theorem \ref{maintheo} continues to hold if the following three conditions {\rm (a')}, {\rm (b')} and {\rm (c')} replace conditions {\rm (a)}, {\rm (b)} and {\rm (c)}: 
 \begin{enumerate}[{\normalfont (a')}]
 \item for all $1\leq k\leq p$, the real limit $b_k^2:=\lim_{n\to\infty}\frac{n^{2p-k}}{\sigma_n^2}\norm{\psi_{k}}^2_{L^2(\mu^{k})}$ exists;
  \item for all $1\leq v\leq u\leq p$ and all pairs $(l,r)$ of integers such that 
$1\leq r\leq v$, $0\leq l\leq r\wedge(u+v-r-1)$,
  \[\lim_{n\to\infty}\frac{n^{2p-\frac{u+v+r-l}{2}}}{\sigma_n^2}\,\norm{\psi_{v}\star_r^{l}\psi_{u}}_{L^2(\mu^{v+u-r-l})}=0\,;\]
  \item there exists some $\epsilon>0$, such that, for all $1\leq r\leq p$ and $0\leq l\leq r-1$, the sequence
  \[\frac{n^{2p-r-\frac{r-l}{2}+\epsilon}}{\sigma_n^2}\norm{\psi_{r}\star_r^{l}\psi_{r}}_{L^2(\mu^{r-l})}\]
 is bounded.
 \end{enumerate}
\end{theorem}

{}
\begin{remark}\label{r:whattocheck1} {\rm 
In case $p=2$, verifying conditions (b') and (c') of Theorem \ref{maintheo2} boils down to checking that the following quantities converge to zero as $n\to\infty$:
\begin{enumerate}
\item[] {\bf 1.} $\frac{n^2}{\sigma_n^2}\|\psi_1\star_1^0\psi_2\|_{L^2(\mu^2)}$, \quad {\bf 2.} $\frac{n^{5/2}}{\sigma_n^2}\|\psi_1\star_1^1\psi_2\|_{L^2(\mu)}$, \quad {\bf 3.} $\frac{n^{3/2}}{\sigma_n^2}\|\psi_2\star_1^0\psi_2\|_{L^2(\mu^{3})}$,\\{\bf 4.} $\frac{n^2}{\sigma_n^2}\|\psi_2\star_1^1\psi_2\|_{L^2(\mu^2)}$,
\end{enumerate}
and the following sequences are bounded for some $\epsilon>0$:
\begin{enumerate}
\item[] {\bf i)} $n\mapsto\frac{n^{5/2+\epsilon}}{\sigma_n^2}\|\psi_1\star_1^0\psi_1\|_{L^2(\mu)}$, \quad {\bf ii)} $n\mapsto\frac{n^{1+\epsilon}}{\sigma_n^2}\|\psi_2\star_2^0\psi_2\|_{L^2(\mu^2)}$, \\{\bf iii)} $n\mapsto \frac{n^{3/2+\epsilon}}{\sigma_n^2}\|\psi_2\star_2^1\psi_2\|_{L^2(\mu)}$.
\end{enumerate}}
\end{remark}
\begin{remark}\label{r:aa} {\rm
\begin{enumerate}

\item[(i) ]  By inspection of pur proofs, one sees that Conditons (a) and (b) in Theorem \ref{maintheo} (resp. Conditions (a') and (b') in Theorem \ref{maintheo2}) are sufficient conditions for the convergence of the finite-dimensional distributions of $W_n$ towards those of $Z$, whereas Conditons (c) and (c') therein imply tightness. Condition (b') in Theorem \ref{maintheo2} implicitly appears in \cite[Section 4 and Section 5]{DP18}, as an analytical sufficient condition ensuring that (in the notation of the present paper) $W_n(1)$ converges in distribution to a one-dimensional standard Gaussian random variable. On the other hand, Condition (b) in Theorem \ref{maintheo} is a substantial improvement of the sufficient conditions for one-dimensional asymptotic normality that can be deduced from \cite[Theorem 5.2]{DP18}. The difference between the conditions emerging from \cite[Theorem 5.2]{DP18} and those deduced in the present paper is explained by the fact that our findings use instead Lemma \ref{genulemma} below, which is a refined version of \cite[Lemma 5.1]{DP18}.
 
\item[ (ii)]  It will become clear from the discussion to follow that Condition (a) in Theorem \ref{maintheo} and Condition (a') in Theorem \ref{maintheo2} are equivalent for the same values of $b^2_k$, that is: for every $k=1,...,p$, the limit $\lim_{n\to\infty}\frac{n^{2p-k}}{\sigma_n^2}\Var\bigl(g_k(X_1,\dotsc,X_k)\bigr)$ exists and is finite if and only if the same holds for $\lim_{n\to\infty}\frac{n^{2p-k}}{\sigma_n^2}\norm{\psi_{k}}^2_{L^2(\mu^{k})}$, and in this case the two limits coincide.

\item[ (iii)] The proofs of Theorems \ref{maintheo} and \ref{maintheo2} will show that Conditions (b) and (c) in Theorem \ref{maintheo} imply Conditons (b') and (c') in Theorem \ref{maintheo2}, whereas the opposite implication does not hold in general. 

\item[(iv)] ({\it Relaxing Conditions} (a) {\it and} (a')) Suppose that all the assumptions of Theorem \ref{maintheo2} are verified, {\it except} for Condition (a'). In such a situation, we can define $ b_{k} ^2(n):= \frac{n^{2p-k}}{\sigma_n^2}\norm{\psi_{k}}^2_{L^2(\mu^{k})}$, $k=1,...,p$, and observe that, for every $k$, the mapping $n\mapsto b^2_{k}(n)$ is bounded (by virtue of \eqref{sigma}). For every $n\geq 1$, we set $Z_n$ to be the Gaussian process obtained from \eqref{e:limsum} by replacing the coefficient $\alpha_{k,p}^2$ with
$$
\alpha^2_{k,p} (n):= \frac{b^2_k(n)}{k!(p-k)!^2}\,, \quad 1\leq k\leq p. 
$$
A standard compactness argument now implies that, for every diverging sequence $\{ n_m \}\subset \N$, there exists a subsequence $\{n'_m\}\subset \{n_m\}$ such that $\alpha^2_{k,p} (n'_m)$ converges to a finite limit $\widehat{\alpha}^2_{k,p}$ as $n'_m\to \infty$ (the value of $\widehat{\alpha}^2_{k,p}$ depending in general on the choice of the subsequence $\{n'_m\}$), so that $Z_{n'_m}$ converges in distribution --- as a random element with values in $D[0,1]$ --- to the Gaussian process obtained from \eqref{e:limsum} by replacing the coefficient $\alpha_{k,p}^2$ with $\widehat{\alpha}^2_{k,p}$. Applying the triangle inequality together with Theorem \ref{maintheo2} to each extracted subsequence $n'_m$, we infer the following conclusion: if $\rho $ is any distance metrizing weak convergence on $D[0,1]$ (see e.g. \cite[Section 11.3]{Dudley}), then for every diverging sequence $\{ n_m \}\subset \N$ there exists a subsequence $\{n'_m\}\subset \{n_m\}$ such that
\begin{equation*}
\rho(W_{n'_m}, Z_{n'_m})\longrightarrow 0, \quad \mbox{as}\,\, n'_m\to \infty,
\end{equation*}
where $\rho(W_n, Z_n)$ is shorthand for the distance between the distributions of $W_n$ and $Z_n$, as random elements with values in $D[0,1]$; one sees immediately that this conclusion is equivalent to the asymptotic relation
\begin{equation}\label{e:rhocv}
\rho(W_n, Z_n)\longrightarrow 0, \quad \mbox{as}\,\, n\to \infty.
\end{equation}
By virtue of Points (ii)-(iii) of the present remark, the exact same conclusion holds if one supposes that all assumptions of Theorem \ref{maintheo} are verified, except for Condition (a). In view of the content of Point (i) above, it follows that Theorem \ref{maintheo} and Theorem \ref{maintheo2} contain and substantially extend the one-dimensional qualitative CLT stated in \cite[Section 5]{DP18}. One should notice that the techniques developed in  \cite{DP18} also allow one to deduce explicit rates of convergence, and that such a feature does not extend to our infinite-dimensional results.
\item[(v)] In \cite[Theorem 1]{MS72} Miller and Sen have proved an invariance principle for the situation of a non-degenerate kernel $\psi\in L^2(\mu^p)$ of order $p$ that does not depend on the sample size $n$. This in particular implies that $b_1^2=(p-1)!^2>0$ and $b_l=0$ for $l=2,\dotsc,p$. Moreover, \eqref{sigma2} implies that 
\[\sigma_n^2\sim \frac{n^{2p-1}}{(p-1)!^2}\Var(g_1(X_1))\,.\]
We now sketch an argument implying that this theorem can be deduced from our results even though the norms of the involved contractions are in general only finite for kernels in $L^4(\mu^p)$. 
In order to cope with this technical issue, for $K\in\N$ we introduce the truncated kernel 
$\psi^{(K)}:=\psi\1_{\{|\psi|\leq K\}}\in L^4(\mu^p)$. It then easily follows from dominated convergence that, with obvious notation, 
one has that $g_l^{(K)}$ converges to $g_l$ as $K\to\infty$ $\mu^l$-a.e. and in $L^2(\mu^l)$, $l=1,\dotsc,p$. In particular, for large enough $K$, $K\geq K_0$ say, it also holds that $(b^{(K)}_1)^2=(p-1)!^2>0$ and $b^{(K)}_l=0$ for $l=2,\dotsc,p$. Writing $U_n^{(K)}=J_p^{(\lfloor nt\rfloor)}(\psi^{(K)})$ for the $U$-process induced by the kernel $\psi^{(K)}$ and 
\[W_n^{(K)}(t):=\frac{ U_n^{(K)}(t)-\E\bigl[U_n^{(K)}(t)\bigr]}{\sigma_n^{(K)}}\]  for its normalized version, 
where $(\sigma_n^{(K)})^2=\Var(U_n^{(K)}(1))$, we have 
\begin{align}\label{decWn}
 W_n(t)&=\frac{\sigma_n^{(K)}}{\sigma_n} W_n^{(K)}(t)+\frac{J_p^{(\lfloor nt\rfloor)}(\rho^{(K)})}{\sigma_n}\,,
\end{align}
where $\rho^{(K)}=\psi\1_{\{|\psi|> K\}}-\E[\psi(X_1,\dotsc,X_p)\1_{\{|\psi(X_1,\dotsc,X_p)|> K\}}]$. Note that 
$\rho^{(K)}$ converges to $0$ $\mu^p$-a.e. and in $L^2(\mu^p)$ as $K\to\infty$.

We will apply 
\cite[Theorem 3.2]{Bil} to the decomposition \eqref{decWn} of $W_n$. Firstly, since 
\[(\sigma_n^{(K)})^2\sim \frac{n^{2p-1}}{(p-1)!^2}\Var(g_1^{(K)}(X_1))\,,\quad K\geq K_0\,,\]
it is straightforward to check that all terms appearing in Theorem \ref{maintheo} (b) for $W_n^{(K)}$ are $o(n^{-1/2})$. Hence, from Theorem \ref{maintheo} and Slutsky's lemma we conclude that  
\[W_n^{(K)}\stackrel{n\to\infty}{\Longrightarrow} \frac{\Var(g_1^{(K)}(X_1))}{\Var(g_1(X_1))}t^{p-1}B
\stackrel{K\to\infty}{\Longrightarrow}t^{p-1}B\,,\]
where $B$ is a standard Brownian Motion. For the remainder term 
\[R_n^{(K)}(t):=\frac{J_p^{(\lfloor nt\rfloor)}(\rho^{(K)})}{\sigma_n}\]
observe that by Markov's inequality, for fixed $K\geq K_0$ and $\epsilon>0$ and with $d$ denoting the Skorohod metric as well as $r^{(K)}(x)=\E|\rho^{(K)}(x,X_2,\dotsc,X_p)|$ we have 
\begin{align*}
\P\Bigl(d\Bigl(W_n,\frac{\sigma_n^{(K)}}{\sigma_n} W_n^{(K)}\Bigr)>\epsilon\Bigr)&\leq 
 \P\bigl(\|R_n^{(K)}\|_\infty>\epsilon\bigr)
 \leq\P\bigl(\sigma_n^{-1} J_p^{(n)}(|\rho_k|)>\epsilon\bigr)
 \\
&\hspace{-2cm}\leq \frac{1}{\epsilon^2\sigma_n^2} \Var\bigl(J_p^{(n)}(|\rho_k|)\bigr)\leq \frac{ n^{2p-1}}{(p-1)!^2 \epsilon^2\sigma_n^2}\Bigl(\Var\bigl(r^{(K)}(X_1)\bigr)+O(n^{p-2})\Bigr)\,.
\end{align*}
Hence,
\begin{align*}
 \limsup_{n\to\infty}\P\Bigl(d\Bigl(W_n,\frac{\sigma_n^{(K)}}{\sigma_n} W_n^{(K)}\Bigr)>\epsilon\Bigr)&\leq 
 \frac{\Var\bigl(r^{(K)}(X_1)\bigr)}{\epsilon^2 \Var\bigl(g_1(X_1)\bigr)}\stackrel{K\to\infty}{\longrightarrow}0\,,
\end{align*}
since $r^{(K)}$ converges to $0$ in $L^2(\mu)$ by the dominated convergence theorem. Thus, \cite[Theorem 3.2]{Bil} implies that for a kernel $\psi\in L^2(\mu^p)$ that does not depend on $n$, the corresponding normalized $U$-process $W_n$ converges weakly to $t^{p-1}B(t)$ if $\Var(g_1(X_1))>0$ (for the necessity of the latter condition see also 
Remark \ref{r:donsker} below). Selecting $p=1$ in the previous discussion, allows one to recover a version of Donsker's theorem \cite[p. 90]{Bil}. If $\psi$ may depend on $n$, the situation is more complicated, since the contraction norms are no longer constant. This applies even to the case of a first Hoeffding projection that is asymptotically dominant.

\end{enumerate}
}
\end{remark}

{}

The following corollary deals with the (simpler) situation of a degenerate kernel.

\begin{corollary}[\bf Degenerate kernels]\label{maincor} Let the assumptions and notation of Section \ref{setting} prevail, define the sequence $W_n$, $n\in\N$, according to \eqref{genW}, and suppose in addition that the kernel $\psi$ is degenerate. Assume that, for all pairs $(l,r)$ of integers such that  $1\leq r\leq p$ and $0\leq l\leq r\wedge (2p-r-1)$,  
\begin{equation}\label{e:symcheck}
\lim_{n\to\infty}n^{\frac{l-r}{2}}\frac{\norm{\psi\star_r^{l}\psi}_{L^2(\mu^{2p-r-l})}}{\norm{\psi}_{L^2(\mu^p)}^2}=0\,
\end{equation}
and that, for all $0\leq l\leq p-1$ and for some $\epsilon>0$, the sequence
\begin{equation}\label{e:symcheck1}
n\mapsto n^{\frac{l-p}{2}+\epsilon}\frac{\norm{\psi\star_p^{l}\psi}_{L^2(\mu^{p-l})}}{\norm{\psi}_{L^2(\mu^p)}^2}
\end{equation}
is bounded.
Then, as $n\to \infty$,
$$
W_n\Longrightarrow \{B(t^p) : t\in [0,1]\}\, ,
$$
where $B := \{ B(t) : t\in[0,1]\}$ denotes a standard Brownian motion. 
\end{corollary}

\begin{remark}\label{r:donsker}{\rm 

%

If $p\geq2$ and the degenerate kernel $\psi$ does not depend on $n$, then condition \eqref{e:symcheck} is not satisfied for $l=r=1$, since 
\begin{align*}
 \lim_{n\to\infty}n^{\frac{l-r}{2}}\frac{\norm{\psi\star_r^{l}\psi}_{L^2(\mu^{2p-r-l})}}{\norm{\psi}_{L^2(\mu^p)}^2}=\frac{\norm{\psi\star_1^{1}\psi}_{L^2(\mu^{2p-2})}}{\norm{\psi}_{L^2(\mu^p)}^2} >0.
\end{align*}
This phenomenon is consistent with the known non-Gaussian fluctuations of degenerate $U$-processes of orders $p\geq2$ having a kernel $\psi$ independent of the sample size --- see \cite{Neu, H79, DM, MT}. 


}
\end{remark}

\subsection{Connection to changepoint analysis}\label{ss:cgpt}

The techniques developed in the present paper can be used to characterize the weak convergence of families of processes that are more general than the ones defined in \eqref{genW} and, in particular, to deal with limit theorems related to {\bf changepoint analysis} (see e.g. \cite{CH88, Ferger94, GH95,  CH_book, Ferger01, Gombay2004, HR_survey, RW19}). In order to illustrate such a connection, we will show how to suitably adapt our results in order to generalise an influential invariance principle for order 2 symmetric $U$-statistics, originally proved by Cs\"org\"o and Horvath in \cite{CH88}. As explained e.g. in \cite{CH_book, Gombay2004, RW19} such an invariance principle has been the starting point of a fruitful line of research, focussing on changepoint testing procedures based on generalisations of Wilcoxon-Mann-Whitney statistics. Further possible extensions of the results of this section, involving in particular antisymmetric kernels \cite{CH88, GH95, CH_book, Ferger01} and rescaled $U$-processes \cite{CH88, Ferger94, RW19}, are outside the scope of the present paper and will be investigated elsewhere.

{} 

As before, we start by fixing a sequence of i.i.d. random variables $X_1, X_2,...$ with values in $(E, \mathcal{E})$, and with common distribution $\mu$. We also consider a sequence of kernels $\{\psi^{(n)} : n\geq 1\}$, such that each mapping $\psi^{(n)} : E^2 \to \R$ is symmetric and square-integrable with respect to $\mu^2$. For applications, $\psi^{(n)}$ is typically chosen in such a way that the quantity $\psi^{(n)} (x,y)$ is small whenever $x,y$ are close, e.g. $\psi^{(n)}(x,y) = \|x-y\|^{\beta}$, $\beta>0$ (assuming $E$ is a normed space), but such a property has no impact on the convergence results discussed below. Here, to simplify the notation we assume from the start that each $\psi^{(n)}$ is centered, that is, $\E[\psi^{(n)}(X_1, X_2)] = 0$ for every $n$. We are interested in the family of $U$-processes $\{Y_n : n\geq 1\}$ given by
$$
Y_n(t) := \sum_{1\leq i\leq \lfloor nt \rfloor < j \leq n} \psi^{(n)} (X_i, X_j), \quad t\in [0,1].
$$
Defining $\psi^{(n)}_u$, $u=1,2$, according to \eqref{defpsis} and writing $\gamma_1(n) := \|\psi_1^{(n)}\|_{L^2(\mu)}$ and $\gamma_2(n) := \|\psi_2^{(n)}\|_{L^2(\mu^2)}$, one deduces immediately that, for $0\leq s\leq t\leq 1$
\begin{eqnarray}\label{e:covy}
&&\Cov(Y_n(s), Y_n(t)) \\ \notag
&& =  \gamma_2(n)^2 \lfloor ns \rfloor ( n - \lfloor nt \rfloor +1) +\gamma_1(n)^2\Big\{  \lfloor ns \rfloor ( n - \lfloor nt \rfloor +1)( n - \lfloor ns \rfloor +1)  \\ \notag
&& \quad\quad\quad\quad\quad\quad+\lfloor ns \rfloor ( n - \lfloor nt \rfloor + 1)( \lfloor nt \rfloor - \lfloor ns \rfloor +1)+  \lfloor ns \rfloor \lfloor nt \rfloor ( n - \lfloor nt \rfloor + 1)   \Big\}.
\end{eqnarray}
We also set $\gamma_n^2 := \Var(Y_n(1/2))$ \footnote{The choice of $t=1/2$ is arbitrary; one could set $\gamma_n^2 := \Var(Y_n( a ))$ any number $a\in (0,1)$ without changing the substance of the subsequent results.}, and $\widetilde{Y}_n := Y_n/\gamma_n$. The next statement corresponds to one of the main findings in \cite{CH88}.

\begin{theorem}[Cs\"org\"o and Horvath \cite{CH88}] \label{t:ch} Under the above assumptions, assume that $ \psi^{(n)} \equiv \psi$ does not depend on $n$, and write $\gamma_1 \equiv \gamma_1(n)$, $n\geq 1$. Then, as $n\to \infty$, one has that $\gamma_n \sim \gamma_1 n^3/4$ and moreover $\widetilde{Y}_n \Longrightarrow 2 A$, where $A = \{A(t) : t\in [0,1]\}$ is defined as
$$
A(t) := (1-t)B(t) +t(B(1) -B(t) ), \quad t\in [0,1], 
$$
with $B$ a standard Brownian motion. 

\end{theorem}

The following statement (containing Theorem \ref{t:ch} as a special case) shows that, by allowing the kernels $\psi^{(n)}$ to explicitly depend on $n$, one can obtain a larger class of functional limit theorems. We recall that a centered continuous Gaussian process $b = \{b(t) : t\in [0,1]\}$ is called a (standard) {\bf Brownian bridge} if $\E[b(t)b(s)] = s\wedge t - st$, for all $s,t\in [0,1]$.
{
\begin{theorem}\label{t:maincp} Let the above assumptions and notation prevail, and assume moreover that: {\rm (I)} the kernels $\psi_1^{(n)}, \psi_2^{(n)}$ verify the asymptotic relations expressed in Conditions {\rm (b')} and {\rm (c')} of Theorem \ref{maintheo2} for $p=2$, and {\rm (II)} as $n\to \infty$,
$$
n^{4-i}\, \frac{\gamma_n(i)^2}{\gamma_n^2} \longrightarrow c^2_i\in [0,\infty), \quad i=1,2.
$$
Then, $\gamma_n^2 \sim (c_1^2 +c_2^2)^{-1} (n^3 \gamma_n(1)^2 + n^2\gamma_n(2)^2)$, and moreover $\widetilde{Y}_n \Longrightarrow c_1\, A + c_2 \, b$, where $b$ is a Brownian bridge independent of $A$.
\end{theorem}
}
Theorem \ref{t:maincp} is proved in Section \ref{ss:dede}. An alternate class of FCLTs displaying Brownian bridges as limits can be found in \cite{GH95}. In the forthcoming Corollary \ref{c:cprg}, we will present a direct application of Theorem \ref{t:maincp} to edge counting in random geometric graphs.

\subsection{Extension to vectors of $U$-processes}\label{mult}
In this subsection we state multivariate extensions of Theorems \ref{maintheo} and \ref{maintheo2}. We first introduce the setup and some notation. Fix a dimension $d\geq 1$ and, for $1\leq i\leq d$, let $p_i$ be a positive integer and suppose that 
$\psi(i)=\psi^{(n)}(i)\in L^2(\mu^{p_i})$ is a symmetric kernel (that may again depend on the sample size $n$). Define the corresponding kernels
$g_k(i)$ and $\psi_s(i)$ (which may also depend on $n$) for all $0\leq k\leq p_i$ and $1\leq s\leq p_i$ in the obvious way. 
W.l.o.g. we may assume that $1\leq p_1\leq p_2\leq\ldots\leq p_d$. 
Moreover, for $1\leq i\leq d$, $t\in[0,1]$ and $n\geq  p_d$, let $U_i(t):=J_{p_i}^{(\floor{nt})}(\psi(i))$, 
\begin{equation*}
W_i(t):=W^{(n)}_i(t):=\frac{U_i(t)-\E[U_i(t)]}{\sqrt{\Var(U_i(1))}} 
\end{equation*}
and $W_i:=(W_i(t))_{t\in[0,1]}$. Then, with obvious notation we have that
\[W:=W^{(n)}:=(W_1,\dotsc,W_d)^T\in D\bigl([0,1];\R^d\bigr)\,.\]
In this section, the vector-valued Gaussian limiting processes $Z=(Z_1,\dotsc,Z_d)^T$ will have zero mean, and a covariance structure that is given by 
\begin{align}\label{limformmult}
\Cov\bigl(Z_i(s),Z_j(t)\bigr)&=\sum_{k=1}^{p_i\wedge p_j}\alpha_k(i,j) (s\wedge t)^k s^{p_i-k} t^{p_j-k}\,,
\end{align} 
where $1\leq i,j\leq d$, $s,t\in[0,1]$ and the $\alpha_k(i,j)$, $1\leq k\leq p_i\wedge p_j$, are real numbers such that $\alpha_k(i,i)\geq 0$. 
Specializing \eqref{limformmult} to the case $i=j$, $i=1,...,d$, one sees immediately that each process $Z_i$ has the form \eqref{e:limsum}, for $p=p_i$ and $\alpha_{k,p}^2 = \alpha_k(i,i)$, thus implying in particular that $Z$ takes a.s. values in $C([0,1];\R^d)$. For $1\leq i\leq d$, we further write 
\[\sigma_n^2(i):=\Var(U_i(1))\,.\]

The next two statements are multidimensional counterparts to Theorem \ref{maintheo} and Theorem \ref{maintheo2}. 
\begin{theorem}[\bf Functional convergence of vectors of general $U$-statistics, I]\label{mult1}
Let the assumptions and notation of this subsection prevail,  and assume that the following three conditions 
 are satisfied: 
\begin{enumerate}[{\normalfont (a)}]
 \item For all $1\leq i\leq j\leq d$ and for all $1\leq k\leq p_i\wedge p_j$, the real limit 
\begin{align*}
b_k(i,j)&:=\lim_{n\to\infty}\frac{n^{p_i+p_j-k}}{\sigma_n(i)\sigma_n(j)}
\Bigl(\langle g_{k}(i),g_k(j)\rangle_{L^2(\mu^{k})}\\
&\hspace{2cm}-\E\bigl[\psi(i)(X_1,\dotsc,X_{p_i})\bigr]\cdot\E\bigl[\psi(j)(X_1,\dotsc,X_{p_j})\bigr]\Bigr)
\end{align*}
 exists; 
 \item for all $1\leq i\leq k\leq d$, all $1\leq v\leq p_i$, all $1\leq u\leq p_k$ and all pairs $(l,r)$ and quadruples $(j,m,a,b)$ of integers such that 
$1\leq r\leq v\wedge u$, $0\leq l\leq r\wedge(u+v-r-1)$ and $(j,m,a,b)\in Q(v,u,r,l)$
 one has that
 \[\lim_{n\to\infty}\frac{n^{p_i+p_k-\frac{u+v+r-l}{2}}}{\sigma_n(i)\sigma_n(k)}\,\norm{g_{j}(i)\star_a^{b}g_{m}(k)}_{L^2(\mu^{j+m-a-b})}=0\,;\]
 \item there exists some $\epsilon>0$, such that, for all $1\leq i\leq d$, all $1\leq r\leq p_i$, all $0\leq l\leq r-1$ and all quadruples $(j,m,a,b)\in Q(r,r,r,l)$, the sequence
 \[\frac{n^{2p_i-r-\frac{r-l}{2}+\epsilon}}{\sigma_n(i)^2} \,\norm{g_{j}(i)\star_a^{b}g_{m}(i)}_{L^2(\mu^{j+m-a-b})}\]
 is bounded.
\end{enumerate}
Then, as $n\to\infty$, one has that $W^{(n)} \Longrightarrow Z$, where $Z = \{Z(t) : t\in [0,1]\}$ is the centered, vector-valued Gaussian process defined by the covariance structure \eqref{limformmult} and where, for $1\leq i,j\leq d$, we have
$$
\alpha_k(i,j) = \frac{b_k(i,j)}{k!(p_i-k)!(p_j-k)!}\,, \quad 1\leq k\leq p_i\wedge p_j. 
$$
%
%

\end{theorem}

\begin{theorem}[\bf Functional convergence of vectors of general $U$-statistics, II]\label{mult2}
The conclusion of Theorem \ref{mult1} continues to hold if the following three conditions {\rm (a')}, {\rm (b')} and {\rm (c')} replace conditions {\rm (a)}, {\rm (b)} and {\rm (c)}: 
 \begin{enumerate}[{\normalfont (a')}]
 \item For all $1\leq i\leq j\leq d$ and for all $1\leq k\leq p_i\wedge p_j$, the real limit 
 \[b_k(i,j):=\lim_{n\to\infty}\frac{n^{p_i+p_j-k}}{\sigma_n(i)\sigma_n(j)}\langle \psi_{k}(i),\psi_k(j)\rangle_{L^2(\mu^{k})}\]
 exists; 
  \item for all $1\leq i\leq k\leq d$, all $1\leq v\leq p_i$, all $1\leq u\leq p_k$	and all pairs $(l,r)$ of integers such that 
$1\leq r\leq v\wedge u$, $0\leq l\leq r\wedge(u+v-r-1)$,
  \[\lim_{n\to\infty}\frac{n^{p_i+p_k-\frac{u+v+r-l}{2}}}{\sigma_n(i)\sigma_n(k)}\,\norm{\psi_{v}(i)\star_r^{l}\psi_{u}(k)}_{L^2(\mu^{v+u-r-l})}=0\,;\]
  \item there exists some $\epsilon>0$, such that, for all $1\leq i\leq d$, all $1\leq r\leq p_i$ and all $0\leq l\leq r-1$, the sequence
  \[\frac{n^{2p_i-r-\frac{r-l}{2}+\epsilon}}{\sigma_n^2(i)}\norm{\psi_{r}(i)\star_r^{l}\psi_{r}(i)}_{L^2(\mu^{r-l})}\]
 is bounded.
 \end{enumerate}
\end{theorem}

\section{Applications}\label{applis}
\subsection{Subgraph counting in random geometric graphs}\label{grg}

{\it Random geometric graphs} are graphs whose vertices are random points scattered on some Euclidean domain, and whose edges are determined by some explicit geometric rule; in view of their wide applicability (for instance, to the modelling of telecommunication networks), these objects represent a very popular alternative to the combinatorial Erd\"os-R\'enyi random graphs. We refer to the texts \cite{Penrose, PR_book} for an introduction to this topic, and for an overview of related applications. Our aim is to use our main findings (Theorem \ref{maintheo} and Theorem \ref{maintheo2}) in order to establish a new collection of FCLTs for arbitrary subgraph counting statistics associated with generic sequences of random graphs. These FCLTs -- whose statements appear in Theorem \ref{graphtheo} below --  hold in full generality and with minimal restrictions with respect to the already existing one-dimensional CLTs; as such, they substantially extend the one-dimensional CLTs proved in \cite[Section 3.5 and Section 3.4]{Penrose}, as well as in \cite{JJ, BhaGo92, DP18}. 

\medskip

We fix a dimension $d\geq1$ as well as a bounded and Lebesgue almost everywhere continuous probability density function $f$ on $\R^d$. Let $\mu(dx):=f(x)dx$ be the corresponding probability measure on $(\R^d,\B(\R^d))$ and suppose that 
$X_1,X_2,\dotsc$ are i.i.d. with distribution $\mu$. Let $X:=\{ X_j : j\in\N\}$. We denote by $\{ t_n : n\in\N\}$ a sequence of radii in $(0,\infty)$ such that $\lim_{n\to\infty}t_n=0$. For each $n\in\N$, we denote by $G(X;t_n)$ the 
\textbf{random geometric graph} obtained as follows. The vertices {of $G(X;t_n)$} are given by the set $V_n:=\{X_1,\dotsc,X_n\}$, which $\Prob$-a.s. has cardinality $n$, 
and two vertices $X_i,X_j$ are connected if and only if $0<\Enorm{X_i-X_j}<t_n$. Furthermore, let $p\geq2$ be a fixed integer and suppose that 
$\Gamma$ is a fixed connected graph on $p$ vertices. For each $n$ we denote by $G_n(\Gamma)$ the number of induced subgraphs of $G(X;t_n)$ which are isomorphic to $\Gamma$. 
Recall that an induced subgraph of $G(X;t_n)$ consists of a non-empty subset $V_n'\subseteq V_n$ with an edge set precisely given by the set of edges of $G(X;t_n)$ whose endpoints are {both} in $V_n'$. 
We will also have to assume that $\Gamma$ is \textbf{feasible} for every 
$n\geq p$. This means that the probability that the restriction of $G(X;t_n)$ to $X_1,\dotsc,X_p$ is isomorphic to $\Gamma$ is strictly positive for $n\geq p$. Note that feasibility depends on the common distribution $\mu$ of the points.
The quantity $G_n(\Gamma)$ is a symmetric $U$-statistic of $X_1,\dotsc,X_n$, since 
\begin{equation*}
 G_n(\Gamma)=\sum_{1\leq i_1<\ldots<i_p\leq n} \psi_{\Gamma,t_n}(X_{i_1},\dotsc,X_{i_p})\,,
\end{equation*}
where $\psi_{\Gamma,t_n}(x_1,\dotsc,x_p)$ equals $1$ if the graph with vertices $x_1,\dotsc,x_p$ and edge set $\{\{x_i,x_j\}\,:\, 0<\Enorm{x_i-x_j}< t_n\}$ is isomorphic to $\Gamma$, and equals $0$ otherwise. 
We denote the corresponding normalized $U$-process by $\{ W_n(t)  : {t\in[0,1]}\} $, i.e. 
\begin{equation*}
 W_n(t)=\frac{U_n(t)-\E[U_n(t)]}{\Var\bigl(G_n(\Gamma)\bigr)^{1/2}}\,,
\end{equation*}
where 
\begin{equation*}
 U_n(t):= G_{\lfloor nt \rfloor} (\Gamma) = \sum_{1\leq i_1<\ldots<i_p\leq \floor{nt}} \psi_{\Gamma,t_n}(X_{i_1},\dotsc,X_{i_p})\,,\quad 0\leq t\leq 1\,.
\end{equation*}

For obtaining asymptotic normality one typically distinguishes between three different asymptotic regimes:
\begin{enumerate}
 \item[\textbf{(R1)}] $nt_n^d\to0$ and $n^p t_n^{d(p-1)}\to\infty$ as $n\to\infty$ (\textit{sparse regime})
 \item[\textbf{(R2)}] $n t_n^d\to\infty$ as $n\to\infty$ (\textit{dense regime})
 \item[\textbf{(R3)}] $nt_n^d\to\rho\in(0,\infty)$ as $n\to\infty$ (\textit{thermodynamic regime})
\end{enumerate}
Note that we could rephrase the regimes  $\textbf{(R1)}$ and $\textbf{(R2)}$ as follows:
\begin{enumerate}
 \item[\textbf{(R1)}] $\displaystyle \Bigl(\frac{1}{n}\Bigr)^{\frac{p}{p-1}}\ll t_n^d\ll\frac{1}{n}$  
\item[\textbf{(R2)}] $\displaystyle \frac{1}{n}\ll t_n^d$,
\end{enumerate}
where, for positive sequence $a_n$ and $b_n$ we write $a_n\ll b_n$, $n\in\N$, if and only if $\lim_{n\to\infty}a_n/b_n=0$.
It turns out that, under regime \textbf{(R2)} one also has to take into account whether the common distribution $\mu$ of the $X_j$ is the uniform distribution $\mathcal{U}(M)$ on some Borel subset $M\subseteq\R^d$, $0<\lambda^d(M) <\infty$ with 
density $f(x)=\lambda^d(M)^{-1}\,{\bf 1}_M(x)$, or not. 
To deal with this peculiarity, we will therefore distinguish between the following four cases:
\begin{enumerate}
 \item[\textbf{(C1)}] $nt_n^d\to0$ and $n^p t_n^{d(p-1)}\to\infty$ as $n\to\infty$. 
 \item[\textbf{(C2)}] $n t_n^d\to\infty$ as $n\to\infty$ and $\mu= \mathcal{U}(M)$ for some Borel subset $M\subseteq\R^d$ s.t. $0<\lambda^d(M) <\infty$.
 \item[\textbf{(C3)}] $n t_n^d\to\infty$ as $n\to\infty$, and $\mu$ is not a uniform distribution.
 \item[\textbf{(C4)}] $nt_n^d\to\rho\in(0,\infty)$ as $n\to\infty$.
 \end{enumerate}

The following important variance estimates will be needed. Except for the case \textbf{(C2)}, which needs a special consideration, these have already been derived in the book \cite{Penrose}. Since it does not make the argument much longer, we provide the whole proof.
 
\begin{proposition}\label{regprop}
 Under all regimes {\normalfont \textbf{(R1)}}, {\normalfont\textbf{(R2)}} and {\normalfont\textbf{(R3)}} it holds that\\
 $\E[G_n(\Gamma)]\sim c n^{p}t_n^{d(p-1)}$ for a constant $c\in(0,\infty)$. Moreover, there exist constants $c_1,c_2,c_3,c_4\in(0,\infty)$ such that, as $n\to\infty$,  
 \begin{enumerate}
  \item[{\normalfont \textbf{(C1)}}] $\Var(G_n(\Gamma))\sim c_1\cdot n^p t_n^{d(p-1)}$,
  \item[{\normalfont \textbf{(C2)}}] $\Var(G_n(\Gamma))\geq c_2\cdot n^{2p-2} t_n^{d(2p-3)}$ for all $n\in\N$,
  \item[{\normalfont \textbf{(C3)}}] $\Var(G_n(\Gamma))\sim c_3\cdot n^{2p-1} t_n^{d(2p-2)}$,
  \item[{\normalfont \textbf{(C4)}}] $\Var(G_n(\Gamma))\sim c_4\cdot n$.
 \end{enumerate}
\end{proposition}

\begin{proof}
For notational convenience, for $ k=0,1,\dotsc,p$ we simply denote by $g_k$ the function 
\[(x_1,\dotsc,x_k)\mapsto \E\bigl[\psi_{\Gamma,t_n}(x_1,\dotsc,x_k,X_1,\dotsc,X_{p-k})\bigr],\] 
corresponding to the kernel $\psi_{\Gamma,t_n}$, i.e. we suppress the dependence on $n$ and 
on the graph $\Gamma$. 
We will make use of formula \eqref{sigma2}. First note that, for $1\leq k\leq p$,
\begin{equation*}
\Var\bigl(g_k(X_1,\dotsc,X_k)\bigr)= \norm{g_k}_{L^2(\mu^k)}^2-\bigl(\E[\psi(X_1,\dotsc,X_p)]\bigr)^2\,.
\end{equation*}
Hence, by \eqref{sigma2}, we have 
\begin{align}\label{rgvar}
\Var(G_n(\Gamma))&=\binom{n}{p}\sum_{k=1}^p\binom{p}{k}\binom{n-p}{p-k}\Bigl( \norm{g_k}_{L^2(\mu^k)}^2-\bigl(\E[\psi(X_1,\dotsc,X_p)]\bigr)^2 \Bigr) \notag\\
&\sim \sum_{k=1}^p  \frac{n^{2p-k}}{k! ((p-k)!)^2 }\Bigl( \norm{g_k}_{L^2(\mu^k)}^2-\bigl(\E[\psi(X_1,\dotsc,X_p)]\bigr)^2 \Bigr)  \,.
\end{align}
For $k=1,\dotsc,p$, we have that (e.g. by dominated convergence)
\begin{align}\label{gksq}
& \norm{g_k}_{L^2(\mu^k)}^2=\int_{({\R^d})^k} g_k(x_1,\dotsc,x_k)^2 d\mu^k(x_1,\dotsc,x_k)\notag\\
& =\int_{({\R^d})^k} \Bigl(\int_{({\R^d})^{p-k}}\psi_{\Gamma,t_n}(x_1,\dotsc,x_k,y_1,\dotsc,y_{p-k})d\mu^{p-k}(y_1,\dotsc,y_{p-k})\Bigr)^2d\mu^k(x_1,\dotsc,x_k)\notag\\
&=\int_{({\R^d})^k}\int_{({\R^d})^{p-k}\times ({\R^d})^{p-k}} \psi_{\Gamma,t_n}(x_1,\dotsc,x_k,y_1,\dotsc,y_{p-k})\psi_{\Gamma,t_n}(x_1,\dotsc,x_k,u_1,\dotsc,u_{p-k})\notag\\
&\hspace{4cm}\cdot\prod_{j=1}^k f(x_j)dx_j\prod_{i=1}^{p-k} f(y_j)dy_j\prod_{l=1}^{p-k} f(u_l)du_l\notag\\
&=\int_{\R^d} f(x_1) \int_{({\R^d})^{k-1}}\int_{({\R^d})^{p-k}\times ({\R^d})^{p-k}} \psi_{\Gamma,1}(0,w_2\dotsc,w_k,z_1,\dotsc,z_{p-k})\\
&\hspace{1cm}\cdot\psi_{\Gamma,1}(0,w_2,\dotsc,w_k,v_1,\dotsc,v_{p-k})\notag\\
&\hspace{1cm}\cdot(t_n^d)^{2p-k-1}\prod_{j=2}^k f(x_1+t_nw_j)dw_j\prod_{i=1}^{p-k} f(x_1+t_nz_i)dz_i\prod_{l=1}^{p-k} f(x_1+t_nu_l)du_l\;dx_1 \notag\\
&\sim d_k (t_n^d)^{2p-k-1}\,,
\end{align}
where, for $1\leq k\leq p$,
\begin{align}\label{e:dk}
 d_k&:=\int_{\R^d} f(x_1)^{2p-k}dx_1 \int_{({\R^d})^{k-1}}\int_{({\R^d})^{p-k}\times ({\R^d})^{p-k}} \psi_{\Gamma,1}(0,w_2\dotsc,w_k,z_1,\dotsc,z_{p-k})\\ \notag
&\hspace{3cm}\cdot \psi_{\Gamma,1}(0,w_2,\dotsc,w_k,v_1,\dotsc,v_{p-k})\prod_{j=2}^k dw_j\prod_{i=1}^{p-k} dz_i\prod_{l=1}^{p-k} du_l\,.
\end{align}
Also, from \cite[Proposition 3.1]{Penrose} we know that 
\begin{align*}
 \E\bigl[G_n(\Gamma)\bigr]&\sim \frac{n^{p}}{p!}t_n^{d(p-1)}\nu\,,
\end{align*}
where 
\begin{align}\label{e:nu}
 \nu= \nu(p, \Gamma) := \int_{\R^d} f(x)^pdx\int_{({\R^d})^{p-1}}\psi_{\Gamma,1}(0,y_2,\dotsc,y_p)dy_2\ldots dy_p>0\,.
\end{align}
This implies that 
\begin{align}\label{epsisq}
\bigl(\E[\psi(X_1,\dotsc,X_p)]\bigr)^2&\sim\nu^2 (t_n^d)^{2p-2}\,.
\end{align}
Since $0<t_n<1$, for $2\leq k\leq p$, this yields that  
\begin{align*}
\Var\bigl(g_k(X_1,\dotsc,X_k)\bigr)&= \norm{g_k}_{L^2(\mu^k)}^2-\bigl(\E[\psi(X_1,\dotsc,X_p)]\bigr)^2\\
&\sim d_k (t_n^d)^{2p-k-1}\,.
\end{align*}
In order to discuss the case $k=1$ we have to  carefully compare $d_1$ to $\nu$.
Note that we have 
\begin{align*}
 d_1&=\int_{\R^d} f(x)^{2p-1}dx\biggl(\int_{({\R^d})^{p-1}}\psi_{\Gamma,1}(0,y_2,\dotsc,y_p)dy_2\ldots dy_p\biggr)^2
\end{align*}
so that $d_1>\nu^2$ if and only if 
\[\int_{\R^d} f(x)^{2p-1}dx>\biggl(\int_{\R^d} f(x)^pdx\biggr)^2\,.\]
By Jensen's inequality we have 
\begin{align*}
 \biggl(\int_{\R^d} f(x)^pdx\biggr)^2= \biggl(\int_{\R^d} f(x)^{p-1} d\mu(x)\biggr)^2&\leq \int_{\R^d} f(x)^{2p-2} d\mu(x)\\
 &=\int_{\R^d} f(x)^{2p-1}dx
\end{align*}
with equality, if and only if, $f(x)^{p-1}$ is $\mu$-a.s. constant, i.e., if and only if $\mu$ is a uniform distribution on some Borel subset $M\subseteq\R^d$ s.t. $0<\lambda^d(M) <\infty$.
Thus, if $\mu$ is not a uniform distribution we obtain that 
\begin{align*}
\Var\bigl(g_1(X_1)\bigr)&\sim(d_1-\nu^2)(t_n^d)^{2p-2}\,,
\end{align*}
whereas, if $\mu$ is a uniform distribution we can only conclude that 
\begin{align*}
\Var\bigl(g_1(X_1)\bigr)&\lesssim (t_n^d)^{2p-2}
\end{align*}
but, in general, we cannot give any lower asymptotic bound on $\Var\bigl(g_1(X_1)\bigr)$.
Note that, for $1\leq k\leq p-1$ we have 
\begin{align}\label{rgvardom}
\frac{n^{2p-k}(t_n^d)^{2p-k-1}}{n^{2p-k-1}(t_n^d)^{2p-k-2}}=n t_n^d\stackrel{n\to\infty}{\longrightarrow}
\begin{cases}
0\,,& \text{under }\textbf{(R1)}\\
\infty\,,& \text{under }\textbf{(R2)}\\
\rho\,,& \text{under }\textbf{(R3)}\,.
\end{cases}
\end{align}
This implies that there are positive constants $c_1,c_3$ and $c_4$ such that
\begin{align}\label{rgvar2}
\Var(G_n(\Gamma))&\sim\begin{cases}
c_1 n^p(t_n^d)^{p-1}\,,&\text{in case }\textbf{(C1)}\\
c_3 n^{2p-1}(t_n^d)^{2p-2}\,,&\text{in case }\textbf{(C3)}\\
c_4 n\,,&\text{in case }\textbf{(C4)}\,,
\end{cases}
\end{align}
whereas, in case \textbf{(C2)} we can conclude (as claimed) that there is a positive constant $c_2$ such that
\begin{align}\label{rgvar3}
\Var(G_n(\Gamma))&\geq \binom{p}{2}\binom{n-p}{p-2}\Bigl( \norm{g_2}_{L^2(\mu^2)}^2-\bigl(\E[\psi(X_1,\dotsc,X_p)]\bigr)^2 \Bigr) 
\sim c_2 n^{2p-2}(t_n^d)^{2p-3}\,.
\end{align}

\end{proof}

The next collection of FCLTs, extending the one-dimensional CLT proved in \cite{Penrose,DP18}, is the main result of the section. Note that, in view of the large number of parameters, in the forthcoming Theorem \ref{graphtheo} we choose to express the distribution of the limit process $Z$ directly in terms of its covariance function \eqref{e:limcov}, rather than using the representation \eqref{e:limsum}. We will also need the following definition: For fixed $\rho\in(0,\infty)$, introduce the positive definite function $\Psi : [0,1]\times [0,1]\to \R : (s,t)\mapsto \Psi(s,t)$ given by
\begin{align}\label{Psi}
 \Psi(s,t):=  \biggl(\sum_{l=1}^p\frac{\rho^{2p-l-1}(d_l-\delta_{l,1}\nu^2)}{l!(p-l)!(p-l)!}\biggr)^{-1}\sum_{k=1}^p \frac{(s\wedge t)^p  (s\vee t)^{p-k}   }{k!(p-k)!(p-k)!}   
\rho^{2p-k-1}(d_k-\delta_{k,1}\nu^2)
\end{align}
where $d_k$, $1\leq k\leq p$, and $\nu$ have been defined in \eqref{e:dk} and \eqref{e:nu}, respectively. 

\medskip

\begin{theorem}[\bf FCLTs for subgraph counting in random geometric graphs]\label{graphtheo}
Let the above assumptions and notation prevail. Then, in the cases
{\normalfont\textbf{(C3)}} and {\normalfont\textbf{(C4)}}, the sequence $W_n = \{ W_n(t)) : {t\in[0,1]}\}$, $n\in\N$ is such that $W_n \Longrightarrow Z$, $n\to \infty$, for some continuous Gaussian process $Z$. 
In the case {\normalfont\textbf{(C1)}}, this convergence holds if, additionally, there is a $\delta>0$ such that $\displaystyle \Bigl(\frac{1}{n}\Bigr)^{\frac{p}{p-1}-\delta}\ll t_n^d\ll\frac{1}{n}$ 
and in the case {\normalfont\textbf{(C2)}}, this convergence holds if, additionally, there is a $\delta>0$ such that $\displaystyle \frac{1}{n}\ll t_n^d\ll n^{-1/2-\delta}$ and if the limits $b_1^2$ and $b_2^2$ do exist, where the parameters $b^2_1$ and $b^2_2$ have been defined in Theorem \ref{maintheo}). The covariance function $\Gamma:[0,1]\times[0,1]\rightarrow\R$ of $Z$ is given by 
\begin{align*}
 \Gamma(s,t)&=\begin{cases}
                                                                                             (s\wedge t)^p\,,& \text{under {\normalfont\textbf{(C1)}}}\\
                                               \frac{(s\wedge t)^p  (s\vee t)^{p-1}}{1+\lambda^{-1}}+ \frac{(s\wedge t)^p  (s\vee t)^{p-2}}{1+\lambda}  \,,& \text{under {\normalfont\textbf{(C2)}}}\\
                                                                                               (s\wedge t)^p  (s\vee t)^{p-1}\,,& \text{under {\normalfont\textbf{(C3)}}}\\
                                                                                              \Psi(s,t)\,,& \text{under {\normalfont\textbf{(C4)}}}\\
                                                                                            \end{cases}
\,,
 \end{align*} 
where $\Psi(s,t)$ is defined in \eqref{Psi} and $\lambda\in[0,+\infty]$ is given in \eqref{lambdac2} below. In particular, in the case {\normalfont\textbf{(C1)}} one has that $W_n \Longrightarrow \{B(t^p) : t\in [0,1]\}$, where $B$ denotes a standard Brownian motion.
\end{theorem}

\begin{remark}{\rm 
 Note that, interestingly, in the case {\normalfont\textbf{(C4)}} the covariance function $\Psi$ of the limiting process depends not only on $\rho$ but also on the difference $d_1-\nu^2$. In particular, the analytic properties of 
 $\Psi$ depend on whether $\mu$ is a uniform distribution or not.}
\end{remark}

\begin{proof}[Proof of Theorem \ref{graphtheo}]
Let us first disregard the case \textbf{(C2)}. Then, by Proposition \ref{regprop}, \eqref{gksq} and \eqref{epsisq} we have 
\begin{align*}
 &\frac{n^{2p-k}}{\sigma_n^2}\bigl(\norm{g_{k}}^2_{L^2(\mu^{k})}-(\E[\psi(X_1,\dotsc,X_p)])^2\bigr)\\
 \sim& \frac{n^{2p-k}\bigl(d_k (t_n^d)^{2p-k-1}-\nu^2 (t_n^d)^{2p-2}\bigr)}{\sum_{l=1}^p\frac{ n^{2p-l}}{l!(p-l)!(p-l)!}\bigl(d_l (t_n^d)^{2p-l-1}-\nu^2 (t_n^d)^{2p-2}\bigr)}.
\end{align*}
By relation \eqref{rgvardom} this implies that 
\begin{align}\label{grbk}
 b_k^2&=\lim_{n\to\infty} \frac{n^{2p-k}}{\sigma_n^2}\bigl(\norm{g_{k}}^2_{L^2(\mu^{k})}-(\E[\psi(X_1,\dotsc,X_p)])^2\bigr)\notag
\\
&=\begin{cases}
 p! \delta_{k,p}\,,& \text{in case \textbf{(C1)}}\\
  (p-1)!(p-1)!\delta_{k,1}\,,&\text{in case \textbf{(C3)}}\\
   \frac{\rho^{2p-k-1}(d_k-\delta_{k,1}\nu^2)}{\sum_{l=1}^p\frac{\rho^{2p-l-1}(d_l-\delta_{l,1}\nu^2)}{l!(p-l)!(p-l)!}}\,,&\text{in case \textbf{(C4)}}.
   \end{cases}
   \end{align}
   In the case \textbf{(C2)}, in general we only know that 
\begin{align*}
\sigma_n^2&\sim \frac{n^{2p-1}}{(p-1)!(p-1)!}\Var(g_1(X_1))+d_2\frac{n^{2p-2}(t_n^d)^{2p-3}}{2(p-2)!(p-2)!}\,,
\end{align*}
where
\[\Var(g_1(X_1))\lesssim(t_n^d)^{2p-2}\,.\]
Thus, using \eqref{rgvardom} we can conclude that 
\begin{equation}\label{bkg3c2}
b_k^2=\lim_{n\to\infty} \frac{n^{2p-k}}{\sigma_n^2}\bigl(\norm{g_{k}}^2_{L^2(\mu^{k})}-(\E[\psi(X_1,\dotsc,X_p)])^2\bigr)=0\,,\quad 3\leq k\leq p\,.
\end{equation}
Hence, if both limits $b_1^2$ and $b_2^2$ do exist, depending on the precise order of $\Var(g_1(X_1))$, the three scenarios 
\begin{enumerate}[(a)]
\item $b_1^2=(p-1)!(p-1)!$ and $b_2^2=0$
\item $b_1^2=0$ and $b_2^2=2(p-2)!(p-2)!$ and 
\item $\displaystyle b_1^2=(p-1)!(p-1)!\bigl(1+\lambda^{-1}\bigl)^{-1}$ and $b_2^2=2(p-2)!(p-2)!\bigl(\lambda +1\bigl)^{-1}$ for some $\lambda\in(0,\infty)$
\end{enumerate}
are possible, where 
\begin{equation}\label{lambdac2}
\lambda=\lim_{n\to\infty} \frac{n^{2p-1}}{(p-1)!(p-1)!}\Var(g_1(X_1)) \cdot\biggl(  d_2\frac{n^{2p-2}(t_n^d)^{2p-3}}{2(p-2)!(p-2)!}\biggr)^{-1}\,.
\end{equation}       
Note that (c) contains (a) and (b) if we allow for $\lambda=0$ and $\lambda=+\infty$ by adopting the conventions that $a/0:=\infty$ and $a/\infty:=0$ for $a\in(0,\infty)$.

In particular, \eqref{grbk} and \eqref{bkg3c2} make sure that condition (a) of Theorem \ref{maintheo} is always satisfied in this example and that the covariance function $\Gamma$ of the potential limiting process given by 
\begin{align*}
 \Gamma(s,t)&=\sum_{k=1}^p \frac{(s\wedge t)^p  (s\vee t)^{p-k}   }{k!(p-k)!(p-k)!}  b_k^2 
 \end{align*} 
coincides with the one in the statement. Now fix integers $1\leq v\leq u\leq p$ and $l,r$ such that 
 $1\leq r\leq v$ and $0\leq l\leq r\wedge (u+v-r-1)$. 
The computations on pages 4196-4197 of \cite{LRP2} show that for all 
\begin{align*}
 (j,m,a,b)\in P&:=\Bigl(\{(j,m,a,b):1\leq b\leq a\leq j\leq m\text{ and } b<m\}\\
 &\hspace{3cm}\cup \{(j,m,a,b): j=m=a\text{ and } b=0\}\Bigr)
 \cap Q(v,u,r,l)
\end{align*}
one has 
  \begin{align}\label{sg1}
   \|g_{j}\star_a^b g_{m}\|^2_{L^2(\mu^{ j+m-a-b})}&=O\bigl(t_n^{d(4p-(j+m+a-b)-1)}\bigr)\notag\\
	&=O\bigl(t_n^{d(4p-(u+v+r-l)-1)}\bigr)\,,
  \end{align}
	where the second relation follows from $0<t_n<1$ and the inequality $j+m+a-b\leq v+u+r-l$
{ (we observe that the authors of \cite{LRP2} actually deal with the rescaled measure $n\cdot\mu$, which is why they obtain another power of $n$ as a prefactor).} 
Now suppose that $(j,m,a,b)\in Q(v,u,r,l)\cap P$. We are going to repeatedly use \eqref{sg1} and Proposition \ref{regprop} for the following estimates:
{In case \textbf{(C1)} we have 
\begin{align*}
\frac{n^{4p+l-r-v-u}}{\sigma_n^4}\,\norm{g_{j}\star_a^{b}g_{m}}_{L^2(\mu^{ j+m-a-b})}^2 
&\lesssim\frac{n^{4p+l-r-v-u} t_n^{d(4p-u-v-r+l-1)}}{n^{2p} t_n^{d(2p-2)}}\\
&=n^{2p-(v+u+r-l)} t_n^{d(2p-(v+u+r-l)+1)}\\
&\leq (nt_n^d)^{2p-(v+u+r-l)} t_n^d\\
&\leq \Bigl(n^pt_n^{d(p-1)}\Bigr)^{-1}\,,
\end{align*}
where we have used that $v+u+r-l\leq 3p$ for the second inequality and, hence, under the assumptions of the theorem it follows that 
\begin{equation*}
 \frac{n^{4p+l-r-v-u}}{\sigma_n^4}\,\norm{g_{j}\star_a^{b}g_{m}}_{L^2(\mu^{ j+m-a-b})}^2 
\lesssim n^{-\delta(p-1)}\,.
\end{equation*}
In case \textbf{(C2)} we obtain 
\begin{align*}
\frac{n^{4p+l-r-v-u}}{\sigma_n^4}\,\norm{g_{j}\star_a^{b}g_{m}}_{L^2(\mu^{ j+m-a-b})}^2 
&\lesssim\frac{n^{4p+l-r-v-u} t_n^{d(4p-u-v-r+l-1)}}{n^{4p-4} t_n^{d(4p-6)}}\\
&=n^{4-(v+u+r-l)} t_n^{d(5-(v+u+r-l))}\\
& \leq nt_n^{2d}\,,
\end{align*}
where we have used that $(v+u+r-l)\geq3$. Hence, under the assumptions of the theorem it follows that 
\begin{equation*}
 \frac{n^{4p+l-r-v-u}}{\sigma_n^4}\,\norm{g_{j}\star_a^{b}g_{m}}_{L^2(\mu^{ j+m-a-b})}^2 \lesssim n^{-2\delta}\,.
\end{equation*}
In case \textbf{(C3)} we similarly obtain 
\begin{align*}
\frac{n^{4p+l-r-v-u}}{\sigma_n^4}\,\norm{g_{j}\star_a^{b}g_{m}}_{L^2(\mu^{ j+m-a-b})}^2 
&\lesssim\frac{n^{4p+l-r-v-u} t_n^{d(4p-u-v-r+l-1)}}{n^{4p-2} t_n^{d(4p-4)}}\\
&=n^{2-(v+u+r-l)} t_n^{d(3-(v+u+r-l))}\\
& \lesssim n^{-1}\,,
\end{align*}
where we have again used that $(v+u+r-l)\geq3$. Finally, in case \textbf{(C4)} we have 
\begin{align*}
\frac{n^{4p+l-r-v-u}}{\sigma_n^4}\,\norm{g_{j}\star_a^{b}g_{m}}_{L^2(\mu^{ j+m-a-b})}^2 
&\lesssim\frac{n^{4p+l-r-v-u} t_n^{d(4p-u-v-r+l-1)}}{n^2}\\
&= n^{-1}\bigl(nt_n^d)^{4p+l-r-v-u-1} \sim n^{-1}\rho^{4p+l-r-v-u-1}\\
&=O(n^{-1})\,.
\end{align*}
}
Eventually, we have to deal with the quadruples $(j,m,a,b)\in Q(v,u,r,l)\setminus P$.
In order to do this, we first remark that we have the asymptotic relations
\begin{align}
 \norm{g_m}_{L^2(\mu^m)}^2&\lesssim (t_n^d)^{2p-m-1}\,,\quad 1\leq m\leq p\quad\text{ and}\label{asrelrg1}\\
 \mu^p(\psi_{\Gamma,t_n}):=\int_{(\R^d)^p} \psi_{\Gamma,t_n} d\mu^p&\lesssim (t_n^d)^{p-1}\,.\label{asrelrg2}
\end{align}
Relation \eqref{asrelrg1} follows from the computation 
\begin{align*}
 &\norm{g_m}_{L^2(\mu^m)}^2=\int_{(\R^d)^m} g_m^2(x_1,\dotsc,x_m)\prod_{j=1}^m f(x_j)dx_j\\
 &=\int_{(\R^d)^m}\prod_{j=1}^m f(x_j)dx_j\int_{(\R^d)^{2p-2m}} \psi_{\Gamma,t_n}(x_1,\dotsc,x_p)\\
 &\hspace{3cm}\cdot\psi_{\Gamma,t_n}(x_1,\dotsc,x_m,z_{m+1},\dotsc,z_p) \prod_{l=m+1}^p f(x_l)f(z_l)dx_ldz_l\\
 &=\int_{(\R^d)^m}\prod_{j=1}^m f(x_j)dx_j\int_{(\R^d)^{2p-2m}} \psi_{\Gamma,1}\bigl(0,t_n^{-1}(x_1-x_2) \dotsc,t_n^{-1}(x_1-x_p)\bigr)\\
 &\hspace{2cm}\cdot\psi_{\Gamma,1}\bigl(0,t_n^{-1}(x_1-x_2),\dotsc,t_n^{-1}(x_1-x_m),t_n^{-1}(x_1-z_{m+1}),\dotsc,t_n^{-1}(x_1-z_{p})\bigr)\\
&\hspace{2cm} \cdot\prod_{l=m+1}^p f(x_l)f(z_l)dx_ldz_l\\
&=(t_n^d)^{2p-m-1}\int_{\R^d} f(x_1)dx_1\int_{(\R^d)^{m-1}}\prod_{j=2}^mf(x_1+t_ny_j)dy_j\\
&\hspace{3cm}\cdot\int_{\R^{2p-2m}}\prod_{l=m+1}^pf(x_1+t_nu_l)f(x_1+t_nv_l)du_ldv_l\\
&\hspace{3cm}\cdot\psi_{\Gamma,1}(0,y_2,\dotsc,y_m,u_{m+1},\dotsc,u_p)\psi_{\Gamma,1}(0,y_2,\dotsc,y_m,v_{m+1},\dotsc,v_p)\\
&\sim (t_n^d)^{2p-m-1}\int_{\R^d} f(x_1)^{2p-m}dx_1\int_{(\R^d)^{m-1}}\prod_{j=2}^mdy_j\int_{\R^{2p-2m}}\prod_{l=m+1}^pdu_ldv_l\\
&\hspace{3cm}\cdot\psi_{\Gamma,1}(0,y_2,\dotsc,y_m,u_{m+1},\dotsc,u_p)\psi_{\Gamma,1}(0,y_2,\dotsc,y_m,v_{m+1},\dotsc,v_p)\\
&\lesssim (t_n^d)^{2p-m-1}\,,
\end{align*}
where made use of the translation invariance and scaling property of the kernel $\psi_{\Gamma,t_n}$ as well as of the a.e.-continuity of $f$. The derivation of \eqref{asrelrg2} is similar but easier and is for this reason omitted.

First, if $a=b=0$ and $j,m\geq1$, then we have 
\begin{align*}
\norm{g_j\star_0^0 g_m}_{L^2(\mu^{j+m})}^2&=\norm{g_j}_{L^2(\mu^j)}^2\cdot\norm{g_m}_{L^2(\mu^m)}^2\\
&\lesssim (t_n^d)^{2p-j-1}(t_n^d)^{2p-m-1}=(t_n^d)^{4p-(j+m)-2}\,.
\end{align*}
Now note that by the definition of the set $Q(v,u,r,l)$ we further have that 
\[j+m=j+m-a-b\leq u+v-r-l\]
which implies that 
\begin{align*}
\norm{g_j\star_0^0 g_m}_{L^2(\mu^{j+m})}&\lesssim (t_n^d)^{4p-(u+v-r-l)-2}=(t_n^d)^{4p-(u+v+r-l)+2r-2}\\
&\leq (t_n^d)^{4p-(u+v+r-l)}\,,
\end{align*}
since $r\geq1$. If $a=b=j=m=0$, then we have 
\begin{align*}
\norm{g_0\star_0^0 g_0}_{L^2(\mu^{0})}^2&=\mu^p(\psi_{\Gamma,t_n})^4\lesssim (t_n^d)^{4p-4}\\
&\leq (t_n^d)^{4p-(u+v+r-l)-1}\,,
\end{align*}
which provides a bound of the same order as \eqref{sg1}. If $a=b=j=0$ and $m\geq 1$, then using $m=j+m-a-b\leq u+v-r-l$ and $r\geq1$,
\begin{align*}
\norm{g_0\star_0^0 g_m}^2_{L^2(\mu^{m})}&=\mu^p(\psi_{\Gamma,t_n})^2\norm{g_m}_{L^2(\mu^m)}^2\\
&\lesssim (t_n^d)^{2p-2}(t_n^d)^{2p-m-1}=(t_n^d)^{4p-m-3}\\
&\leq (t_n^d)^{4p-(u+v-r-l)-3}=(t_n^d)^{4p-(u+v+r-l)-3+2r}\\
&\leq (t_n^d)^{4p-(u+v+r-l)-1}\,,
\end{align*}
which again yields a bound of the same order as \eqref{sg1}.
The only remaining possibility is that $1\leq a=b=m=j\leq p-1$. In this case, we first claim that 
\[2j+1\leq u+v+r-l\,.\] 
Indeed, if $j<u$, then $2j<u+v\leq u+v+r-l$ since $j\leq v$. 
On the other hand, if $j=u$, then $j=v$ and we must also have $r=j$ and $l\leq r-1=j-1$ since $j=a\leq r\leq v=j$ and $0\leq l\leq u+v-r-1=j-1=r-1$. Hence, $u+v+r-l\geq 2j+r-l\geq 2j+1$. 
Thus, we obtain that 
\begin{align*}
\norm{g_j\star_j^j g_j}^2_{L^2(\mu^0)}&=\norm{g_j}_{L^2(\mu^j)}^4 \\
&\lesssim (t_n^d)^{2p-j-1}(t_n^d)^{2p-j-1}=(t_n^d)^{4p-(2j+1)-1}\\
&\leq (t_n^d)^{4p-(u+v+r-l)-1}
\end{align*}
which is the same bound as in \eqref{sg1}. Since all these bounds are at most the same as the bound in \eqref{sg1} we conclude that the above estimates continue to hold for all 
$(j,m,a,b)\in Q(v,u,r,l)\setminus P$.
Since the estimates just proven are independent of the variables $v,u,l$ and $r$ this implies that conditions (b)  and (c) of Theorem \ref{maintheo} are satisfied in the asserted cases.
\end{proof}

As announced, the following statement is a consequence of Theorem \ref{t:maincp} and provides a changepoint counterpart to the previous theorem, in the special case of edge counting.

\begin{corollary}\label{c:cprg} Under the above assumptions and notation, suppose that the sequence $\{t_n\}$ verifies condition {\bf (C1)} for $p=2$, and write $\sigma^2_n := \Var(G_n({\bf edge}))$. Then, if there is a $\delta>0$ such that $\displaystyle \Bigl(\frac{1}{n}\Bigr)^{2-\delta}\ll t_n^d\ll\frac{1}{n}$, the process $T_n := \{ T_n(t) : t\in [0,1]\}$ defined by 
\begin{eqnarray*}
T_n(t) &:=& \frac{1}{\sigma_n}\sum_{1\leq i\leq \lfloor nt \rfloor <j\leq n} \left( {\bf 1}_{\{ 0< \| X_i - X_j\| < t_n\} }- \eta_n\right),\\
&=&  \frac{1}{\sigma_n} \sum_{1\leq i\leq \lfloor nt \rfloor <j\leq n}  {\bf 1}_{\{ 0< \| X_i - X_j\| < t_n\} } - \frac{\eta_n}{\sigma_n}  \lfloor nt\rfloor (n - \lfloor nt\rfloor +1), 
\end{eqnarray*}
where $\eta_n := \P\left[0< \| X_1-X_2\| \leq t_n\right]$, is such that $T_n\Longrightarrow \sqrt{2}\, b$, where $b$ is a standard Brownian bridge.
\end{corollary}

The proof of Corollary \ref{c:cprg} (whose details are left to the reader) follows from the fact that, under the regime {\bf (C1)} and in the notation of Theorem \ref{t:maincp}, one has that $\gamma_n^2 \sim \sigma_n^2/2$, and also $c_1 = 0$ and $c_2 = 2$ --- in such a way that the limiting process $\sqrt{2}\, b$ exclusively emerges from the fluctuations of the second (degenerate) Hoeffding projections of the $U$-statistics $T_n(t)$, $t\in [0,1]$. Writing $k:= \lfloor nt \rfloor $ for a fixed $t$, the sum $$ S(n,t) := \sum_{1\leq i\leq \lfloor nt \rfloor <j\leq n}  {\bf 1}_{\{ 0< \| X_i - X_j\| \leq t_n\} }$$ counts the number of edges in $G(X; t_n)$ such that one endpoint  belongs to the set $\{X_1,..., X_k\}$ and the other belongs to $\{X_{k+1}, ..., X_n\}$; a small value of $S(n,t)$ implies that most distances between the elements of the two blocks of variables are larger than $t_n$. The random variable $S(n,t)$ is a special case of the family of $U$-statistics used for {\bf graph-based changepoint detection} defined e.g. in \cite[formula (3.1)]{CC19} (for $k=0$ and $G = G(X; t_n)$); we refer the reader to such a reference, as well as to the seminal contribution \cite{CZ15}, for an overview of changepoint analysis techniques based on the use of random geometric graphs. For testing procedures (see e.g. \cite{Ferger01, CH_book}), one is typically interested in understanding the asymptotic distribution of such quantities as
$$
M_n := \max_{t\in [0,1]}( - T_n(t)) , \quad \mbox{or} \quad A_n := \argmax_{t\in [0,1]} \, (-T_n(t)),
$$
where $\argmax_{t\in [0,1]} g(t)$ stands conventionally for the smallest maximizer of a function $g$ admitting a maximum\footnote{The domain of the $\argmax$ operator can be extended to $D[0,1]$ --- see e.g. \cite[p. 491]{Ferger01}.}. Corollary \ref{c:cprg} implies that $M_n$ and $A_n$ converge in distribution to $m:= \sqrt{2} \max b_t$  and $a:= \argmax b_t$, respectively. It is a well-known fact (see e.g. \cite{Ferger95, CH_book}) that $m/\sqrt{2}$ is distributed according to the Kolmogorov-Smirnov law, whereas $a$ is uniformly distributed on $[0,1]$. 

\begin{remark} {\rm The fact that $A_n$ converges in distribution to $\argmax b_t$ is justified by the observation that, according e.g. to \cite[Example 1.2]{Ferger99}, $b$ is a continuous Gaussian process having a.s. a unique maximizer in $[0,1]$, in such a way that the desired conclusion can be deduced by an application of the Continuous Mapping Theorem \cite[Theorem 2.7]{Bil} analogous to \cite[proof of Theorem 2.1]{Ferger01}. }

\end{remark}

More general limit theorems (involving in particular an independent process $A$, as in Theorem \ref{t:maincp}) could be obtained by considering an adequately renormalized version of $T_n$ under the remaining regimes {\bf (C2)}---{\bf (C4)}.

\subsection{$U$-statistics of order 2 with a dominant diagonal component}\label{ss:quadest}

\subsubsection{General statements}\label{sss:vv1}

In the paper \cite{van_der_vaart}, a remarkable collection of one-dimensional CLTs was proved for sequences of  $U$-statistics of order 2 displaying size-dependent kernels, as well as dominant non-linear Hoeffding components. The Gaussian fluctuations of the $U$-statistics considered in \cite{van_der_vaart} emerge asymptotically from the fact that the corresponding kernels tend to concentrate around a diagonal, a phenomenon leading to Gaussianity if one assumes some additional Lyapounov-type condition. The scope of the applications developed in \cite{van_der_vaart} covers e.g. the estimation of quadratic functionals of densities and regression functions, as well as the estimation of mean responses with missing data  (see Section \ref{sss:vv2} below, as well as \cite[Section 3]{van_der_vaart}, and \cite{BR88, L1, L2, LM00}). 

{}

Our aim in this section is to use our Theorem \ref{maintheo} in order to prove a functional version of the forthcoming Theorem \ref{t:vdv}, corresponding to a special (but fundamental) case of \cite[Theorem 2.1]{van_der_vaart}. Two explicit examples related to kernels based on wavelets and on Fourier bases, respectively, are studied in full detail in Section \ref{sss:vv2}.

{}

In order to state the announced results, we adopt a notation similar to \cite{van_der_vaart} and consider a sequence of i.i.d. random variables $\{X_i : i\geq 1\}$ with values in the measurable space $(E , \mathcal{E})$ and with common distribution $\mu$. We also consider a sequence $\{K_n : n \geq 1\} \subset L^2(\mu^2)$ of symmetric kernels 
$$
K_n : E\times E \to \mathbb{R} : (x,y)\mapsto K_n(x,y).
$$
For every $n$, we define the constant $\sigma_n$ and the processes $U_n = \{U_n(t) : t\in [0,1]\}$ and $W_n = \{W_n(t) : t\in [0,1]\}$ according to \eqref{e:ut0}---\eqref{sigma2}, in the special case $p=2$ and $\psi = K_n$, that is:
\begin{equation}\label{e:utex}
W_n(t) = \frac{\sum_{1\leq 1< i<j\leq \lfloor nt \rfloor} K_n(X_i, X_j) -  \lfloor nt \rfloor\left( \lfloor nt \rfloor-1\right) \E\,  K_n(X_1, X_2)} {\sigma_n}, \,\,t\in [0,1].
\end{equation}
We write $k_n :=\mathbb{E} \, K_n^2(X_1,X_2) = \|K_n\|^2_{L^2(\mu^2)}$, and assume that
\begin{equation}\label{e:ratio}
\frac{k_n}{n}\xrightarrow{n\to\infty}\infty,
\end{equation}
and moreover
\begin{equation}\label{assumption1}
\sup_n\sup_{\|f\|_{L^2(\mu)}=1}\int_{ E }\left(\int_{E} f(v)K_n(x,v)\mu(dv)\right)^2\mu(dx) = \sup_n \| K_n\|_{op}^2 <\infty,
\end{equation}
where $\|\bullet \|_{op}$ denotes the operator norm of the Hilbert-Schmidt operator\\ $f\mapsto \int_E K_n(\cdot ,y)f(y)\mu(dy)$, and
\begin{equation}\label{assumption2}
\|K_n\|_{\infty}\lesssim k_n.
\end{equation}

Assumptions \eqref{e:ratio} and \eqref{assumption1} are easily checkable conditions implying that the linear part of the Hoeffding decomposition of $W_n(1)$ vanishes in $L^2(\P)$ as $n\to \infty$. Assumption \eqref{assumption2} can be relaxed (see e.g. formula (10) and Lemma 2.1 in \cite{van_der_vaart}), but we decided to avoid such a level of generality --- which is not needed for the examples developed below --- in order to keep our paper within bounds.

\begin{theorem}[{\bf Theorem 2.1 and Lemma 2.1 in \cite{van_der_vaart}}] \label{t:vdv} 
Assume that there exists a sequence of finite measurable partitions $\mathcal{P}_n := \{\mathcal{X}_{n,m} : m=1,...,M_n\}$, $n\geq 1 $, of the set $E$, such that
\begin{align}
&\frac{1}{k_n}\sum_m\int_{\mathcal{X}_{n,m}}\int_{\mathcal{X}_{n,m}} K_n^2 \,  d\mu d\mu \xrightarrow{n\to\infty} 1, \label{assumpion3}\\
&\frac{1}{k_n}\max_m\int_{\mathcal{X}_{n,m}}\int_{\mathcal{X}_{n,m}} K_n^2 \, d\mu d\mu \xrightarrow{n\to\infty} 0,\label{assumpion4}\\
&\max_m \mu(\mathcal{X}_{n,m})\frac{k_n}{n}\xrightarrow{n\to\infty} 0,\label{assumption5}\\
&\liminf_{n\to\infty} n\min_m \mu(\mathcal{X}_{n,m})>0.\label{assumption6}
\end{align}
Then, as $n\to \infty$, one has that $\sigma^2_n \sim \frac{ n^2 k_n}{2}$ and 
$$
W_n(1) \stackrel{\rm Law}{\longrightarrow} Z,
$$
where $Z$ is a standard Gaussian random variable.
\end{theorem}

The main abstract result of the present section is the following functional version of the previous statement. 

\begin{theorem}\label{t:fvdv} Let the setting and assumptions of Theorem \ref{t:vdv} prevail. If, in addition, one has that
\begin{align}
&\sup_n n^{1/2+\epsilon_1}\max_m \mu(\mathcal{X}_{n,m})<\infty,\text{ for some }\epsilon_1>0, \label{assumption7}\\
&\sup_n\max_m\mu(\mathcal{X}_{n,m})\frac{k_n}{n^{1-\epsilon_2}}<\infty\textbf{ or }\liminf_{n\to\infty} n^{1-\epsilon_2}\min_m \mu(\mathcal{X}_{n,m})>0,\text{ for some }\epsilon_2>0,\label{assumption7.1}\\
&\sup_n\frac{n^{1+\alpha_1}}{k_n}<\infty\text{, for some }\alpha_1>0 \label{assumption8},\\
&\sup_n \frac{n^{\alpha_2}}{k_n}\int_E\int_E\left| K_n(x,y){\bf 1}_{\left\{ (x,y)\in\bigcap \left(\mathcal{X}_{n,m}\times\mathcal{X}_{n,m}\right)^c\right\}}\right|^2\mu(dx)\mu(dy)<\infty,\text{ for some } \alpha_2>0\label{assumption9},
\end{align}
then
$$
W_n\Longrightarrow \left\{B(t^2) : t\in [0,1] \,  \right \}, \quad n\to \infty,
$$
where $B$ is a standard Brownian motion issued from zero. 
\end{theorem}

As we will point out in Remark \ref{r:roles} below each of the four assumptions \eqref{assumption7}, \eqref{assumption7.1}, \eqref{assumption8} and \eqref{assumption9}  plays a substantially different role in the proof.

\begin{proof}[Proof of Theorem \ref{t:fvdv}]  Using the notation introduced in this section, for every $n$ we define $$Q_n := \bigcup_{m=1}^{M_n} \mathcal{X}_{m,n} \times \mathcal{X}_{m,n} \subset E^2,$$ and write $Q_n^c = \bigcap_{m=1}^{M_n}( \mathcal{X}_{m,n} \times \mathcal{X}_{m,n} )^c $ in order to denote the complement of $Q_n$ in $E^2$. We set $A_n := K_n {\bf 1}_{Q_n}$ and $B_n := K_n -A_n  = K_n {\bf 1}_{Q_n^c}$, and define
$ T_n := \{T_n(t) : t\in [0,1]\}$ and $R_n := \{R_n(t): t\in [0,1]\}$ as
\begin{eqnarray*}
T_n(t) &:=& \frac{\sum_{1\leq 1< i<j\leq \lfloor nt \rfloor} A_n(X_i, X_j) -  \lfloor nt \rfloor\left( \lfloor nt \rfloor-1\right) \E\,  A_n(X_1, X_2)} {\sigma_n}, \\ 
R_n(t) &:=&  \frac{\sum_{1\leq 1< i<j\leq \lfloor nt \rfloor} B_n(X_i, X_j) -  \lfloor nt \rfloor\left( \lfloor nt \rfloor-1\right) \E\,  B_n(X_1, X_2)} {\sigma_n},
\end{eqnarray*}
in such a way that $W_n = T_n+R_n$. Our first remark is that, for every $f\in L^2(\mu)$ with unit norm, one has that
\begin{eqnarray*}
&& \int_{ E }\left(\int_{E} f(v)A_n(x,v)\mu(dv)\right)^2\mu(dx) = \sum_{m=1}^{M_n} \int_{ \mathcal{X}_{m,n} }\left(\int_{ \mathcal{X}_{m,n}} f(v)K_n(x,v)\mu(dv)\right)^2\mu(dx) \\
&& \leq \sum_{m=1}^{M_n} \int_{ E }\left(\int_{ \mathcal{X}_{m,n}} f(v)K_n(x,v)\mu(dv)\right)^2\mu(dx)\leq \|K_n \|_{op} \sum_{m=1}^{M_n} \int_{\mathcal{X}_{m,n}} \hspace{-2mm}f(x)^2 \mu(dx) =  \|K_n \|_{op},
\end{eqnarray*}
from which we infer that
\begin{equation}\label{e:op}
\sup_n  \|A_n\|_{op},\,\, \sup_n \|B_n\|_{op} <\infty,
\end{equation}
where we have applied \eqref{assumption1} and the triangle inequality in order to deal with $B_n$. Using the identity \eqref{sigma2} (in the case $\psi = B_n$) together with \eqref{e:op}, with the relations $\sigma_n^2 \sim n^2k_n/2$ and $n/k_n\to 0$ and with \eqref{assumpion3}, shows immediately that, as $n\to \infty$, $\E[R_n(t)^2] = o(\sigma^2_n)$ for every $t\in [0,1]$. This in turn implies that $\sigma_n^2\sim \Var (T_n(1))$. We will now study $R_n$ and $T_n$ separately, and prove that
\begin{itemize}
\item[\rm (i)] the sequence $\{R_n : n\geq 1\}$ is tight in $D[0,1]$, so that $R_n \Longrightarrow 0$ (zero function of $D[0,1]$);

\item[\rm (ii)] the sequence $\{T_n : n\geq 1\}$ verifies the assumptions of Theorem \ref{maintheo} in the case $p=2$, with $\alpha_{1,2} = 0$ and $\alpha_{2,2} = 1$, and therefore $T_n$ weakly converges to $B(t^2)$ in $D[0,1]$.  

\end{itemize}
\noindent[\underline{Proof of (i)}] We first define the functions $g_0, g_1, g_2$ according to \eqref{gk}, in the case $p=2$ and $\psi = B_n$, so that the Hoeffding decomposition of the $U$-statistic $R_n(t)$, $t\in [0,1]$, is
\begin{align*}
R_n(t) &= \frac{\nt - 1}{\sigma_n} \sum_{j=1}^{\nt}[g_1(X_j) - g_0] +\frac{1}{\sigma_n} \sum_{1\leq i< j\leq \nt}[g_2(X_i, X_j) - g_1(X_i) - g_1(X_j)+g_0] \\
&:= R'_n(t) + R''_n(t).
\end{align*}
Now fix $0\leq s< t\leq 1$. Then, writing $g_1 - g_0 :=  \psi_1$, as before, and using as always the symbol $c$ in order to denote an absolute finite constant whose exact value might change from line to line,
\begin{align*}
& \E\left |R'_n(t) -R'_n(s)\right|^2 \\
 \leq& \frac{c}{\sigma_n^2} \left\{ \E \left| \ns \sum_{\ns < j\leq \nt} \psi_1(X_j)\right|^2 + \E \left| (\nt - \ns) \sum_{1\leq j\leq \nt} \psi_1(X_j) \right|^2\right\} \\
 := &\frac{c}{\sigma_n^2} (Y_1+Y_2).
\end{align*}
We can assume without loss of generality that $\alpha_1\in (0,1]$; we have that
\begin{align}\label{tightness1}
\frac{Y_1}{\sigma_n^2} \leq \frac{c}{k_n} \E\left[ \sum_{\ns < j\leq \nt} \psi_1(X_j)^2 \right]
\leq & \frac{c}{k_n} \big( \nt - \ns \big)^{1+{\alpha_1}} \E |\psi_1(X_1)| ^2\notag\\ \leq  &c \left( \frac{ \nt - \ns}{n} \right)^{1+{\alpha_1}},
\end{align}
where we have used \eqref{assumption1} and \eqref{assumption8} to deduce the last inequality. Analogously, one shows that
\begin{eqnarray}\label{tightness2}
\frac{Y_2}{\sigma_n^2} &\leq& c \left( \frac{ \nt - \ns}{n} \right)^{2} \frac{n}{k_n}  \E |\psi_1(X_j)|^2 \leq c \left( \frac{ \nt - \ns}{n} \right)^{1+{\alpha_1}},
\end{eqnarray}
where the last inequality follows again from \eqref{assumption1} and from \eqref{e:ratio}, as well as from the fact that $\frac{ \nt - \ns}{n}\in [0,1]$. 

{}

\noindent 
In order to deal with $R''$, we adopt as before the notation $\psi_2(X_i, X_j) := g_2(X_i, X_j) - g_1(X_i) - g_1(X_j)+g_0$, and observe that, for every $a\geq 1$, $\E[| \psi_2(X_i, X_j)|^a ] \leq c \int_E\int_E | B_n |^a\,  d\mu^2$, for some absolute constant $c$ depending solely on $a$. For every $n$ and every $0\leq s < t \leq 1$, we define the set of integers
\begin{eqnarray*}
H_n(s,t) &:=& \{ (i,j) : 0\leq i<j\leq \nt\} \backslash   \{ (i,j) : 0\leq i<j\leq \ns\}.
\end{eqnarray*}
Clearly, $|H_n(s,t)| ={\lfloor nt\rfloor \choose 2}-{\lfloor ns\rfloor \choose 2}=\frac{1}{2}(\lfloor nt\rfloor +\lfloor ns\rfloor -1)(\lfloor nt\rfloor -\lfloor ns\rfloor)$. For fixed $0\leq s < t \leq 1$, one has that
$$
\E\left |R''_n(t) -R''_n(s)\right|^2 = \frac{1}{\sigma_n^2}\,  \E\left| \sum_{(i,j)\in H_n(s,t) }\psi_2(X_i, X_j)\right|^2 .
$$
In order to bound such a quantity, we use orthogonality of the summands in the above sum, fix an $\alpha_2>0 $ such that condition \eqref{assumption9} is satisfied, and note that
\begin{align}
\E\left |R''_n(t) -R''_n(s)\right|^2&=\frac{1}{\sigma_n^2}\sum_{(i_1,i_2)\in H_n(s,t)}\E\left[\psi_2(X_{i_1},X_{i_2})^2\right]\nonumber\\
&\leq\frac{c}{\sigma_n^2}|H_n(s,t)| \int_E\int_E B_n^2\,  d\mu^2\nonumber\\
&\leq \frac{c\int_E\int_E B_n^2\,  d\mu^2}{k_n}\cdot\frac{\lfloor nt\rfloor -\lfloor ns\rfloor}{n}\nonumber\\
&\leq \frac{cn^{\alpha_2}\int_E\int_E B_n^2\,  d\mu^2}{k_n}\cdot\left(\frac{\lfloor nt\rfloor -\lfloor ns\rfloor}{n}\right)^{1+\alpha_2}.\label{tightness3}
\end{align}
We have therefore shown (in \eqref{tightness1}, \eqref{tightness2} and \eqref{tightness3}) that $\{R_n\}$ satisfies the tightness criterion of Lemma \ref{lemma_tightness}, for $\alpha = \min\left({\alpha_1},\alpha_2\right)$ and $\beta=2$, and the proof of Point (i) is concluded.

{}

\noindent[\underline{Proof of (ii)}] In this part of the proof, we denote by $g_0,g_1, g_2$ the functions obtained from \eqref{gk} by selecting $p=2$ and $\psi = A_n$. Note that each of the three kernels $g_i$ implicitly depends on $n$ and that, by virtue of \eqref{e:op}, one has 
\begin{equation}\label{e:pivot} \sup_n \left( \E [ g_1(X_1)^2]+ | g_0 | \right) <\infty.
\end{equation}
Since $g_2 = A_n$, $\sigma^2_n \sim k_n n^2/2$ and \eqref{assumpion3} is in order, we see immediately that the constants $b_1 $ and $b_2$ appearing at Point (a) of Theorem \ref{maintheo} are such that $b_1 = 0$ and $b_2^2= 2$, yielding $\alpha_{1,2} = 0$ and $\alpha_{2,2} = 1$. In order to conclude our proof, we have now to check that the quantities appearing at Points 1.-6. of Remark \ref{r:whattocheck} all converge to zero as $n\to \infty$  and that the quantities in points i)-iv) of the same remark are bounded for some $\epsilon>0$. This is immediately done for the quantities at Points 1., i) and ii), by virtue of \eqref{e:pivot}. To deal with the quantity at Point iii), we note that, for some $\epsilon>0$, 
\begin{align*}
&\frac{n^{5/2+\epsilon}}{\sigma_n^2}\|g_1\|_{L^4(\mu)}^2\\
&\sim\frac{n^{1/2+\epsilon}}{k_n}\left[\int_E\left(\int_E K_n(x,y)\, \mathbf{1}\left[(x,y)\in\bigcup_m \mathcal{X}_{n,m}\times\mathcal{X}_{n,m}\right]\mu(dy)\right)^4\mu(dx)\right]^{1/2}\\
&=\frac{n^{1/2+\epsilon}}{k_n}\left[\sum_m\int_{\mathcal{X}_{n,m}}\left(\int_{\mathcal{X}_{n,m}}K_n(x,y)\mu(dy)\right)^4\mu(dx)\right]^{1/2}\\
&\leq \frac{n^{1/2+\epsilon}}{k_n}\left[\sum_m\int_{\mathcal{X}_{n,m}}\left(\int_{\mathcal{X}_{n,m}}K_n(x,y)\mu(dy)\right)^2\left(\int_{\mathcal{X}_{n,m}}k_n\mu(dy)\right)^2\mu(dx)\right]^{1/2},\,\text{by (\ref{assumpion3})}\\
&\leq n^{1/2+\epsilon}\max_m\mu(\mathcal{X}_{n,m})\left[\sum_m\int_{\mathcal{X}_{n,m}}\left(\int_{\mathcal{X}_{n,m}}K_n(x,y)\mu(dy)\right)^2\mu(dx)\right]^{1/2}\\
&\lesssim n^{1/2+\epsilon}\max_m\mu(\mathcal{X}_{n,m}),\quad\text{by (\ref{e:op})},
\end{align*}
which is bounded by \eqref{assumption7}. On the other hand,
\begin{align*}
&\frac{n^{2}}{\sigma_n^2}\|g_1\star_1^0 g_2\|_{L^2}=\frac{n^2}{\sigma_n^2}\left(\int_E \int_E \left(g_1(x)g_2(x,y)\right)^2\mu(dx)\mu(dy)\right)^{1/2}\\
&\sim\frac{1}{k_n}\Bigg\{\int_E\int_E\left( \int_E K_n(x,z){\bf 1}\left[(x,z)\in\bigcup_m \mathcal{X}_{n,m}\times\mathcal{X}_{n,m}\right]\mu(dz) \right.\\
&\hspace{1cm}\cdot\left.K_n(x,y){\bf 1}\left[(x,y)\in\bigcup_m \mathcal{X}_{n,m}\times\mathcal{X}_{n,m}\right]\right)^2\mu(dx)\mu(dy)\Bigg\}^{1/2}\\
&=\frac{1}{k_n}\left[\sum_m \int_{\mathcal{X}_{n,m}}\int_{\mathcal{X}_{n,m}}\left(\int_{\mathcal{X}_{n,m}}K_n(x,z)\mu(dz)\right)^2K_n^2(x,y)\mu(dx)\mu(dy)\right]^{1/2}\\
&\stackrel{(\ref{assumption2})}\leq \frac{1}{k_n}\left[\sum_m \int_{\mathcal{X}_{n,m}}\int_{\mathcal{X}_{n,m}}\left(\int_{\mathcal{X}_{n,m}}K_n(x,z)\mu(dz)\right)^2k_n^2\,\mu(dx)\mu(dy)\right]^{1/2}\\
&\leq \sqrt{\max_m\mu(\mathcal{X}_{n,m})}\left[\sum_m\int_{\mathcal{X}_{n,m}}\left(\int_{\mathcal{X}_{n,m}}K_n(x,z)\mu(dz)\right)^2\,\mu(dx)\right]^{1/2}\xrightarrow{n\to\infty}0,
\end{align*}
by (\ref{assumption7}) and (\ref{e:op}), showing that the quantity at Point 2. vanishes. We can deal at once with the quantities at Point 3. and 5. by means of the following considerations. For a fixed $n$, denote by $\{\lambda_j : j\geq 1\}$ and $\{e_j : j\geq 1\}$, respectively, the sequence of eigenvalues (taken in decreasing order) and eigenfunctions of the Hilbert-Schmidt operator on $L^2(\mu)$ given by $f\mapsto \int_E A_n (\cdot, y) f(y)\mu(dy)$. Then, such eigenfunctions form an orthonormal system in $L^2(\mu)$, and one has that $A_n = g_2 = \sum_i \lambda_i e_i\otimes e_i$, with convergence in $L^2(\mu^2)$. Such a relation yields that $\| A_n\|_{op} = \lambda_1$, $g_2\star_1^1 g_2 = \sum_i \lambda^2_i e_i\otimes e_i$, $g_1= \sum_i \lambda_i \mu_ i e_i$ (where $\mu_i := \int_E e_i d\mu$) and $g_1\star_1^1g_2 = \sum_i\lambda_i ^2\mu_i e_i$. Since $|\mu_i | \leq 1$ (by Cauchy-Schwarz), we infer that
$$
\|g_1\star_1^1g_2\|_{L^2(\mu)},\, \|g_2\star_1^1g_2\|_{L^2(\mu^2)}  \leq \sqrt{\sum_i \lambda_i^4} \leq \|A_n\|_{op}\|A_n\|_{L^2(\mu^2)},
$$
and the desired convergence to zero follows from \eqref{e:ratio}, \eqref{assumpion3} and \eqref{e:op}. The vanishing of the quantity at Point 4. follows from
\begin{align*}
&\frac{n^{3/2}}{\sigma_n^2}\|g_2\star_1^0 g_2\|_{L^2}\\
\sim&\frac{1}{n^{1/2}k_n}\left(\sum_m\int_{\mathcal{X}_{n,m}}\int_{\mathcal{X}_{n,m}}\int_{\mathcal{X}_{n,m}} \hspace{-2mm}K^2_n(x,y)K^2_n(x,z)\mu(dx)\mu(dy)\mu(dz)\right)^{1/2} \\
\leq& \frac{1}{n^{1/2}k_n}\Bigg\{\sum_m\int_{\mathcal{X}_{n,m}}\int_{\mathcal{X}_{n,m}}\left(\int_{\mathcal{X}_{n,m}} K^4_n(x,y)\mu(dx)\right)^{1/2}\\
&\hspace{4cm}\cdot\left(\int_{\mathcal{X}_{n,m}}K^4_n(x,z)\mu(dx)\right)^{1/2}\mu(dy)\mu(dz)\Bigg\}^{1/2}\\
=&\frac{1}{n^{1/2}k_n}\left\lbrace\sum_m \left(\int_{\mathcal{X}_{n,m}}\left(\int_{\mathcal{X}_{n,m}}K^4_n(x,z)\mu(dx)\right)^{1/2}\mu(dz)\right)^2\right\rbrace^{1/2}\\
\leq& \left\lbrace\frac{1}{k_n^2}\max_m\frac{\mu(\mathcal{X}_{m,n})}{n}\sum_m \int_{\mathcal{X}_{n,m}}\int_{\mathcal{X}_{n,m}}K^4_n(x,z)\mu(dx)\mu(dz)\right\rbrace^{1/2}\xrightarrow{n\to\infty}0,
\end{align*}
where we have applied \eqref{assumption2} and \eqref{assumption5}. One has also that, for some $\epsilon>0$, 
\begin{align*}
&\frac{n^{1+\epsilon}}{\sigma_n^2}\|g_2\|^2_{L^4(\mu)}\sim\frac{1}{n^{1-\epsilon}k_n}\left\lbrace\sum_m\int_{\mathcal{X}_{n,m}}\int_{\mathcal{X}_{n,m}}K_n^4(x,y)\mu(dx)\mu(dy)\right\rbrace^{1/2}\\\
&=n^{\epsilon}\left\lbrace \frac{1}{k_n^2}\max_m\frac{\mu(\mathcal{X}_{n,m})}{n}\sum_m\int_{\mathcal{X}_{n,m}}\int_{\mathcal{X}_{n,m}}K_n^4(x,y)\mu(dx)\mu(dy)\right\rbrace^{1/2}\hspace{-0.3cm} \frac{1}{n^{1/2}\max_m\mu(\mathcal{X}_{n,m})^{1/2}}
\end{align*}
is bounded, by \eqref{assumption2}, \eqref{assumpion3}, \eqref{assumption5}, \eqref{assumption6} and \eqref{assumption7.1} -- this yields the boundedness of the sequence at Point iv). Finally, the convergence to zero of the quantity at Point 6. is a direct consequence of  \eqref{e:ratio} and \eqref{assumpion3}.

\end{proof}

\begin{remark}\label{r:roles}{\rm  By inspection of the previous proof, one sees that Assumptions \eqref{assumption8} and \eqref{assumption9} imply that the off-diagonal part of the $U$-process $W_n$ is tight in the space $D[0,1]$. On the other hand, Assumptions \eqref{assumption7} and \eqref{assumption7.1} are needed in order to ensure that the (dominating) diagonal component of $W_n$ meets the requirements of Theorem \ref{maintheo}. Note that Assumption \eqref{assumption7} is such that (a) it does not appear in \cite{van_der_vaart}, and (b) it would be needed if one wanted to prove a one-dimensional CLT for $W_n(1)$ by using the techniques developed in \cite{DP18}. This slight discrepancy between the assumptions of \cite{DP18} and \cite{van_der_vaart} is explained by the fact that the sufficient conditions discovered in \cite{DP18} would imply not only a CLT for $W_n(1)$, but also that $\E[W_n(1)^4]\to 3$, and consequently need to be stronger. 
}
\end{remark}

\subsubsection{Two examples}\label{sss:vv2}

As an application of Theorem \ref{t:fvdv}, we will consider two families of kernels satisfying the set of sufficient conditions for functional convergence pointed out in the previous section. As explained in \cite[Section 3]{van_der_vaart} both types of $U$-statistics can be used in the non-parametric estimation of quadratic functional of densities -- see also \cite{BR88, LM00}

\begin{itemize}

\item[(I)] ({\it Wavelet-based kernels}) 
Following \cite[Section 4.1]{van_der_vaart}, we consider expansions of functions $f\in L_2(\mathbb{R}^d)$ on an orthonormal basis of compactly supported, bounded wavelets of the form
$$f(x)=\sum_{j\in\mathbb{Z}^d}\sum_{v\in\lbrace 0,1\rbrace ^d}\left<f,\psi_{0,j}^v\right>\psi_{0,j}^v(x)+\sum_{i=0}^{\infty}\sum_{j\in\mathbb{Z}^d}\sum_{v\in\lbrace 0,1\rbrace ^d\setminus\lbrace 0\rbrace}\left<f,\psi_{i,j}^v\right>\psi_{i,j}^v(x).$$
The functions $\psi_{i,j}^v$ are orthogonal for different indices $(i,j,v)$ and given by scaled and translated versions of the $2^d$ base functions $\psi_{0,0}^v$:
$$\psi_{i,j}^v(x)=2^{id/2}\psi_{0,0}^v(2^ix-j).$$
We concentrate on functions $f$ with support in $E=[0,1]^d$. As noted in \cite[Section 4.1]{van_der_vaart}, for each resolution level $i$ and vector $v$, only the order $2^{id}$ elements $\psi_{i,j}^v$ are nonzero in $E$. We denote the corresponding set of indices $j$ by $J_i$. We then truncate the expansion at the level of resolution $i=I$ and look at the kernel
$$K_n(x,y)=\sum_{j\in J_0}\sum_{v\in\lbrace 0,1\rbrace^d}\psi_{0,j}^v(x)\psi_{0,j}^v(y)+\sum_{i=0}^I\sum_{j\in J_i}\sum_{v\in\lbrace 0,1\rbrace^d\setminus\lbrace 0\rbrace}\psi_{i,j}^v(x)\psi_{i,j}^v(y).$$

\item[(II)] ({\it Kernels based on Fourier expansions}) 
Any function $f\in L_2[-\pi,\pi]$ can be represented through the Fourier series $f=\sum_{j\in\mathbb{Z}}f_je_j$ for $e_j(x)=e^{ijx/\sqrt{2\pi}}$ and $f_j=\int_{-\pi}^{\pi}fe_jd\lambda$, where $\lambda$ is the Lebesgue measure. We can write $f_k=\sum_{|j|\leq k}f_je_j$ to obtain an orthogonal projection of $f$ onto a $2k+1$-dimensional space. Assuming that $k$ depends on $n$, we can also write down the corresponding kernel as:
$$K_n(x,y)=\sum_{|j|\leq k}e_j(x)e_j(y)=\frac{\sin\left(\left(k+\frac{1}{2}\right)(x-y)\right)}{2\pi\sin\left(\frac{1}{2}(x-y)\right)}$$
and note that $K_n(x,y)=D_k(x-y)$, where $D_k$ is the well-known Dirichlet kernel.

\end{itemize}

\begin{theorem} Let the above assumption and notation prevail.

 \begin{enumerate}

\item[\rm 1.] Let $\mu$ be any probability measure on $[0,1]^d$ with a Lebesgue density that is bounded and bounded away from zero. The sequence of wavelet-based kernels $\{ K_n : n\geq 1\}$ defined at Point {\rm (I)} above satisfies the assumptions of Theorem \ref{t:vdv}, with respect to $\mu$, as soon as $n\ll k_n\ll n^2$, for $k_n=2^{Id}$. Moreover, a sufficient condition for such a sequence to satisfy the assumptions of Theorem \ref{t:fvdv} is $n^{1+\gamma_1}\lesssim k_n\lesssim n^{2-\gamma_2}$, for some $\gamma_1,\gamma_2>0$.
\item[\rm 2.] Let $\mu$ be any measure on $\mathbb{R}$ with a bounded Lebesgue density and $k_n=2k+1$. The sequence of Fourier-based kernels $\{ K_n : n\geq 1\}$ defined at Point {\rm (II)} above satisfies the assumptions of Theorem \ref{t:vdv} as soon as $n\ll k_n\ll n^2$. In addition, a sufficient condition for such a sequence to meet the assumptions of Theorem \ref{t:fvdv} is $n^{1+\eta_1}\lesssim k_n\lesssim n^{2-\eta_2}$, for some $\eta_1,\eta_2>0$.

\end{enumerate}

\end{theorem}

\begin{proof}
\begin{enumerate}
\item
For $n\ll k_n\ll n^2$ and $K_n$ defined in point {\rm (I)} above, the assumptions of Theorem \ref{t:vdv} are verified in \cite[Proposition 4.1]{van_der_vaart}. The authors note that, by assumption, each function $\psi_{I,j}^v$ is supported within a set of the form $2^{-I}(C+j)$ for a given cube $C$ that depends on the type of the wavelet, for any $v$. They take $\mathcal{X}_{n,m}$ to be blocks (cubes) of $l_n^d$ adjacent cubes $2^{I}(C+j)$, giving $M_n=O(k_n/l_n^d)$ sets $\mathcal{X}_{n,m}$. In order for the assumptions (\ref{assumption1})-(\ref{assumption6}) to be satisfied, the authors require that 
\begin{enumerate}[a)]
\item $M_n\to\infty$, 
\item $M_n\lesssim n$, 
\item $\frac{M_n}{k_n}\to 0$, 
\item $M_n^{-1}k_n/n\to 0$. 
\end{enumerate}
Now assume $n^{1+\gamma_1}\lesssim k_n\lesssim n^{2-\gamma_2}$ for some $\gamma_1,\gamma_2>0$. Condition \eqref{assumption8} is then automatically satisfied.
As noted in the proof of \cite[Proposition 4.1]{van_der_vaart}, $\mu(\mathcal{X}_{m,n})$ is of order $\frac{1}{M_n}$. Now, it is also noted in the proof of \cite[Proposition 4.1]{van_der_vaart} that, if $K_n(x_1,x_2)\neq 0$ then there exists some $j$ such that $x_1,x_2\in 2^{-I}(C+j)$. Moreover, the set of $(x_1,x_2)$ in the complement of $\bigcup_m\mathcal{X}_{n,m}\times\mathcal{X}_{n,m}$ where $K_n(x_1,x_2)\neq 0$ is contained in the union $U$ of all cubes $2^{-I}(C+j)$ that intersect the boundary of some $\mathcal{X}_{n,m}$. It is also noted that the number of such cubes is of order $M_n^{1/d}k_n^{1-1/d}$ and that $\mu\left(2^{-I}(C+j)\right)\lesssim \frac{1}{k_n}$. Therefore, using $\|K_n\|_{\infty}\lesssim k_n$, we note that, for any $\alpha_2>0$,
\begin{align*}
&\frac{n^{\alpha_2}}{k_n}\int\int\left| K_n(x,y){\bf 1}\left[(x,y)\in\bigcap \left(\mathcal{X}_{n,m}\times\mathcal{X}_{n,m}\right)^c\right]\right|^2\mu(dx)\mu(dy)\\
\lesssim &\frac{n^{\alpha_2}k_n^2}{k_n} M_n^{1/d} k_n^{1-1/d}\left(\frac{1}{k_n}\right)^2\\
=& n^{\alpha_2}\left(\frac{M_n}{k_n}\right)^{1/d}.
\end{align*}
Condition \eqref{assumption9} requires this quantity to be bounded for some $\alpha_2>0$. It will indeed be bounded for $\alpha_2\leq\frac{1}{2d}$ if we choose $M_n=k_n^{1/2}$. Moreover, for $M_n=k_n^{1/2}$, $\frac{M_n}{k_n}\to 0$. Also, $\frac{k_n}{M_nn}=\frac{k_n^{1/2}}{n}\to 0$, as $k_n\ll n^2$. Under the same assumption, $M_n\to\infty$ and $n^{1/2+\gamma_1/2}\lesssim M_n\lesssim n^{1-\gamma_2/2}\lesssim n$ and so  conditions  \eqref{assumption7} and \eqref{assumption7.1} are also satisfied (as $\mu(\mathcal{X}_{m,n})$ is of order $\frac{1}{M_n}$). Therefore all the conditions a), b), c), d) from above, as well as conditions \eqref{assumption7}-\eqref{assumption9}, are satisfied. This finishes the proof.

\item For $n\ll k_n\ll n^2$, the assumptions (\ref{assumption1})-(\ref{assumption6}) for kernel $K_n$ of Point {\rm (II)} are verified in \cite[Proposition 4.2]{van_der_vaart}. The authors take a partition $(-\pi, \pi]=\bigcup_m\mathcal{X}_{n,m}$ in $M_n=\frac{2\pi}{\delta}$ intervals of length $\delta$ for $\delta\to 0$ for $\frac{1}{\sqrt{k_n}}\ll \delta\ll \frac{n}{k_n}$ and introduce an $\epsilon>0$ such that $\frac{1}{\sqrt{k_n}}\ll \epsilon\ll\delta$.

Now, assume that $n^{1+\eta_1}\lesssim k_n\lesssim n^{2-\eta_2}$, for some $\eta_1,\eta_2>0$. This makes condition (\ref{assumption8}) readily satisfied.
In order for condition (\ref{assumption7}) to be satisfied, we require that $n^{1/2+\epsilon_1}\delta$ is bounded for some $\epsilon_1>0$. Moreover, condition \eqref{assumption7.1} will be satisfied if $\delta\lesssim \frac{n^{1-\epsilon_2}}{k_n}$ for some $\epsilon_2>0$.

The authors of \cite{van_der_vaart} note that the complement of $\bigcup_m\mathcal{X}_{n,m}\times \mathcal{X}_{n,m}$ is contained in $\lbrace (x_1,x_2):|x_1-x_2|>\epsilon\rbrace$ except for the set of $2(M_n-1)$ triangles indicated in \cite[Figure 3]{van_der_vaart}. Now, by the argument of the proof of \cite[Proposition 4.2]{van_der_vaart}, for any $\alpha_2>0$,
\begin{align*}
\frac{n^{\alpha_2}}{k_n}\int_{-\pi}^{\pi}\int_{-\pi}^{\pi}{\bf 1}[|x-y|>\epsilon]\left| K_n(x,y)\right|^2\mu(dx)\mu(dy)\lesssim \epsilon n^{\alpha_2}+\frac{n^{\alpha_2}}{\epsilon^2 k_n},
\end{align*}
which is bounded, if $\frac{n^{\alpha_2/2}}{k_n^{1/2}}\lesssim\epsilon\lesssim n^{-\alpha_2}$.
Each of the remaining triangles in the complement of $\bigcup_m\mathcal{X}_{n,m}\times\mathcal{X}_{n,m}$ has sides of length of order $\epsilon$. Hence, for a typical triangle $\Delta$ and an interval $I$ of length of the order $\epsilon$,
\begin{align}\label{triangles}
\frac{n^{\alpha_2}}{k_n}\int\int_{\Delta} \left|K_n(x,y)\right|^2dx dy\underset{u=x-y}{\stackrel{v=y}\lesssim} \frac{n^{\alpha_2}}{k_n}\int_I \int_0^{\epsilon}|D_n(u)|^2du dv\lesssim \frac{n^{\alpha_2}}{k_n}\epsilon k_n=\epsilon n^{\alpha_2}.
\end{align}
There are $2(M_n-1)$ such triangles.
Therefore, condition \eqref{assumption8} will be satisfied if, in addition to  $\frac{n^{\alpha_2/2}}{k_n^{1/2}}\lesssim\epsilon\lesssim n^{-\alpha_2}$, $M_n\epsilon  n^{\alpha_2}=\frac{2\pi\epsilon n^{\alpha_2}}{\delta}$ is bounded, i.e. $\epsilon n^{\alpha_2}\lesssim \delta$.

Summing up, conditions \eqref{assumption1} - \eqref{assumption9} are satisfied if, for some $\delta,\epsilon,\epsilon_1,\epsilon_2,\alpha_2>0$,
$$\frac{n^{3\alpha_2/2}}{k_n^{1/2}}\lesssim \epsilon n^{\alpha_2}\lesssim\delta\lesssim\min\left(\frac{1}{n^{1/2+\epsilon_1}},\frac{n^{1-\epsilon_2}}{k_n}\right).$$
Such choices of $\delta,\epsilon,\epsilon_1,\epsilon_2,\alpha_2>0$ exist. Indeed, let $\alpha_2=\frac{\min(\eta_1,\eta_2)}{6}$, \\$\epsilon_1=\frac{1}{2}\left(\eta_1-\frac{\min(\eta_1,\eta_2)}{2}\right)$, $\epsilon_2=\frac{1}{2}\left(\eta_2-\frac{\min(\eta_1,\eta_2)}{2}\right)$. Then, under the assumption $n^{1+\eta_1}\lesssim k_n\lesssim n^{2-\eta_2}$, we have that $\frac{n^{3\alpha_2/2}}{k_n^{1/2}}\lesssim\min\left(\frac{1}{n^{1/2+\epsilon_1}},\frac{n^{1-\epsilon_2}}{k_n}\right)$ and it suffices to choose $\delta=\frac{n^{3\alpha_2/2}}{k_n^{1/2}}$ and $\epsilon=\frac{n^{\alpha_2/2}}{k_n^{1/2}}.$ This finishes the proof.

\end{enumerate}
\end{proof}

\section{Technical results and proofs of main statements }\label{s:proofs}

Unless otherwise specified, for the rest of the section we adopt the same conventions and notation put forward in Section \ref{setting}.

\subsection{A new product formula}

We start by proving a new product formula for symmetric $U$-statistics with arguments of possibly different sizes. 
In order to state it, we need to recall the Hoeffding decomposition of not necessarily symmetric kernel functions: Let $f\in L^1(\mu^p)$. Then, $f$ can be decomposed as follows: For all $(x_1,\dotsc,x_p)\in E^p$ one has 
\begin{align}\label{HDnonsym}
 f(x_1,\dotsc,x_p)&=\sum_{J\subseteq[p]} f_J\bigl((x_i)_{i\in J}\bigr)\,,
\end{align}
where we follow the convention that in $(x_i)_{i\in J}$ the coordinates $i$ appear in increasing order, i.e. if $J=\{i_1,\dotsc,i_k\}$ with $k=|J|$ and $1\leq i_1<\ldots<i_k\leq p$, then $(x_i)_{i\in J}=(x_{i_1},\dotsc,x_{i_k})$. 
The kernels $f_J$, $J\subseteq[p]$, are given by 
\begin{align}\label{HDnonsym2}
 f_J\bigl((x_i)_{i\in J}\bigr)&=\sum_{K\subseteq J} (-1)^{|J|-|K|}\int_{E^{p-|K|}}f(x_1,\dotsc,x_p)d\mu^{p-|K|}\bigl((x_i)_{i\in[p]\setminus K}\bigr)
\end{align}
and they are \textit{canonical} with respect to $\mu$ in the sense that for each $\emptyset\not=J\subseteq p$ with $|J|=k$, each $j\in J$ and all $(x_i)_{i\in J\setminus\{j\}}\in E^{|J|-1}$ one has that 
\begin{align}\label{can}
 \int_E f_J\bigl(x_{i_1},\dotsc,x_{i_{l-1}},y,x_{i_{l+1}},\dotsc,x_{i_k}\bigr)d\mu(y)=0\,,
\end{align}
where we again suppose that $J=\{i_1,\dotsc,i_k\}$, $1\leq i_1<\ldots<i_k\leq p$ and where $i_l=j$. For a detailed discussion and proofs of these facts we refer the reader to \cite[Chapter 9]{Major}. Note that, if the kernel $f$ is symmetric as in 
Section \ref{setting}, then we can define the (symmetric) functions $g_k$, $0\leq k\leq p$ by 
\[g_k(y_1,\dotsc,y_k)=\int_{E^{p-k}}f\bigl(y_1,\dotsc,y_k,x_1,\dotsc,x_{p-k}\bigr)d\mu^{p-k}(x_1,\dotsc,x_{p-k})\]
as before and we obtain that, for every subset $J\subseteq[p]$ with $1\leq k:=|J| \leq p$, 
\begin{align*}
 f_J(x_1,\dotsc,x_k)&=\sum_{l=0}^s (-1)^{k-l}\sum_{1\leq i_1<\ldots<i_l\leq k} g_l\bigl(x_{i_1},\dotsc,x_{i_l}\bigr)=f_k(x_1,\dotsc,x_k)\,,
\end{align*}
where the symmetric and degenerate kernel $f_k$ has been defined in \eqref{defpsis}.

For the statement of our product formula we have to fix some more notation: Let us fix two positive integers $p$ and $q$. Then, for nonnegative integers $l,n,m,r$ such that $n\leq m$, 
$r\leq p\wedge q$, $l\geq p+q-2r$ and sets $L\subseteq[m]$ with $|L|=l$, we denote by $\Pi_{r,n,m}(L)$ the collection of all triples 
\[(A,B,C)\in\D_{2r+l-p-q}(n)\times \D_{p-r}(n)\times \D_{q-r}(m)\]  
such that $L$ is the disjoint union of $A$, $B$ and $C$.

\begin{proposition}[\bf Product formula]\label{pform}
Let $p,q\geq1$ be positive integers and assume that $\psi\in L^2(\mu^{ p})$ and $\phi\in L^2(\mu^{ q})$ are degenerate, symmetric kernels of orders $p$ and $q$ respectively. Moreover, let $n\geq p$ and $m\geq q$ be positive integers with $m\geq n$.
Then, whenever $n\geq p+q$ we have the Hoeffding decomposition:
\[J_p^{(n)}(\psi) J_q^{(m)}(\phi)=\sum_{\substack{M\subseteq[m]:\\\abs{M}\leq p+q}} U_M=\sum_{k=|p-q|}^{p+q}\sum_{s=0}^{q\wedge k\wedge(m-n)}\sum_{\substack{M\subseteq[m]:\\\abs{M}=k,\\ \abs{M\cap\{n+1,\dotsc,m\}}=s}} U_M\,,\]
where, for a set $M \subseteq[m]$ with $0\leq k:=|M|\leq p+q$ and $0\leq s:=\abs{M\cap\{n+1,\dotsc,m\}}\leq q\wedge k\wedge(m-n)$, the Hoeffding component $U_M$ is given by 
\begin{align}\label{prodform}
U_M&= \sum_{r=\lceil \frac{p+q-k}{2}\rceil}^{p\wedge (q-s)\wedge(p+q-k)}\binom{n-k+s}{p+q-r-k}\notag\\
&\hspace{3cm}\cdot\sum_{(A,B,C)\in\Pi_{r,n,m}(M)} \bigl(\psi\star_r^{p+q-r-k}\phi\bigr)_M\bigl((X_i)_{i\in A}, (X_i)_{i\in B}, (X_i)_{i\in C}\bigr)\,.
\end{align}
Moreover, for such an $M$, we further have the bound 
\begin{align}\label{varum}
 &\sqrt{\Var(U_M)}\notag\\
\leq & \hspace{-3mm}\sum_{r=\lceil \frac{p+q-k}{2}\rceil}^{p\wedge (q-s)\wedge(p+q-k)}\binom{n-k+s}{p+q-r-k}\binom{k-s}{2r+k-p-q,p-r,q-r-s} \|\psi\star_r^{p+q-r-k}\phi\bigr\|_{L^2(\mu^{k})}\,.
\end{align}

\end{proposition}

\begin{remark}\label{rempform}{\rm 
The above product formula is an extension of the one proved in \cite[Proposition 2.6]{DP18}, for symmetric and degenerate $U$-statistics based on the same range $X_1,\dotsc,X_n$ of data. Indeed, suppose that $n=m$. Then, if $|M|=k$ and (necessarily) $s=0$ it is not hard to verify that 
\begin{align*}
 &\sum_{(A,B,C)\in\Pi_{r,n,m}(M)} \bigl(\psi\star_r^{p+q-r-k}\phi\bigr)_M\bigl((X_i)_{i\in A}, (X_i)_{i\in B}, (X_i)_{i\in C}\bigr)\\
 &=\frac{k!}{(2r+k-p-q)!(p-r)!(q-r)!}\bigl(\widetilde{\psi\star_r^{p+q-r-k}\phi}\bigr)_k\bigl(X_i,i\in M\bigr)
\end{align*}
and the product formula reduces to the one in \cite[Proposition 2.6]{DP18}.
The main difference in general is that, in the situation of Proposition \ref{pform} and for $n\not = m$, the product is no longer (in general) a finite sum of degenerate and symmetric $U$-statistics. However, its Hoeffding decomposition (in the sense of not necessarily symmetric statistics --- see e.g. \cite{KR, dobler_peccati}) is still 
completely explicit and hence suitable for providing useful bounds. 
}
\end{remark}

\begin{proof}[Proof of Proposition \ref{pform}]
 Write
\begin{equation*}
W:=J_p^{(n)}(\psi)=\sum_{J\in\D_p(n)}W_J\quad\text{and}\quad V:=J_q^{(m)}(\phi)=\sum_{K\in\D_q(m)} V_K
\end{equation*}
for the respective Hoeffding decompositions of $W$ and $V$, i.e.
\begin{align*}
 W_J=\psi(X_j,j\in J)\,,\quad J\in\D_p(n)\quad\text{and}\quad V_K=\phi(X_i,i\in K)\,,\quad K\in\D_q(m)\,.
\end{align*}
From Theorem 2.6 in \cite{dobler_peccati} we know that the Hoeffding decomposition of $VW$  is given by 
\begin{equation*}
 VW=\sum_{\substack{M\subseteq[m]:\\\abs{M}\leq p+q}} U_M\,,
\end{equation*}
where, for $M\subseteq[m]$ with $\abs{M}\leq p+q$ we have 
\begin{align}
U_M&=\sum_{L\subseteq M}(-1)^{\abs{M}-\abs{L}}
\sum_{\substack{J\in\D_p(n),K\in\D_q(m):\\J\Delta K\subseteq L,\\
M\subseteq J\cup K}}\E\bigl[W_JV_K\,\bigl|\,\F_L\bigr]\,,\label{hdvw}
\end{align}
where $\F_L:=\sigma(X_j,j\in L)$.
Note that $U_M=0$ whenever $\abs{M}<\abs{p-q}$ because $\abs{J\Delta K}\geq \abs{p-q}$ for all $J\in\D_p(n)$ and $K\in\D_q(m)$. 
Moreover, $U_M=0$ if $|M\cap\{n+1,\dotsc, m\}|> q$ since $M\cap\{n+1,\dotsc, m\}\subseteq K$ and $|K|=q$. Hence, we have 
\begin{align}\label{pf1}
 VW=\sum_{\substack{M\subseteq[m]:\\\abs{M}\leq p+q\,,\\ |M\cap\{n+1,\dotsc, m\}|\leq q }} U_M=\sum_{k=|p-q|}^{p+q}\sum_{s=0}^{q\wedge k} \sum_{\substack{M\subseteq[m]:\\\abs{M}=k\,,\\ |M\cap\{n+1,\dotsc, m\}|=s }} U_M\,.
\end{align}
Since $K\cap\{n+1,\dotsc, m\}\subseteq J\Delta K\subseteq L\subseteq M$ and $M\cap\{n+1,\dotsc, m\}\subseteq K\cap\{n+1,\dotsc, m\}$ this also implies that we can 
restrict our attention to sets $M$ and $L$ that satisfy
\[M\cap\{n+1,\dotsc, m\}=L\cap\{n+1,\dotsc, m\}=K\cap\{n+1,\dotsc, m\}\,.\]
 Writing $k:=|M|$, $r:=\abs{J\cap K}$, $l:=\abs{L}$ and $s:=|M\cap\{n+1,\dotsc, m\}|$ it follows that $r\leq p\wedge(q-s)$ and, since 
 \begin{equation}\label{inteq}
  \abs{J\cap K}=\abs{J}+\abs{K}-\abs{J\cup K}=p+q-\abs{J\cup K}\,,
 \end{equation}
it follows from $M\subseteq J\cup K$ that $r\leq p+q-k$. Moreover, since $|J\cup K|=\abs{J\cap K}+\abs{J\Delta K}$ and $J\Delta K\subseteq L\subseteq M$, it follows again from \eqref{inteq} that 
$2r\geq p+q-l\geq p+q-k$. In particular, we have $l\geq \abs{p-q}\vee (p+q-2r)=p+q-2r$. Moreover, note that 
\begin{align*}
 \abs{J\cap K\cap L}&=\abs{L}-\abs{L\cap(J\Delta K)}=\abs{L}-\abs{J\Delta K}=l-(p+q-2r)=2r+l-p-q\,,\\
 \abs{(J\cap K)\setminus L}&=\abs{(J\cap K)}-\abs{J\cap K\cap L}=r-(2r+l-p-q)=p+q-r-l\quad\text{and}\\
 \abs{L\setminus(J\cap K)}&=\abs{L}-\abs{J\cap K\cap L}=l-(2r+l-p-q)=p+q-2r\,.
\end{align*}
Note that we have
\begin{align}\label{hdvw2}
 \E\bigl[W_JV_K\,\bigl|\,(X_i)_{i\in L}\bigr]&=\E\bigl[\psi(X_j,j\in J)\phi(X_k,k\in K)\,\bigl|\,X_i,i\in L\bigr]\notag\\
 &=\bigl(\psi\star_r^{p+q-r-l} \phi\bigr)\bigl((X_i)_{i\in L\cap J\cap K},(X_j)_{j\in J\setminus K}, (X_k)_{k\in K\setminus J}\bigr)\,.
\end{align}
Let us now fix $M$ and $L$.
Then, for each $(A,B,C)\in\Pi_{r,n,m}(L)$, there are precisely 
\[\binom{n-k+s}{p+q-r-k}\]
pairs $(J,K)\in\D_p(n)\times\D_q(m)$ such that $M\subseteq J\cup K$, $J\cap K\cap L=A$, $J\setminus K=B$ and $K\setminus J=C$. Indeed, given these restrictions it only remains to choose the elements of $(J\cap K)\setminus L\subseteq[n]$ in such a way 
that 
\[M\setminus L\subseteq (J\cap K)\setminus L\,.\]
The claim now follows from the facts that $\abs{M\cap[n]}=k-s$,  
\[\bigl((J\cap K)\setminus L\bigr)\setminus(M\setminus L)=(J\cap K)\setminus M\]
and
\begin{equation*}
 \babs{(J\cap K)\setminus L}-\babs{M\setminus L}=p+q-r-l-(k-l)=p+q-r-k\,.
\end{equation*}
Thus we have proved that 
\begin{align}\label{pf2}
 U_M&=\sum_{l=|p-q|}^k\sum_{\substack{L\subseteq M:\\|L|=l}}(-1)^{\abs{M}-\abs{L}}\sum_{r=\lceil \frac{p+q-l}{2}\rceil}^{p\wedge (q-s)\wedge(p+q-k)}\binom{n-k+s}{p+q-r-k}\notag\\
 &\hspace{3cm}\cdot\sum_{(A,B,C)\in\Pi_{r,n,m}(L)}\Bigl(\psi\star_r^{p+q-r-l}\phi\Bigr)\bigl((X_i)_{i\in A}, (X_i)_{i\in B}, (X_i)_{i\in C}\bigr)\notag\\
 &= \sum_{r=\lceil \frac{p+q-k}{2}\rceil}^{p\wedge (q-s)\wedge(p+q-k)}\binom{n-k+s}{p+q-r-k} \sum_{l=p+q-2r}^k  \sum_{\substack{L\subseteq M:\\|L|=l}}(-1)^{\abs{M}-\abs{L}}\notag\\
 &\hspace{2cm}\sum_{(A,B,C)\in\Pi_{r,n,m}(L)}\Bigl(\psi\star_r^{p+q-r-l}\phi\Bigr)\bigl((X_i)_{i\in A}, (X_i)_{i\in B}, (X_i)_{i\in C}\bigr) \,,
\end{align}
Now, suppose that $(A,B,C)\in\Pi_{r,n,m}(M)$ such that, in particular, $\abs{A}=2r+k-p-q$. Moreover, suppose that 
\begin{align*}
 T_L(A,B,C)&:=\E\Bigl[\bigl(\psi\star_r^{p+q-r-k}\phi\bigr)\bigl((X_i)_{i\in A}, (X_i)_{i\in B}, (X_i)_{i\in C}\bigr)\,\bigl|\,\F_L\Bigr]\not=0\,.
\end{align*}
Then, it is easy to see that $B\cup C\subseteq L$, that
\begin{align*}
  T_L(A,B,C)&=(\psi\star_r^{p+q-r-l}\phi\bigr)\bigl((X_i)_{i\in A\cap L}, (X_i)_{i\in B}, (X_i)_{i\in C}\bigr)
\end{align*}
and that $(A\cap L,B,C)\in\Pi_{r,n,m}(L)$. Moreover, for each given $(A,B,C)\in\Pi_{r,n,m}(L)$, there is a unique $(\hat{A},B,C)\in\Pi_{r,n,m}(M)$ such that $(\hat{A}\cap L,B,C)=(A,B,C)$, namely one has to take
$\hat{A}=A\cup(M\setminus L)$. From these observations we infer that 
\begin{align}\label{pf5}
& \sum_{(A,B,C)\in\Pi_{r,n,m}(L)}\Bigl(\psi\star_r^{p+q-r-l}\phi\Bigr)\bigl((X_i)_{i\in A}, (X_i)_{i\in B}, (X_i)_{i\in C}\bigr)\notag\\
 =&\sum_{(\hat{A},B,C)\in \Pi_{r,n,m}(M)}  T_L(\hat{A},B,C)\,.
\end{align}
Now, recall that by the Hoeffding decomposition for non-symmetric kernels, for each $(A,B,C)\in\Pi_{r,n,m}(M)$ we have that 
\begin{align}\label{pf4}
 &\bigl(\psi\star_r^{p+q-r-k}\phi\bigr)_M\bigl((X_i)_{i\in A}, (X_i)_{i\in B}, (X_i)_{i\in C}\bigr)\notag\\
 &=\sum_{L\subseteq M}(-1)^{\abs{M}-\abs{L}}\E\Bigl[\bigl(\psi\star_r^{p+q-r-k}\phi\bigr)\bigl((X_i)_{i\in A}, (X_i)_{i\in B}, (X_i)_{i\in C}\bigr)\,\bigl|\,\F_L\Bigr]\notag\\
 &= \sum_{L\subseteq M}(-1)^{\abs{M}-\abs{L}}T_L(A,B,C)  \,.
\end{align}
Thus, from \eqref{pf2}, \eqref{pf5} and \eqref{pf4} we can conclude that 
\begin{align}\label{pf3}
U_M&= \sum_{r=\lceil \frac{p+q-k}{2}\rceil}^{p\wedge (q-s)\wedge(p+q-k)}\binom{n-k+s}{p+q-r-k} \notag\\
&\hspace{4cm}\cdot\sum_{l=p+q-2r}^k  \sum_{\substack{L\subseteq M:\\|L|=l}}(-1)^{\abs{M}-\abs{L}}\sum_{(A,B,C)\in\Pi_{r,n,m}(M)}T_L(A,B,C)\notag\\
 &=\sum_{r=\lceil \frac{p+q-k}{2}\rceil}^{p\wedge (q-s)\wedge(p+q-k)}\binom{n-k+s}{p+q-r-k} \sum_{(A,B,C)\in\Pi_{r,n,m}(M)}  \sum_{L\subseteq M}(-1)^{\abs{M}-\abs{L}}T_L(A,B,C)\notag\\
 &= \sum_{r=\lceil \frac{p+q-k}{2}\rceil}^{p\wedge (q-s)\wedge(p+q-k)}\binom{n-k+s}{p+q-r-k} \notag\\
 &\hspace{3cm}\cdot\sum_{(A,B,C)\in\Pi_{r,n,m}(M)} \bigl(\psi\star_r^{p+q-r-k}\phi\bigr)_M\bigl((X_i)_{i\in A}, (X_i)_{i\in B}, (X_i)_{i\in C}\bigr)\,,
\end{align}
as claimed. The bound \eqref{varum} then follows immediately from 
\begin{equation*}
 \|\bigl(\psi\star_r^{p+q-r-k}\phi\bigr)_M\bigr\|_{L^2(\mu^{k})}\leq\|\psi\star_r^{p+q-r-k}\phi\bigr\|_{L^2(\mu^{k})}
\end{equation*}
and from the fact that 
\[|\Pi_{r,n,m}(M)|=\binom{k-s}{2r+k-p-q,p-r,q-r-s}\,.\]
\end{proof}

In the next subsection, we focus on convergence of finite-dimensional distributions (f.d.d.) for processes of the form \eqref{genW}. Our approach extends the general (quantitative) CLTs from \cite{DP18}.

\subsection{F.d.d. convergence}\label{fidi}
\subsubsection{A general qualitative multivariate CLT}\label{qmclt}
Fix a positive integer $d$ and, for $1\leq i\leq d$ and $n\in\N$, let $p_i\leq m_{n,i}\leq n$ be positive integers. We will always assume that the sequences {$\{ m_{n,i} : {n\in\N}\}$ diverge to $\infty$ as $n\to\infty$, for each $i=1,\dotsc,d$, in such a way that there are positive constants $0<c_i\leq1$ such that $c_i n\leq m_{n,i}\leq n$ for all $n\in\N$, s.t. we have $m_{n,i}\asymp n$ for $i=1,\dotsc,d$.} Moreover, let $\psi^{(i)}=\psi^{(i,n)}\in L^4(\mu^{p_i})$ be degenerate kernels . Define 
\[\phi^{(i)}=\phi^{(i,n)}:=\frac{\psi^{(i)}}{\sqrt{\binom{m_{n,i}}{p_i}}}\] 
as well as
\begin{equation*}
 \sigma_n(i)^2:=\Var\bigl(J_{p_i}^{(m_{n,i})}(\phi^{(i)})\bigr)=\norm{\psi^{(i)}}^2_{L^2(\mu^{p_i})}\,.
\end{equation*}
For $i=1,\dotsc,d$ write $Y(i):=J_{p_i}^{(m_{n,i})}(\phi^{(i)})$ as well as 
\[Y=Y_n:=(Y_1,\dotsc,Y_d)^T\,.\]
Then, $Y$ is a centered random vector with components in $L^4(\P)$. We will write $\mathbf{V}=\mathbf{V}_n=\{ v_{i,k} : {1\leq i,k\leq d}\}$ for its covariance matrix. Throughout the section, we denote by 
$Z=Z_n=(Z_1,\dotsc,Z_d)^T\sim N_d(0,\mathbf{V})$
a centered Gaussian vector with the same covariance matrix as $Y$. Note that, due to degeneracy, we have $v_{i,k}=0$ unless $p_i=p_k$. The following finite-dimensional CLT is one of our crucial tools.

\begin{proposition}\label{fidiprop}
With the above notation and definitions, assume that \\$\mathbf{C}:=\lim_{n\to\infty}\mathbf{V}_n \in\R^{d\times d}$ exists. Then, $Y_n$ converges in distribution to $N_d(0,\mathbf{C})$, provided Conditions {\normalfont(i)-(iii)} below hold for all $1\leq i\leq k\leq d\,$:
\begin{enumerate}[{\normalfont(i)}]
\item $\displaystyle\lim_{n\to\infty}n^{a/2-r}\norm{\psi^{(i,n)}\star_r^{a-r}\psi^{(k,n)}}_{L^2(\mu^{p_i+p_k-a})}=0$ for all pairs $(a,r)$ of integers such that\\
$1\leq a\leq \min(p_i+ p_k-1, 2(p_i\wedge p_k))$ and $\lceil\frac{a}{2}\rceil\leq r\leq a\wedge p_i\wedge p_k$,
\item $\displaystyle \lim_{n\to\infty}n^{a/2-r}\norm{\psi^{(i,n)}}_{L^2(\mu^{p_i})}\norm{\psi^{(i,n)}\star_r^{a-r}\psi^{(i,n)}}_{L^2(\mu^{2p_i-a})}=0$ for all for all pairs $(a,r)$ of integers such that $1\leq a\leq 2p_i-1$ and $\lceil\frac{a}{2}\rceil\leq r\leq a\wedge p_i$, and 
\item $\displaystyle \lim_{n\to\infty}\frac{\norm{\psi^{(i,n)}}^3_{L^2(\mu^{p_i})}}{\sqrt{n}}=0$.
\end{enumerate}
\end{proposition}

\begin{proof} For $1\leq i,k\leq d$, we use the notation
\begin{equation*}
 Y(i)Y(k)=\sum_{\substack{M\subseteq[n]: \abs{M}\leq p_i+p_k}} U_M(i,k)
\end{equation*}
to indicate the Hoeffding decomposition of $Y(i)Y(k)$. The following bound is taken from Lemma 4.1 in \cite{DP18}: for $h\in C^3(\R^d)$ whose partial derivatives up to 
order three are all bounded, there exist constants $\tilde{M}_2(h), M_3(h)\in(0,\infty)$ such that 
\begin{align*}
 &\babs{\E[h(Y)]-\E[h(Z)]}\leq \frac{1}{4p_1}\tilde{M}_2(h)\sum_{i,k=1}^d(p_i+p_k)\biggl(\sum_{\substack{M\subseteq[n]:\\\abs{M}\leq p_i+p_k-1}}\Var\bigl(U_M(i,k)\bigr)\biggr)^{1/2}\\
&\;+\frac{2M_3(h)\sqrt{d}}{9p_1}\sum_{i=1}^d p_i\sigma_n(i)\biggl(\sum_{\substack{M\subseteq[n]:\\\abs{M}\leq 2p_i-1}}\Var\bigl(U_M(i,i)\bigr)\biggr)^{1/2}\\
&\;+\frac{\sqrt{2d}M_3(h)}{9p_1\sqrt{n}}\sum_{i=1}^d p_i^{3/2}\sigma_n(i)^3\sqrt{\kappa_{p_i}}\,,
\end{align*}
and each finite constant $\kappa_{p_i}$ only depends on $p_i$, $1\leq i\leq d$. We now apply Proposition \ref{pform} in order to bound 
\begin{equation*}
\biggl(\sum_{\substack{M\subseteq[n]:\\\abs{M}\leq p_i+p_k-1}}\Var\bigl(U_M(i,k)\bigr)\biggr)^{1/2}
\end{equation*}
for $1\leq i\leq k\leq d$. We will, for notational convenience, assume that $m_{n,i}\leq m_{n,k}$. 
Moreover, for integers $p,q\geq0$ we will write $M(p,q):=\min(2(p\wedge q),p+q-1)$.  From \eqref{varum} we know that for $M\subseteq [m_{n,k}]$ such that $|M|=p_i+p_k-a$ for some $a\in\{1,\dotsc,M(p_i,p_k)\}$ \\
and $|M\cap\{m_{n,i}+1,\dotsc, m_{n,k}\}|= s\in\{0,1,\dotsc,p_k\wedge(p_i+p_k-a)\wedge(m_{n,k}-m_{n,i})\}$ we have
\begin{align}\label{varum2}
& \sqrt{\Var\bigl(U_M(i,k)\bigr)}\notag\\
\leq&\hspace{-2mm}\sum_{r=\lceil\frac{a}{2}\rceil}^{a\wedge p_i\wedge (p_k-s)}\binom{m_{n,i}-p_i-p_k+a+s}{a-r}\binom{p_i+p_k-a}{p_i-r,p_k-r,2r-a}\|\phi^{(i)}\star_r^{a-r}\phi^{(k)}\|_{L^2(\mu^{p_i+p_k-a})}\notag\\
 &=:b_{i,k}(a,s)\,.
\end{align}
Then, we have 
\begin{align*}
 &\sum_{\substack{M\subseteq[n]:\\\abs{M}\leq p_i+p_k-1}}\Var\bigl(U_M(i,k)\bigr)\\
 &=
 \sum_{a=1}^{M(p_i,p_k)} \,\,\sum_{s=0}^{p_k\wedge(p_i+p_k-a)\wedge(m_{n,k}-m_{n,i})}\sum_{\substack{M\subseteq[m_{n,k}]:\\ |M|=p_i+p_k-a\,,\\ |M\cap\{m_{n,i}+1,\dotsc, m_{n,k}\}|= s}}\Var\bigl(U_M(i,k)\bigr)\\
 &\leq \sum_{a=1}^{M(p_i,p_k)} \sum_{s=0}^{p_k\wedge(p_i+p_k-a)\wedge(m_{n,k}-m_{n,i})}\sum_{\substack{M\subseteq[m_{n,k}]:\\ |M|=p_i+p_k-a\,,\\ |M\cap\{m_{n,i}+1,\dotsc, m_{n,k}\}|= s}}b_{i,k}^2(a,s)\\
 &=\sum_{a=1}^{M(p_i,p_k)} \sum_{s=0}^{p_k\wedge(p_i+p_k-a)\wedge(m_{n,k}-m_{n,i})}\binom{m_{n,k}-m_{n,i}}{s}\binom{m_{n,i}}{p_i+p_k-a-s}b_{i,k}^2(a,s)\\
 &\leq \sum_{a=1}^{M(p_i,p_k)} \binom{m_{n,k}}{p_i+p_k-a}\sum_{s=0}^{p_k\wedge(p_i+p_k-a)\wedge(m_{n,k}-m_{n,i})}b_{i,k}^2(a,s)\,.
\end{align*}
Now, writing 
\begin{equation*}
 K(p_i,p_k,a,r):=\binom{p_i+p_k-a}{p_i-r,p_k-r,2r-a}
\end{equation*}
and using the inequality \eqref{varum2}, we obtain
\begin{align*}
 &\biggl(\sum_{\substack{M\subseteq[n]:\\\abs{M}\leq p_i+p_k-1}}\Var\bigl(U_M(i,k)\bigr)\biggr)^{1/2}\\
 &\leq \sum_{a=1}^{M(p_i,p_k)} \sqrt{\binom{m_{n,k}}{p_i+p_k-a}}
 \sum_{s=0}^{p_k\wedge(p_i+p_k-a)\wedge(m_{n,k}-m_{n,i})}\sum_{r=\lceil\frac{a}{2}\rceil}^{a\wedge p_i\wedge (p_k-s)}\binom{m_{n,i}-p_i-p_k+a+s}{a-r}\\
 &\hspace{4cm}\cdot K(p_i,p_k,a,r)\norm{\phi^{(i)}\star_r^{a-r}\phi^{(k)}}_{L^2(\mu^{p_i+p_k-a})}\\
 &\leq \frac{p_k+1}{\sqrt{\binom{m_{n,i}}{p_i}} \sqrt{\binom{m_{n,k}}{p_k}}} \sum_{a=1}^{M(p_i,p_k)} m_{n,k}^{(p_i+p_k-a)/2}\\
 &\hspace{5cm}\cdot\sum_{r=\lceil\frac{a}{2}\rceil}^{a\wedge p_i\wedge (p_k-s)} m_{n,i}^{a-r} K(p_i,p_k,a,r)
 \norm{\psi^{(i)}\star_r^{a-r}\psi^{(k)}}_{L^2(\mu^{p_i+p_k-a})}\\
 &\leq C(p_i,p_k)\sum_{a=1}^{M(p_i,p_k)} m_{n,k}^{(p_i-a)/2}\sum_{r=\lceil\frac{a}{2}\rceil}^{a\wedge p_i\wedge p_k} m_{n,i}^{a-r-p_i/2} \norm{\psi^{(i)}\star_r^{a-r}\psi^{(k)}}_{L^2(\mu^{p_i+p_k-a})}\\
&{\leq C(p_i,p_k)\sum_{a=1}^{M(p_i,p_k)} \Bigl(\frac{m_{n,k}}{m_{n,i}}\Bigr)^{p_i/2} m_{n,k}^{a/2-r}\sum_{r=\lceil\frac{a}{2}\rceil}^{a\wedge p_i\wedge p_k}  \norm{\psi^{(i)}\star_r^{a-r}\psi^{(k)}}_{L^2(\mu^{p_i+p_k-a})}
\,,}
\end{align*}
{where the finite constant $C(p_i,p_k)$ only depends on $p_i$ and $p_k$, as $m_{n,i}\leq m_{n,k}$. Since we also have that $c_i n\leq m_{n,i} \leq m_{n,k}\leq n$ we can further bound 
\begin{align*}
&\biggl(\sum_{\substack{M\subseteq[n]:\\\abs{M}\leq p_i+p_k-1}}\Var\bigl(U_M(i,k)\bigr)\biggr)^{1/2}\\
&\leq c_i^{-p_i/2} C(p_i,p_k)\sum_{a=1}^{M(p_i,p_k)} \sum_{r=\lceil\frac{a}{2}\rceil}^{a\wedge p_i\wedge p_k} 
n^{a/2-r} \norm{\psi^{(i)}\star_r^{a-r}\psi^{(k)}}_{L^2(\mu^{p_i+p_k-a})},
\end{align*}
and the desired conclusion follows immediately.}
\end{proof}

\begin{remark}\label{fidirem1}{\rm
 Note that the conditions {\normalfont (i)-(iii)} in Proposition \ref{fidiprop} are the same as those we would obtain in the case $m_{n,i}=n$ for $i=1,\dotsc,d$. In particular they make sure that the vector 
 \begin{equation*}
  V:=\bigl(J_{p_1}^{(n)}(\phi^{(1)}),\dotsc,J_{p_d}^{(n)}(\phi^{(d)})\bigr)^T
 \end{equation*}
converges in distribution to $N_d(0,\Gamma)$, whenever $\Gamma:=\lim_{n\to\infty}\E\bigl[V V^T\bigr]$ exists.}
\end{remark}

\begin{remark}\label{r:lolo} {\rm Let the integers $\{m_{n,i}\}$ and kernels $\{\varphi^{(i,n)}\}$ be defined as above. For each $i=1,...,d$ and $n\geq 1$ consider sets of pairs of integers of the type 
$$
A(i,n) \subset \{(k_1, ...,k_{p_i}) : 1\leq k_1 <...< k_{p_i}  \leq n\},
$$
and assume that, as $n\to \infty$ and for every $i=1,...,d$, $|A(i,n) | \sim \binom{m_{n,i}}{p_i}$. For every $i=1,...,d$ now set
$$
H(i,n) := \sum_{(k_1, ...,k_{p_i})  \in A(i,n)} \varphi^{(i,n)}(X_{k_1}, ..., X_{k_{p_i}} )\, ,
$$
and write ${\bf K}_n$, $n\geq 1$, to denote the covariance matrix of the vector \\$H_n = (H(1,n),....,H(d,n))$. Then, the proof of Proposition \ref{fidiprop} can be straightforwardly adapted to show that, if ${\bf K}_n$ converges to a positive definite matrix ${\bf K}_\star$ and Conditions (i)--(iii) in Proposition \ref{fidiprop} are veirified, then $H_n$ converges in distribution to $Z\sim N_d(0, {\bf K}_\star)$. Such a conclusion plays a role in the proof of Theorem \ref{t:maincp}.

}
\end{remark}

\subsubsection{F.d.d. convergence for general symmetric $U$-processes}
Let $\psi:E^p\rightarrow\R$ be a symmetric kernel of order $p$ which is not necessarily degenerate and which might explicitly depend on the sample size $n$. Fix time points $0\leq t_1<\ldots<t_m\leq1$.
Then, for each $j=1,\dotsc,m$, the random variable 
$F_j:=J_p^{(\floor{nt_j})}(\psi)$ has the Hoeffding decomposition 
\begin{equation*}
 F_j=\E[F_j]+\sum_{k=1}^p \binom{\floor{nt_j}-k}{p-k} J_k^{(\floor{nt_j})}(\psi_k)\,,
\end{equation*}
where the symmetric and degenerate kernels $\psi_k:E^k\rightarrow\R$ of order $k$ are given by 
\begin{align*}
\psi_k(x_1,\dotsc,x_k)&=\sum_{l=0}^k(-1)^{k-l}\sum_{1\leq i_1<\dotsc<i_l\leq k} g_l(x_{i_1},\dotsc,x_{i_l})
 \end{align*}
and the symmetric functions $g_l:E^l\rightarrow\R$ are defined by 
\begin{equation*}
 g_l(y_1,\dotsc,y_l):=\E\bigl[\psi(y_1,\dotsc,y_l,X_1,\dotsc,X_{p-l})\bigr]\,.
\end{equation*}
Without loss of generality, we can thus assume that $t_1>0$ and also that $0<\|\psi\|_{L^2(\mu^{p})}<+\infty$ which implies that 
\[0<\sigma_n^2:=\Var(J_p^{(n)}(\psi))=\sum_{k=1}^p\binom{n-k}{p-k}^2\binom{n}{k}\|\psi_k\|_{L^2(\mu^{k})}^2<+\infty\]
for all $n\geq p$. We will further write
\begin{equation*}
 W_j:=\frac{F_j-\E[F_j]}{\sigma_n}
\end{equation*}
for $j=1,\dotsc,m$. Our goal is to use Proposition \ref{fidiprop} in order to find conditions ensuring that the vector $W:=W_n:=(W_1,\dotsc,W_m)^T$ converges to some multivariate normal distribution 
$N_m(0,\mathbf{D})$, which requires in particular that the limit $\mathbf{D}:=\lim_{n\to\infty}\E[WW^T]\in\R^{m\times m}$ exists. Let us write $d:=mp$ and for $i=1,\dotsc,d$ let $i=k_ip+s_i$, where $k_i\in\{0,1,\dotsc,m-1\}$ and 
$s_i\in\{1,\dotsc,p\}$ as well as $m_{n,i}:=\floor{n t_{k_i+1}}$. Moreover, similarly as in \cite[Section 5]{DP18}, we define 
\begin{align*}
 \phi^{(i)}&:=\phi^{(n,i)}:=\frac{\binom{m_{n,i}-s_i}{p-s_i}\psi_{s_i}}{\sigma_n}\quad\text{and}\\
  \psi^{(i)}&:=\psi^{(n,i)}:=\sqrt{\binom{m_{n,i}}{s_i}}\phi^{(i)}\,.
\end{align*}
With this notation at hand, we define the random vector $Y:=(Y_1,\dotsc,Y_d)^T$, where  $Y_i:=J^{(m_{n,i})}_{s_i}(\phi^{(i)})$, $1\leq i\leq d$.
In this way, our notation is fitted to the framework of Subsection \ref{qmclt}. We are going to reformulate the conditions from Proposition \ref{fidiprop}. 
First note that $\E[Y_iY_j]=0$ whenever $s_i\not=s_j$, due to the degeneracy of the involved kernels. On the other hand, if $s_i=s_j=s$, then 
\begin{equation*}
 \E[Y_i Y_j]=
 \frac{\binom{m_{n,i}\wedge m_{n,j}}{s}\binom{m_{n,i}-s}{p-s}\binom{m_{n,j}-s}{p-s}}{\sigma_n^2}\norm{\psi_{s}}^2_{L^2(\mu^{s})}\,.
\end{equation*}
Since $m_{n,i}=\floor{nt_i}$ for $1\leq i\leq d$, the covariance matrix of $Y$ thus converges to some limit $\Gamma\in\R^{d\times d}$ if and only if the real limit
\begin{align}\label{fidi1}
 \lim_{n\to\infty}\frac{n^{2p-s}}{\sigma_n^2}\norm{\psi_{s}}^2_{L^2(\mu^{s})}
\end{align}
exists for $s=1,\dotsc,p$. 
Moreover, for $1\leq i,k\leq d$ we have 
\begin{align*}
 &n^{a/2-r} \norm{\psi^{(i,n)}\star_r^{a-r}\psi^{(k,n)}}_{L^2(\mu^{s_i+s_k-a})}\\
 &=n^{a/2-r}\frac{ \sqrt{\binom{m_{n,i}}{s_i}}\sqrt{\binom{m_{n,k}}{s_k}}\binom{m_{n,i}-s_i}{p-s_i}\binom{m_{n,k}-s_k}{p-s_k}  }{\sigma_n^2}
 \norm{\psi_{s_i}\star_r^{a-r}\psi_{s_k}}_{L^2(\mu^{s_i+s_k-a})}\,.
\end{align*}
Thus, $\lim_{n\to\infty}n^{a/2-r} \norm{\psi^{(i,n)}\star_r^{a-r}\psi^{(k,n)}}_{L^2(\mu^{s_i+s_k-a})}=0$ if and only if 
\begin{equation*}
 \lim_{n\to\infty}\frac{n^{2p+\frac{a-s_i-s_k}{2}-r}}{\sigma_n^2}\,\norm{\psi_{s_i}\star_r^{a-r}\psi_{s_k}}_{L^2(\mu^{s_i+s_k-a})}=0\,.
\end{equation*}
Further, for $i=1,\dotsc,d$, we have 
\begin{align*}
\frac{\norm{\psi^{(i,n)}}^3_{L^2(\mu^{s_i})}}{\sqrt{n}}&=\frac{\binom{m_{n,i}}{s_i}^{3/2}\binom{m_{n,i}-s_i}{p-s_i}^{3/2}}{\sqrt{n}\sigma_n^3}\norm{\psi_{s_i}}^3_{L^2(\mu^{s_i})}\,.
\end{align*}
Hence, $\lim_{n\to\infty}\frac{\norm{\psi^{(i,n)}}^3_{L^2(\mu^{s_i})}}{\sqrt{n}}=0$ if and only if 
\begin{align*}
 \lim_{n\to\infty}\frac{n^{3p-\frac{3s_i+1}{2}}}{\sigma_n^3}\norm{\psi_{s_i}}^3_{L^2(\mu^{s_i})}=\biggl(\lim_{n\to\infty}n^{-1/3}
\frac{n^{2p-s_i}}{\sigma_n^2}\norm{\psi_{s_i}}^2_{L^2(\mu^{s_i})}\biggr)^{3/2}=0\,,
\end{align*}
which is implied by \eqref{fidi1}.
Taking into account that, for $j=1,\dotsc,m$,
\begin{align*}
 1&=\Var\bigl(J_p^{(n)}(\psi/\sigma_n)\bigr)\geq\Var(W_j)=\sum_{i=(j-1)p+1}^{jp} \norm{\psi^{(i,n)}}^2_{L^2(\mu^{s_i})}
\end{align*}
and that $(W_1,\dotsc,W_m)^T$ is obtained from $(Y_1,\dotsc,Y_d)^T$ by applying a linear functional, from Proposition \ref{fidiprop} we thus deduce the following result. Note that we also apply the reindexing $l:=a-r$.

\begin{theorem}\label{fiditheo}
With the above notation and definitions, the vector $(W_1,\dotsc,W_m)^T$ converges, as $n\to\infty$, to a multivariate normal distribution, whenever the following conditions hold for all $1\leq v\leq u\leq p$:
\begin{enumerate}[{\normalfont (a)}]
 \item The real limit $\lim_{n\to\infty}\frac{n^{2p-v}}{\sigma_n^2}\norm{\psi_{v}}^2_{L^2(\mu^{v})}$ does exist and
  \item $\displaystyle \lim_{n\to\infty}\frac{n^{2p-\frac{u+v+r-l}{2}}}{\sigma_n^2}\,\norm{\psi_{v}\star_r^{l}\psi_{u}}_{L^2(\mu^{v+u-r-l})}=0$ for all pairs $(l,r)$ of integers such that 
 $1\leq r\leq v$ and $0\leq l\leq r\wedge(u+v-r-1)$.
\end{enumerate}
\end{theorem}

Due to the complicated expressions of the kernels $\psi_s$, the following result, which is a rectified version of Lemma 5.1 of \cite{DP18}, is often useful for bounding the contraction norms appearing in Theorem \ref{fiditheo}. 
Recall the definition of the set $Q(i,k,r,l)$ defined before Theorem \ref{maintheo}.

\begin{lemma}\label{genulemma}
For positive integers $1\leq r, i, k \leq p$ and $0\leq l\leq p$ such that $0\leq l\leq r\leq i\wedge k$ there exists a constant $K(i,k,r,l)\in(0,\infty)$ which only depends on $i,k,r$ and $l$ such that 
\begin{align*}
 \|\psi_{i}\star_r^l\psi_{k}\|_{L^2(\mu^{ i+k-r-l})}&\leq K(i,k,r,l) \max_{(j,m,a,b)\in Q(i,k,r,l)}\|g_j\star_a^b g_m\|_{L^2(\mu^{j+m-a-b})}\,.
 \end{align*}
\end{lemma}  
\begin{proof}[\bf Proof of Lemma \ref{genulemma}]
 Recall that we have 
\begin{align}\label{cb9}
 &\bigl(\psi_{i}\star_r^l\psi_{k}\bigr)(x_1,\dotsc,x_{k+i-2r},y_{l+1},\dotsc,y_r)\notag\\
 &=\int_{E^l}\psi_i(y_1,\dotsc,y_r,x_1,\dotsc,x_{i-r})\notag\\
 &\hspace{2cm}\cdot \psi_k(y_1,\dotsc,y_r,x_{i-r+1},\dotsc,x_{k+i-2r})d\mu^{\otimes l}(y_{1},\dots,y_l)\,.
\end{align}
Recalling also the expression \eqref{defpsis} of the kernels $\psi_i$ and $\psi_k$, respectively, and taking into account that $\mu$ is a probability measure as well as that, for $k\geq s$, the functions $g_k$ from \eqref{gk} satsify 
\begin{equation}\label{itg}
\int_{E^{k-s}}g_k(x_1,\dotsc,x_k)d\mu^{\otimes{k-s}}(x_{s+1},\dotsc,x_k)=g_s(x_1,\dotsc,x_s)
\end{equation}
by virtue of Fubini's theorem, we see that $ \bigl(\psi_{i}\star_r^l\psi_{k}\bigr)(x_1,\dotsc,x_{k+i-2r},y_{l+1},\dotsc,y_{r})$ is a linear combination, with coefficients only depending on $i,k,r$ and $l$ but not on $n$, of expressions of the form 
\begin{align*}
&G_{(a,b,j,m)}(x_{i_1},\dotsc,x_{i_{j-b-d}},y_{q_1},\dotsc, y_{q_c},x_{k_1},\dotsc, x_{k_{m-b-e}})\\
& :=\int_{E^t} g_j(u_1,\dotsc,u_b,y_{m_1},\dotsc,y_{m_d},  x_{i_1},\dotsc,x_{i_{j-b-d}}) \\
&\hspace{2cm}\cdot g_m(u_1,\dotsc,u_b,y_{n_1},\dotsc,y_{n_e},x_{k_1},\dotsc, x_{k_{m-b-e}})d\mu^{\otimes b}(u_1,\dots,u_b)
\\&=\bigl(g_j\star_a^b g_m\bigr)(x_{i_1},\dotsc,x_{i_{j-b-d}},y_{q_1},\dotsc, y_{q_c},x_{k_1},\dotsc, x_{k_{m-b-e}})
\,,
\end{align*}
where $0\leq j\leq i$, $0\leq m\leq k$, $0\leq b\leq l$, $0\leq b\leq a\leq r$, $1\leq i_1<\ldots<i_{j-b-d}\leq i-r$,\\
$i-r+1\leq k_1<\ldots<k_{l-b-e}\leq k+i-2r$ such that, in particular, the sets $\{i_1,\dotsc,i_{j-t-a}\}$ and $\{k_1,\dotsc,k_{m-t-b}\}$ are disjoint. 
Furthermore, we have $l+1\leq m_1<\ldots<m_d\leq r$, $l+1\leq n_1<\ldots<n_e\leq r$, $l+1\leq q_1<\ldots<q_c\leq r$ such that $\{q_1,\dotsc,q_c\}=\{m_1,\dotsc,m_d\}\cup\{n_1,\dotsc,n_e\}$, $d\leq j-b$, $e\leq m-b$ and 
$a:=b+|\{m_1,\dotsc,m_d\}\cap\{n_1,\dotsc,n_e\}|\leq b+c$. Note that $c\leq r-l$ and, hence, also $a-b\leq c\leq r-l$ as well as $a\leq(b+d)\wedge(b+e)\leq j\wedge m$.
Moreover, the number $j+m-a-b$ of arguments of the function $g_j\star_a^b g_m$ is at most as large as the number 
$i+k-r-l$ of arguments of the function $\psi_{i}\star_r^l\psi_{k}$. Finally, if $j=m=p$, then $i=k=p$ and $g_j=g_m=\psi$. This also implies that $b=l$ and $a=r$. Hence, we conclude that $(j,m,a,b)\in Q(i,k,r,l)$. 

Now, using the fact that $\mu$ is a probability measure, we obtain that 
\begin{align}\label{gl1}
 &\int_{E^{k+i-r-l}}G^2_{(a,b,j,m)}(x_{i_1},\dotsc,x_{i_{j-b-d}},y_{q_1},\dotsc, y_{q_c},x_{k_1},\dotsc, x_{k_{m-b-e}})\notag\\
 &\hspace{3cm} d\mu^{\otimes i+k-l-r}(x_1,\dotsc,x_{k+i-2r},y_{l+1},\dotsc,y_r)\notag\\
 &=\int_{E^{j+m-a-b}} \bigl(g_j\star_a^b g_m\bigr)^2d \mu^{\otimes j+m-a-b}=\|g_j\star_a^bg_m\|^2_{L^2(\mu^{\otimes j+m-a-b})}\,.
\end{align}
Since $\psi_{i}\star_r^l\psi_{k}$ is a finite linear combination with coefficients depending uniquely on $i,k,r$ and $l$ of the $G_{(a,b,j,m)}$, the claim thus follows from \eqref{gl1} and 
Minkowski's inequality.
\end{proof}

The next result is a direct consequence of Theorem \ref{fiditheo}, of Lemma \ref{genulemma} and of the fact that, for all $1\leq v\leq p$, we have 
\begin{align*}
 \norm{\psi_{v}}^2_{L^2(\mu^{v})}=\sum_{l=1}^v\binom{v}{l}(-1)^{v-l}\norm{\hat{g}_{l}}^2_{L^2(\mu^{l})}\,,
\end{align*}
where $\hat{g}_l:=g_l-\E[\psi(X_1,\dotsc,X_p)]\,, 1\leq l\leq p$ (see e.g. \cite[Theorem 4.3]{Vit92}). 

\begin{corollary}\label{fidicor}
 With the above notation and definitions, the vector $(W_1,\dotsc,W_m)^T$ converges, as $n\to\infty$, to a multivariate normal distribution if the following conditions hold for all $1\leq v\leq u\leq p$:
\begin{enumerate}[{\normalfont (a)}]
 \item the real limit $\lim_{n\to\infty}\frac{n^{2p-v}}{\sigma_n^2}\bigl(\norm{g_{v}}^2_{L^2(\mu^{v})}-(\E[\psi(X_1,\dotsc,X_p)])^2\bigr)$ exists, and
 \item $\displaystyle \lim_{n\to\infty}\frac{n^{2p-\frac{u+v+r-l}{2}}}{\sigma_n^2}\,\norm{g_{j}\star_i^{k}g_{m}}_{L^2(\mu^{j+m-i-k})}=0$
for all pairs $(l,r)$ and quadruples $(j,m,i,k)$ of integers such that\\ 
 $1\leq r\leq v$, $0\leq l\leq r\wedge(u+v-r-1)$ and $(j,m,i,k)\in Q(v,u,r,l)$.
 \end{enumerate}
\end{corollary}

\begin{remark}\label{fidirem}{\rm
\begin{enumerate}[{\normalfont (i)}]
\item The conditions given in Theorem \ref{fiditheo} and in Corollary \ref{fidicor} do not depend on the finite sequence $0<t_1<\ldots<t_m\leq1$ used to define the vector $(W_1,\dotsc,W_m)^T$. Hence, both statements yield sufficient conditions 
for f.d.d. convergence of the sequence $(W_n(t))_{t\in[0,1]}$ of processes.
 \item  Note that the respective conditions {\normalfont (b)} of Theorem \ref{fiditheo} and of Corollary \ref{fidicor} are the same as those we would get to obtain the asymptotic normality of the single $U$-statistic $J_p^{(n)}(\psi/\sigma_n)$ 
 (see \cite[Section 5]{DP18}). 
 \item It is easy to see from the following computation that the respective conditions {\normalfont (a)} in Theorem \ref{fiditheo} and in Corollary \ref{fidicor} imply that the covariance function of $W_n$ converges 
pointwise to an explicit limit. Such a condition was not necessary in the univariate case dealt with in \cite[Section 5]{DP18} since there the $U$-statistic could simply be 
 normalized to have variance one. For $s,t\in[0,1]$ we have 
\begin{align*}
 \Cov\bigl(W_n(s),W_n(t)\bigr)&=\E\bigl[W_n(s)W_n(t)\bigr]\\
 &=\sigma_n^{-2}\sum_{k,l=1}^p\binom{\floor{ns}-k}{p-k}\binom{\floor{nt}-l}{p-l}\E\bigl[J_k^{(\floor{ns})}(\psi_k)J_l^{(\floor{nt})}(\psi_l)\bigr]\\
 &=\sigma_n^{-2}\sum_{k=1}^p\binom{\floor{ns}-k}{p-k}\binom{\floor{nt}-k}{p-k}\binom{\floor{ns}\wedge\floor{nt}}{k}\|\psi_k\|_{L^2(\mu^k)}^2\\
 &\sim\sigma_n^{-2}\sum_{k=1}^p \frac{\floor{ns}^{p-k}}{(p-k)!}\frac{\floor{nt}^{p-k}}{(p-k)!}\frac{\bigl(\floor{ns}\wedge\floor{nt}\bigr)^k}{k!}\|\psi_k\|_{L^2(\mu^k)}^2\\
 &\sim\sum_{k=1}^p  \frac{(s\wedge t)^p (s\vee t)^{p-k}}{k!(p-k)!(p-k)!}\frac{ n^{2p-k} }{\sigma_n^{2}} \|\psi_k\|_{L^2(\mu^k)}^2\,.
\end{align*}
Note further that, using \eqref{sigma}, we can conclude that for fixed $s,t\in[0,1]$, the sequence $\Cov\bigl(W_n(s),W_n(t)\bigr)$, $n\in\N$, is always bounded.
\end{enumerate}
}
\end{remark}

\subsection{Criteria for tightness}
We are going to establish tightness using Lemma \ref{lemma_tightness} and, as a result, obtain the following theorem:

\begin{theorem}[\bf Tightness of general $U$-processes]\label{gentightness}
 Let $p\in\N$ and suppose that $\psi=\psi(n)\in L^4(\mu^p)$, $n\geq p$, is a sequence of symmetric kernels. For $t\in[0,1]$ let $U(t):=U_n(t):=J_p^{(\floor{nt})}(\psi)$ and define $W(t) = W_n(t)$ by \eqref{genW}, where 
 $\sigma_n^2:=\Var(U_n(1))=\Var(J_p^{(n)}(\psi))$. 
Suppose that there is an a.s. continuous Gaussian process $Z=(Z(t))_{t\in[0,1]}$ such that the finite-dimensional distributions of $W_n$, $n\in \N$, converge to those of $Z$.
Consider the following conditions:
 \begin{enumerate}[{\normalfont (i)}]
  \item There is an $\epsilon>0$ such that for all $1\leq r\leq p$ and all $0\leq l\leq r-1$, the sequence 
 \[\frac{n^{2p-r-\frac{r-l}{2}+\epsilon}}{\sigma_n^2}\bigl\|\psi_{r}\star_r^{l}\psi_{r}\bigr\|_{L^2(\mu^{r-l})}\]
 is bounded.
 \item There is an $\epsilon>0$ such that for all $1\leq r\leq p$, all $0\leq l\leq r-1$ and for all quadruples $(j,m,a,b)\in Q(r,r,r,l)$ the sequence 
\[\frac{n^{2p-r-\frac{r-l}{2}+\epsilon}}{\sigma_n^2}\,\norm{g_{j}\star_a^{b}g_{m}}_{L^2(\mu^{j+m-a-b})}\]
is bounded.
\end{enumerate}
Then, one has that {\rm (ii)} $ \Rightarrow$ {\rm (i)}, and also that {\rm (i)} is sufficient in order for the sequence $W_n$, $n\in \N$, to be tight in $D[0,1]$.
\end{theorem}

The proof of Theorem \ref{gentightness} is detailed in the forthcoming Subsections \ref{ss:tgd} and \ref{ss:tgg}.
There, we are however not going to establish \eqref{Bil} directly, but will show that there is a finite constant $C_1>0$ such that, under the assumptions of Theorem \ref{gentightness}, for all $n\in\N$ and all 
$0\leq s\leq t\leq1$ we have the inequality
\begin{align}\label{Bil2}
 \E\babs{W_n(t)-W_n(s)}^4&\leq  C_1\biggl(\frac{\floor{nt}-\floor{ns}}{n}\biggr)^{1+\epsilon}\,,
\end{align}
where $\epsilon$ is the same as in the statement of Theorem \ref{gentightness}. This is sufficient by Lemma \ref{lemma_tightness}.
\subsubsection{Proof of Theorem \ref{gentightness}, I: degenerate kernels}\label{ss:tgd} Throughout the present and subsequent section, we can assume without loss of generality that $\epsilon \in (0,1]$
Let us first assume that $W_n$ is a $U$-process of order $p$ based on a degenerate kernel $\phi$, i.e. for $0\leq t\leq1$ we have
\begin{equation*}
 W(t):=W_n(t):=J_p^{(\floor{nt})}(\phi)\,.
\end{equation*}
For $0\leq s\leq t\leq 1$ let $I(n,s,t):=\{\floor{ns}+1,\dotsc\floor{nt}\}$ and for $J\in\D_p(\floor{nt})$ write 
\begin{equation*}
 V_J:={\bf 1}_{\{J\cap I(n,s,t)\not=\emptyset\}}\phi(X_i,i\in J)\,.
\end{equation*}
Then we have that 
\begin{align*}
 W_n(t)-W_n(s)=\sum_{J\in\D_p(\floor{nt})} V_J
 \end{align*}
is a degenerate (non-symmetric) $U$-statistic of order $p$, based on $X_1,\dotsc,X_{\floor{nt}}$. 
In \cite[Theorem 2]{ISH}, the following bounds are given:
\begin{align}\label{is}
 &A\max_{k=0,\dotsc,p}\,\max_{1\leq j_1<\ldots<j_k\leq p}\,\sum_{1\leq i_{j_1}<\ldots<i_{j_k}\leq \floor{nt}}
 \E\Biggl[\biggl(\hspace{-2mm}\sum_{\substack{i_s: 1\leq s\leq p,\\ s\notin\{j_1,\dotsc,j_k\},\\ 1\leq i_1<\ldots<i_p\leq\floor{nt}}}\hspace{-1mm}\E\bigl[V^2_{\{i_1,\dotsc,i_p\}}\,\bigl|\,X_{i_{j_1}},\dotsc,X_{i_{j_1}}\bigr]\biggr)^2\Biggr]\notag\\
 &\leq \E\babs{W_n(t)-W_n(s)}^4\\
 &\leq B \max_{k=0,\dotsc,p}\,\max_{1\leq j_1<\ldots<j_k\leq p}\,\sum_{1\leq i_{j_1}<\ldots<i_{j_k}\leq \floor{nt}}
 \E\Biggl[\biggl(\hspace{-4mm}\sum_{\substack{i_s: 1\leq s\leq p,\\ s\notin\{j_1,\dotsc,j_k\},\\ 1\leq i_1<\ldots<i_p\leq\floor{nt}}}\hspace{-4mm}\E\bigl[V^2_{\{i_1,\dotsc,i_p\}}\,\bigl|\,X_{i_{j_1}},\dotsc,X_{i_{j_1}}\bigr]\biggr)^2\Biggr]\notag,
\end{align}
where $A$ and $B$ are finite constants which only depend on $p$. We will now make use of the upper bound in \eqref{is}. Note that for each $k=0,\dotsc,p$, each labelling $1\leq j_1<\ldots<j_k\leq p$ and 
all choices $1\leq i_{j_1}<\ldots<i_{j_k}\leq \floor{nt}$ 
we have that 
\begin{align}
 &\E\Biggl[\biggl(\sum_{\substack{i_s: 1\leq s\leq p,\\ s\notin\{j_1,\dotsc,j_k\},\\ 1\leq i_1<\ldots<i_p\leq\floor{nt}}}\E\bigl[V^2_{\{i_1,\dotsc,i_p\}}\,\bigl|\,X_{i_{j_1}},\dotsc,X_{i_{j_k}}\bigr]\biggr)^2\Biggr]\notag\nonumber\\
&\leq \E\Biggl[\biggl(\sum_{\substack{L\in\D_{p-k}(\floor{nt}):\\L\cap \{i_{j_1},\dotsc,i_{j_k}\}=\emptyset}}\E\bigl[V_{\{i_{j_1},\dotsc,i_{j_k}\} \cup L}^2\,\bigl|\,X_{i_{j_1}},\dotsc,X_{i_{j_k}}\bigr]\biggr)^2\Biggr]\label{is1}
 \end{align}
and, that for each set $L\in\D_{p-k}(\floor{nt})$ such that $L\cap \{i_{j_1},\dotsc,i_{j_k}\}=\emptyset$ and \\$\bigl(L\cup \{i_{j_1},\dotsc,i_{j_k}\}\bigr)\cap I(n,s,t)\not=\emptyset$ we have 
\begin{align*}
 \E\bigl[V_{\{i_{j_1},\dotsc,i_{j_k}\} \cup L}^2\,\bigl|\,X_{i_{j_1}},\dotsc,X_{i_{j_1}}\bigr]=\bigl(\phi\star_p^{p-k}\phi\bigr)\bigl(X_{i_{j_1}},\dotsc,X_{i_{j_k}}\bigr)\,.
\end{align*}
Thus, we can further bound
\begin{align}
&\E\babs{W_n(t)-W_n(s)}^4\leq B \max_{k=0,\dotsc,p} \sum_{\substack{K\in\D_k(\floor{nt})}}\E\Biggl[\biggl(\sum_{\substack{L\in\D_{p-k}(\floor{nt}):\\L\cap K=\emptyset,\\ (L\cup K)\cap I(n,s,t)\not=\emptyset}}\E\bigl[V_{K\cup L}^2\,\bigl|\,\F_K\bigr]\biggr)^2\Biggr]\label{nub}\\
&=B \max_{k=0,\dotsc,p}\Biggl(\sum_{\substack{K\in\D_k(\floor{nt}):\\ K\cap I(n,s,t)=\emptyset}}\E\Biggl[\biggl(\binom{\floor{nt}-k}{p-k}-\binom{\floor{ns}-k}{p-k}\biggr)^2\bigl(\phi\star_p^{p-k}\phi\bigr)^2(X_i,i\in K)\biggr)\Biggr]\notag\\
&\hspace{2cm} +\sum_{\substack{K\in\D_k(\floor{nt}):\\ K\cap I(n,s,t)\not=\emptyset}}\E\Biggl[\binom{\floor{nt}-k}{p-k}^2\bigl(\phi\star_p^{p-k}\phi\bigr)^2(X_i,i\in K)\biggr)\Biggr]\Biggr)\notag\\
&\leq B \max_{k=0,\dotsc,p-1}\binom{\floor{ns}}{k}\biggl(\binom{\floor{nt}-k}{p-k}-\binom{\floor{ns}-k}{p-k}\biggr)^2  \E\Biggl[\bigl(\phi\star_p^{p-k}\phi\bigr)^2(X_i,i\in K)\Biggr]\notag\\
&\; +B \max_{k=1,\dotsc,p}\biggl(\binom{\floor{nt}}{k}-\binom{\floor{ns}}{k}\biggr)\binom{\floor{nt}-k}{p-k}^2  \E\Biggl[\bigl(\phi\star_p^{p-k}\phi\bigr)^2(X_i,i\in K)\Biggr]\notag\\
&=B \max_{k=0,\dotsc,p-1}\binom{n-k}{p-k}^2 \binom{\floor{ns}}{k}\binom{n-k}{p-k}^{-2} \biggl(\binom{\floor{nt}-k}{p-k}-\binom{\floor{ns}-k}{p-k}\biggr)^2\notag  \\
&\hspace{10cm}\bigl\|\phi\star_p^{p-k}\phi\bigr\|_{L^2(\mu^k)}^2\notag\\
&\; +B \max_{k=1,\dotsc,p}\binom{n}{k} \binom{n}{k}^{-1} \biggl(\binom{\floor{nt}}{k}-\binom{\floor{ns}}{k}\biggr)\binom{\floor{nt}-k}{p-k}^2  \bigl\|\phi\star_p^{p-k}\phi\bigr\|_{L^2(\mu^k)}^2\notag\\
&\lesssim  \max_{k=0,\dotsc,p-1}\biggl(\frac{\floor{nt}-\floor{ns}}{n}\biggr)^2 n^{2p-k} \bigl\|\phi\star_p^{p-k}\phi\bigr\|_{L^2(\mu^k)}^2\notag\\
&\; + \max_{k=1,\dotsc,p}\frac{\floor{nt}-\floor{ns}}{n} n^{2p-k}  \bigl\|\phi\star_p^{p-k}\phi\bigr\|_{L^2(\mu^k)}^2\notag\\
&\lesssim  \biggl(\frac{\floor{nt}-\floor{ns}}{n}\biggr)^2 n^{2p} \norm{\phi}_{L^2(\mu^p)}^4 + \frac{\floor{nt}-\floor{ns}}{n} \max_{k=1,\dotsc,p} n^{2p-k}  \bigl\|\phi\star_p^{p-k}\phi\bigr\|_{L^2(\mu^k)}^2\notag\\
&\lesssim \biggl(\frac{\floor{nt}-\floor{ns}}{n}\biggr)^{1+\epsilon}\biggl(n^{2p} \norm{\phi}_{L^2(\mu^p)}^4 + \max_{k=1,\dotsc,p} n^{2p-k+\epsilon}  \bigl\|\phi\star_p^{p-k}\phi\bigr\|_{L^2(\mu^k)}^2\biggr)\label{tight1}\,,
\end{align}
for each $\epsilon\in(0,1]$.
Now, if, for some $\epsilon\in(0,1]$, there is a $C_1=C_1(\epsilon)\in(0,\infty)$ such that for all $n\in\N$ we have 
\begin{align}\label{tight2}
 n^{2p} \norm{\phi}_{L^2(\mu^p)}^4 + \max_{k=1,\dotsc,p} n^{2p-k+\epsilon}  \bigl\|\phi\star_p^{p-k}\phi\bigr\|_{L^2(\mu^k)}^2\leq C_1
\end{align}
we conclude from \eqref{tight1} that \eqref{Bil2} is satisfied. This concludes the argument in the case of degenerate kernels.
\begin{remark}
{\rm{Incidentally, one can show that inequality \eqref{nub} also holds in the opposite direction when the constant $B$ appearing there is replaced with a small enough positive constant $C$, which only depends on $p$. Our way of bounding $\mathbb{E}|W_n(t)-W_n(s)|^4$ is therefore optimal with respect to the order in $n$.}}
\end{remark}

\subsubsection{Proof of Theorem \ref{gentightness}, II: general kernels}\label{ss:tgg}

For $t\in[0,1]$ recall the definition \eqref{genW} of $W_n(t)$ and \eqref{sigma} of $\sigma_n^2$.
Then, defining for $1\leq r\leq p$, 
\begin{equation*}
 \phi^{(r)}:=\phi^{(n,r)}:=\frac{\binom{n-r}{p-r} \psi_r}{\sigma_n}\,,
\end{equation*}
the $\phi^{(r)}$ are degenerate kernels and, with the notation $V_r(t):= J_r^{(\floor{nt})}(\phi^{(r)})$, we have 
\begin{equation*}
 W(t)=\sum_{r=1}^p \frac{\binom{\floor{nt}-r}{p-r}}{ \binom{n-r}{p-r}}V_r(t)\,.
\end{equation*}
Since 
\[\lim_{n\to\infty}\frac{\binom{\floor{nt}-r}{p-r}}{ \binom{n-r}{p-r}}=t^r\]
for each $r=1,\dotsc,p$, by an application of Prohorov's theorem, it follows that the sequence $(W_n(t))_{t\in[0,1]}$, $n\in\N$, is tight whenever $(V_r(t))_{t\in[0,1]}$ is tight for every $r=1,\dotsc,p$.

Sufficient conditions for this to hold were given in the previous subsection. Indeed, $(V_r(t))_{t\in[0,1]}$ is tight if for some $\epsilon\in(0,1]$ the sequence
\begin{align}\label{tight4}
 n^{2r} \norm{\phi^{(r)}}_{L^2(\mu^r)}^4 + \max_{k=1,\dotsc,r} n^{2r-k+\epsilon}  \bigl\|\phi^{(r)}\star_r^{r-k}\phi^{(r)}\bigr\|_{L^2(\mu^r)}^2
\end{align}
is bounded from above by a constant. Now, first note that from \eqref{sigma} we see that there is a finite constant $L_r$ such that 
\begin{align}
  n^{2r} \norm{\phi^{(r)}}_{L^2(\mu^r)}^4 &=  n^{2r} \binom{n-r}{p-r}^4 \frac{\norm{\psi_r}_{L^2(\mu^r)}^4}{\sigma_n^4}\leq L_r
\end{align}
for all $n\in\N$. Moreover, for $k=1,\dotsc,r$ we have that 
\begin{align*}
  n^{2r-k+\epsilon}  \bigl\|\phi^{(r)}\star_r^{r-k}\phi^{(r)}\bigr\|_{L^2(\mu^r)}^2&=n^{2r-k+\epsilon}\binom{n-r}{p-r}^4 \frac{\bigl\|\psi_r\star_r^{r-k}\psi_r\bigr\|_{L^2(\mu^r)}^2}{\sigma_n^4}\\
  &\leq D_r n^{4p-2r-k +\epsilon} \frac{\bigl\|\psi_r\star_r^{r-k}\psi_r\bigr\|_{L^2(\mu^r)}^2}{\sigma_n^4}\,,
 \end{align*}
where $D_r$ is a finite constant only depending on $r$.
The proof is concluded by letting $l:=r-k$ and by applying Lemma \ref{genulemma}.\qed

\subsection{Proofs of Theorem \ref{maintheo}, Theorem \ref{maintheo2}, Corollary \ref{maincor}, Theorem \ref{mult1} and Theorem \ref{mult2}}

\begin{proof}[Proof of Theorem \ref{maintheo}]
It is well-known that f.d.d. convergence and tightness together imply weak convergence in $D[0,1]$ (see e.g. \cite[Section 13]{Bil}). F.d.d. convergence and Gaussianity of the limit process $Z$ now follow from Corollary \ref{fidicor}. Moreover, the needed formula for the covariance $\Gamma$ of $Z$ follows from Remark \ref{fidirem} (iii) as
\begin{align*}
& \lim_{n\to\infty}\Cov\bigl(W_n(s),W_n(t)\bigr) \\
 &=\lim_{n\to\infty} \sigma_n^{-2}\sum_{k=1}^p\frac{(s\wedge t)^p  n^{2p-k} (s\vee t)^{p-k}}{k!(p-k)!(p-k)!}\|\psi_k\|_{L^2(\mu^k)}^2\\
 &=\sum_{k=1}^p \frac{(s\wedge t)^p  (s\vee t)^{p-k}   }{k!(p-k)!(p-k)!} \sum_{l=1}^k\binom{k}{l}(-1)^{k-l}\lim_{n\to\infty} \frac{n^{2p-k}}{\sigma_n^2} \biggl(\norm{g_{l}}^2_{L^2(\mu^{l})}-\Bigl(\int_E\psi d\mu\Bigr)^2\biggr)\\
& =\sum_{k=1}^p \frac{(s\wedge t)^p  (s\vee t)^{p-k}   }{k!(p-k)!(p-k)!}  b_k^2 \,.
 \end{align*}
Since this is the same covariance function as that of the process given in \eqref{e:limsum} (with $\alpha_{k,p}^2$ as given in the statement)
we conclude that the limiting process $Z$ may be chosen to have a.s. continuous paths. Now, tightness is implied by Theorem \ref{gentightness}-(ii).
\end{proof}

\begin{proof}[Proof of Theorem \ref{maintheo2}] 
The proof is similar to that of Theorem \ref{maintheo} and follows from Theorems \ref{fiditheo} and Theorem \ref{gentightness}-(i).
\end{proof}

\begin{proof}[Proof of Corollary \ref{maincor}] We directly use Theorem \ref{maintheo}.
In this case, we have $g_k=0$ for all $0\leq k\leq p-1$ and $g_p=\psi$. Moreover, 
\[\sigma_n^2=\binom{n}{p} \norm{\psi}_{L^2(\mu^p)}^2\sim \frac{n^p}{p!}\norm{\psi}_{L^2(\mu^p)}^2 \,.\]
Hence, we have $b_k^2=0$ for all $1\leq k\leq p-1$ and $b_p^2=p!$. From this, we obtain that $\Gamma(s,t)=(s\wedge t)^p$, implying the result.
\end{proof}

\begin{proof}[Outline of the proof of Theorem \ref{mult2}]
Finite-dimensional convergence can again be proved by means of Proposition \ref{fidiprop} with the dimension $d$ appearing there given by 
$m\sum_{j=1}^d p_j$, where $m$ is the number of points $0\leq t_1<t_2<\ldots<t_m\leq 1$ under consideration. 
The computation in the proof of Theorem \ref{maintheo} can be easily generalized to yield the claimed limiting covariance structure. 
Hence, as already pointed out, the vector-valued limiting process $Z$ is an element of $C([0,1];\R^d)$.
Next, computations very similar to those leading to Theorem \ref{fiditheo} yield Conditions (a') and (b'). Finally, Condition (c') 
is obtained by observing that the laws of the family $\{ W^{(n)} : n\geq 1\}$ are tight whenever those of $\{ W^{(n)}(i) : n\geq 1\}$ are tight for each $1\leq i\leq d$; this last point is established by means of
Theorem \ref{gentightness}.
\end{proof}

\begin{proof}[Proof of Theorem \ref{mult1}]
From the expresssion \eqref{defpsis} it is a simple combinatorial exercise to deduce that 
\begin{align*}
\langle \psi_{k}(i),\psi_k(j)\rangle_{L^2(\mu^{k})}&=\sum_{q=0}^k (-1)^{k-q}\binom{k}{q}\langle g_{q}(i),g_q(j)\rangle_{L^2(\mu^{q})}\\
&= \sum_{q=1}^k (-1)^{k-q}\binom{k}{q}\langle \hat{g}_{q}(i),\hat{g}_q(j)\rangle_{L^2(\mu^{q})} \,,
\end{align*}
where $\hat{g}_{q}(i)=g_q(i)-\E[\psi(i)(X_1,\dotsc,X_{p_i})]$. Thus, Condition (a) is yielded by Condition (a') of Theorem \ref{mult2}. 
Finally, Conditions (b) and (c) are derived from Conditions (b') and (c') of Theorem \ref{mult2} by an application of Lemma \ref{genulemma}.
\end{proof}

\subsection{Proof of Theorem \ref{t:maincp}}\label{ss:dede}

For every $n$, define $\pi_n : [n] \to [n]$ to be the bijection given by $\pi_n(i) := n-i+1$, $i=1,...,n$, and also set $\beta(n,t):= n - \lfloor nt\rfloor +1$. We define, for $t\in [0,1]$,
$$
U_n(t) := \sum_{1\leq i<j\leq \lfloor nt\rfloor} \psi^{(n)}(X_i, X_j),
$$
and 
$$
I_n(t) := \sum_{ (i,j) : \pi_n(j) < \pi_n(i) < \beta(n,t)} \psi^{(n)}(X_{i}, X_j).
$$
Then, one has that $Y_n(t) = U_n(1) - U_n(t) - I_n(t)$, in such a way that the tightness of $\{\widetilde{Y}_n\}$ in $D[0,1]$ follows from a direct application of Theorem \ref{gentightness} first to $U_n/\gamma_n$ and then to $I_n/\gamma_n$. The asymptotic Gaussianity of the finite-dimensional distributions of $Y^{(n)}$ now follows from Remark \ref{r:lolo}, and one can check that the covariance function of $Y^{(n)}$ converges to that of $c_1 A + c_2 b$ by a direct computation.

\qed

\section*{Acknowledgements}
We thank two anonymous referees for several constructive remarks. We are grateful to Yannick Baraud, Andrew D. Barbour, Omar El-Dakkak and Ivan Nourdin for useful discussions. The research developed in this paper is supported by the FNR grant {\bf FoRGES (R-AGR-3376-10)} at Luxembourg University.


\bibliographystyle{imsart-number} 
\bibliography{Bibliography}       

\begin{thebibliography}{50}

\bibitem{AG93}
\begin{barticle}[author]
\bauthor{\bsnm{Arcones},~\bfnm{Miguel~A.}\binits{M.~A.}} \AND
  \bauthor{\bsnm{Gin\'e},~\bfnm{Evarist}\binits{E.}}
(\byear{1993}).
\btitle{{Limit theorems for $U$-processes}}.
\bjournal{Ann. Probab.}
\bvolume{21}
\bpages{1494-1542}.
\end{barticle}
\endbibitem

\bibitem{barbour2009}
\begin{barticle}[author]
\bauthor{\bsnm{Barbour},~\bfnm{A.}\binits{A.}} \AND
  \bauthor{\bsnm{Janson},~\bfnm{S.}\binits{S.}}
(\byear{2009}).
\btitle{{A Functional Combinatorial Central Limit Theorem}}.
\bjournal{Electron. J. Probab.}
\bvolume{14}
\bpages{2352--2370}.
\bdoi{10.1214/EJP.v14-709}
\end{barticle}
\endbibitem

\bibitem{barbour}
\begin{barticle}[author]
\bauthor{\bsnm{Barbour},~\bfnm{A.~D.}\binits{A.~D.}}
(\byear{1990}).
\btitle{{Stein's Method for Diffusion Approximations}}.
\bjournal{Probab. Theory and Relat. Fields}
\bvolume{84}
\bpages{297-322}.
\end{barticle}
\endbibitem

\bibitem{B94}
\begin{barticle}[author]
\bauthor{\bsnm{Basalykas},~\bfnm{A.}\binits{A.}}
(\byear{1994}).
\btitle{Functional limit theorems for random multilinear forms}.
\bjournal{Stoch. Proc. Appl.}
\bvolume{53}
\bpages{175-191}.
\end{barticle}
\endbibitem

\bibitem{BhaGo92}
\begin{barticle}[author]
\bauthor{\bsnm{Bhattacharya},~\bfnm{R.~N.}\binits{R.~N.}} \AND
  \bauthor{\bsnm{Ghosh},~\bfnm{J.~K.}\binits{J.~K.}}
(\byear{1992}).
\btitle{A class of {$U$}-statistics and asymptotic normality of the number of
  {$k$}-clusters}.
\bjournal{J. Multivariate Anal.}
\bvolume{43}
\bpages{300--330}.
\bdoi{10.1016/0047-259X(92)90038-H}
\bmrnumber{1193616}
\end{barticle}
\endbibitem

\bibitem{BR88}
\begin{barticle}[author]
\bauthor{\bsnm{Bickel},~\bfnm{P.~J.}\binits{P.~J.}} \AND
  \bauthor{\bsnm{Ritov},~\bfnm{Y.}\binits{Y.}}
(\byear{1988}).
\btitle{Estimating integrated squared density serivatives: sharp best order of
  convergence estimates}.
\bjournal{Sankhya Ser. A}
\bvolume{50}
\bpages{381-393}.
\end{barticle}
\endbibitem

\bibitem{Bil}
\begin{bbook}[author]
\bauthor{\bsnm{Billingsley},~\bfnm{P.}\binits{P.}}
(\byear{1999}).
\btitle{Convergence of probability measures},
\bedition{second} ed.
\bseries{Wiley Series in Probability and Statistics: Probability and
  Statistics}.
\bpublisher{John Wiley \& Sons, Inc., New York}
\bnote{A Wiley-Interscience Publication}.
\bdoi{10.1002/9780470316962}
\bmrnumber{1700749}
\end{bbook}
\endbibitem

\bibitem{Bor_book}
\begin{bbook}[author]
\bauthor{\bsnm{Borovskikh},~\bfnm{V.~Yu.}\binits{V.~Y.}}
(\byear{1996}).
\btitle{{U-Statistics in Banach Spaces}}.
\bpublisher{V.S.P. Intl Science}.
\end{bbook}
\endbibitem

\bibitem{CZ15}
\begin{barticle}[author]
\bauthor{\bsnm{Chen},~\bfnm{H.}\binits{H.}} \AND
  \bauthor{\bsnm{Zhang},~\bfnm{N.}\binits{N.}}
(\byear{2015}).
\btitle{Graph-based chenge-point detection}.
\bjournal{Ann. Stat.}
\bvolume{43}
\bpages{139-176}.
\end{barticle}
\endbibitem

\bibitem{chen}
\begin{barticle}[author]
\bauthor{\bsnm{Chen},~\bfnm{Xiaohui}\binits{X.}}
(\byear{2018}).
\btitle{{Gaussian and bootstrap approximations for high-dimensional
  U-statistics and their applications}}.
\bjournal{Ann. Statist.}
\bvolume{46}
\bpages{642--678}.
\end{barticle}
\endbibitem

\bibitem{CK2017}
\begin{barticle}[author]
\bauthor{\bsnm{Chen},~\bfnm{X.}\binits{X.}} \AND
  \bauthor{\bsnm{Kato},~\bfnm{K.}\binits{K.}}
(\byear{2020}).
\btitle{{Jackknife multiplier bootstrap: finite sample approximations to the
  $U$-process supremum with applications}}.
\bjournal{Probab. Theory Relat. Fields}
\bvolume{176}
\bpages{1097–1163}.
\end{barticle}
\endbibitem

\bibitem{CC19}
\begin{barticle}[author]
\bauthor{\bsnm{Chu},~\bfnm{L.}\binits{L.}} \AND
  \bauthor{\bsnm{Chen},~\bfnm{H.}\binits{H.}}
(\byear{2019}).
\btitle{Asymptotic distribution-free change-point detection for multivariate
  and non-Euclidean data}.
\bjournal{Ann. Stat.}
\bvolume{47}
\bpages{382-314}.
\end{barticle}
\endbibitem

\bibitem{CH88}
\begin{barticle}[author]
\bauthor{\bsnm{Cs\"org\"o},~\bfnm{M.}\binits{M.}} \AND
  \bauthor{\bsnm{Horvath},~\bfnm{L.}\binits{L.}}
(\byear{1988}).
\btitle{Invariance principles for changepoint problems}.
\bjournal{J. Multivariate Anal.}
\bvolume{27}
\bpages{151-168}.
\end{barticle}
\endbibitem

\bibitem{CH_book}
\begin{bbook}[author]
\bauthor{\bsnm{Cs\"org\"o},~\bfnm{M.}\binits{M.}} \AND
  \bauthor{\bsnm{Horvath},~\bfnm{L.}\binits{L.}}
(\byear{1997}).
\btitle{Limit theorems in changepoint analysis}.
\bpublisher{Wiley}.
\end{bbook}
\endbibitem

\bibitem{dej}
\begin{barticle}[author]
\bauthor{\bparticle{de} \bsnm{Jong},~\bfnm{P.}\binits{P.}}
(\byear{1990}).
\btitle{A central limit theorem for generalized multilinear forms}.
\bjournal{J. Multivariate Anal.}
\bvolume{34}
\bpages{275--289}.
\end{barticle}
\endbibitem

\bibitem{DlPG}
\begin{bbook}[author]
\bauthor{\bparticle{De~la} \bsnm{Pena},~\bfnm{V.~H.}\binits{V.~H.}} \AND
  \bauthor{\bsnm{Gin\'e},~\bfnm{E.}\binits{E.}}
(\byear{1999}).
\btitle{Decoupling: From Dependence to Independence. Springer, 1999.}
\bpublisher{Springer-Verlag}.
\end{bbook}
\endbibitem

\bibitem{dobler_peccati}
\begin{barticle}[author]
\bauthor{\bsnm{D{\"o}bler},~\bfnm{C.}\binits{C.}} \AND
  \bauthor{\bsnm{Peccati},~\bfnm{G.}\binits{G.}}
(\byear{2017}).
\btitle{{Quantitative de Jong Theorems in any dimension}}.
\bjournal{Electron. J. Probab.}
\bvolume{22}
\bpages{1--35}.
\end{barticle}
\endbibitem

\bibitem{DP18}
\begin{barticle}[author]
\bauthor{\bsnm{D\"obler},~\bfnm{C.}\binits{C.}} \AND
  \bauthor{\bsnm{Peccati},~\bfnm{G.}\binits{G.}}
(\byear{2019}).
\btitle{{Quantitative CLTs for symmetric $U$-statistics using contractions}}.
\bjournal{Electron. J. Probab.}
\bvolume{24}
\bpages{1-43}.
\bdoi{10.1214/19-EJP264}
\end{barticle}
\endbibitem

\bibitem{Dudley}
\begin{bbook}[author]
\bauthor{\bsnm{Dudley},~\bfnm{R.~M.}\binits{R.~M.}}
(\byear{2002}).
\btitle{Real analysis and probability}.
\bpublisher{Cambridge University Press}.
\end{bbook}
\endbibitem

\bibitem{DM}
\begin{barticle}[author]
\bauthor{\bsnm{Dynkin},~\bfnm{E.~B.}\binits{E.~B.}} \AND
  \bauthor{\bsnm{Mandelbaum},~\bfnm{A.}\binits{A.}}
(\byear{1983}).
\btitle{{Symmetric statistics, Poisson point processes, and multiple Wiener
  integrals.}}
\bjournal{Ann. Statist.}
\bvolume{11}
\bpages{739-745}.
\end{barticle}
\endbibitem

\bibitem{Ferger94}
\begin{barticle}[author]
\bauthor{\bsnm{Ferger},~\bfnm{D.}\binits{D.}}
(\byear{1994}).
\btitle{{An extension of the Cs{\"o}rg{\H o}-Horv{\'a}th functional limit
  theorem and Its applications to changepoint problems}}.
\bjournal{J. Multivariate Anal.}
\bvolume{51}
\bpages{338-351}.
\end{barticle}
\endbibitem

\bibitem{Ferger95}
\begin{barticle}[author]
\bauthor{\bsnm{Ferger},~\bfnm{D.}\binits{D.}}
(\byear{1995}).
\btitle{{The joint distribution of the running maximum and its location of
  $D$-valued Markov processes.}}
\bjournal{Journal of Applied Probability}
\bvolume{32}
\bpages{842-845}.
\end{barticle}
\endbibitem

\bibitem{Ferger99}
\begin{barticle}[author]
\bauthor{\bsnm{Ferger},~\bfnm{D.}\binits{D.}}
(\byear{1999}).
\btitle{On the uniqueness of maximizers of Markov-Gaussian processes}.
\bjournal{Stat. Probab. Letters}
\bvolume{45}
\bpages{71-77}.
\end{barticle}
\endbibitem

\bibitem{Ferger01}
\begin{barticle}[author]
\bauthor{\bsnm{Ferger},~\bfnm{D.}\binits{D.}}
(\byear{2001}).
\btitle{Analysis of change-point estimators under the null hypothesis}.
\bjournal{Bernoulli}
\bvolume{7}
\bpages{487-506}.
\end{barticle}
\endbibitem

\bibitem{GM07}
\begin{barticle}[author]
\bauthor{\bsnm{Gin\'e},~\bfnm{E.}\binits{E.}} \AND
  \bauthor{\bsnm{Mason},~\bfnm{D.}\binits{D.}}
(\byear{2007}).
\btitle{{On local $U$-statistic processes and the estimation of densities of
  functons of several sample variables}}.
\bjournal{Ann. Statist.}
\bvolume{35}
\bpages{1105-1145}.
\end{barticle}
\endbibitem

\bibitem{Gombay2004}
\begin{barticle}[author]
\bauthor{\bsnm{Gombay},~\bfnm{E.}\binits{E.}}
(\byear{2004}).
\btitle{{U-Statistics in Sequential Tests and Change Detection}}.
\bjournal{Sequential Analysis: Design Methods and Applications}
\bvolume{23}
\bpages{257-274}.
\end{barticle}
\endbibitem

\bibitem{GH95}
\begin{barticle}[author]
\bauthor{\bsnm{Gombay},~\bfnm{E.}\binits{E.}} \AND
  \bauthor{\bsnm{Horvath},~\bfnm{L.}\binits{L.}}
(\byear{1995}).
\btitle{{An application of U-statistics to change-point analysis}}.
\bjournal{Acta Sci. Math. (Szeged)}
\bvolume{60}
\bpages{345-357}.
\end{barticle}
\endbibitem

\bibitem{H79}
\begin{barticle}[author]
\bauthor{\bsnm{Hall},~\bfnm{P.}\binits{P.}}
(\byear{1979}).
\btitle{{On the invariance principle for $U$-statistics}}.
\bjournal{Stoch. Proc. Appl.}
\bvolume{9}
\bpages{163-174}.
\end{barticle}
\endbibitem

\bibitem{HR_survey}
\begin{barticle}[author]
\bauthor{\bsnm{Horvath},~\bfnm{L.}\binits{L.}} \AND
  \bauthor{\bsnm{Rice},~\bfnm{G.}\binits{G.}}
(\byear{2014}).
\btitle{Extensions of some classical methods in change point analysis}.
\bjournal{TEST}
\bvolume{23}
\bpages{219-255}.
\end{barticle}
\endbibitem

\bibitem{ISH}
\begin{barticle}[author]
\bauthor{\bsnm{Ibragimov},~\bfnm{R.}\binits{R.}} \AND
  \bauthor{\bsnm{Sharakhmetov},~\bfnm{Sh.}\binits{S.}}
(\byear{2002}).
\btitle{Bounds on moments of symmetric statistics}.
\bjournal{Studia Sci. Math. Hungar.}
\bvolume{39}
\bpages{251--275}.
\bdoi{10.1556/SScMath.39.2002.3-4.1}
\bmrnumber{1956938}
\end{barticle}
\endbibitem

\bibitem{JJ}
\begin{barticle}[author]
\bauthor{\bsnm{Jammalamadaka},~\bfnm{S.~R.}\binits{S.~R.}} \AND
  \bauthor{\bsnm{Janson},~\bfnm{S.}\binits{S.}}
(\byear{1986}).
\btitle{Limit theorems for a triangular scheme of {$U$}-statistics with
  applications to inter-point distances}.
\bjournal{Ann. Probab.}
\bvolume{14}
\bpages{1347--1358}.
\bmrnumber{866355}
\end{barticle}
\endbibitem

\bibitem{J18}
\begin{barticle}[author]
\bauthor{\bsnm{Janson},~\bfnm{S.}\binits{S.}}
(\byear{2018}).
\btitle{{Renewal theory of asymmetric $U$-statistics}}.
\bjournal{Electron. J. Probab.}
\bvolume{23}
\bpages{1-27}.
\end{barticle}
\endbibitem

\bibitem{KR}
\begin{barticle}[author]
\bauthor{\bsnm{Karlin},~\bfnm{S.}\binits{S.}} \AND
  \bauthor{\bsnm{Rinott},~\bfnm{Y.}\binits{Y.}}
(\byear{1982}).
\btitle{{Applications of Anova Type Decompositions for Comparisons of
  Conditional Variance Statistics Including Jackknife Estimates}}.
\bjournal{Ann. Statist.}
\bvolume{10}
\bpages{485-501}.
\end{barticle}
\endbibitem

\bibitem{kasprzak2018}
\begin{barticle}[author]
\bauthor{\bsnm{Kasprzak},~\bfnm{M.~J.}\binits{M.~J.}}
(\byear{2020}).
\btitle{Stein's method for multivariate {B}rownian approximations of sums under
  dependence}.
\bjournal{Stochastic Process. Appl.}
\bvolume{130}
\bpages{4927--4967}.
\bdoi{10.1016/j.spa.2020.02.006}
\bmrnumber{4108478}
\end{barticle}
\endbibitem

\bibitem{LRP2}
\begin{barticle}[author]
\bauthor{\bsnm{Lachi{{\`e}}ze-Rey},~\bfnm{R.}\binits{R.}} \AND
  \bauthor{\bsnm{Peccati},~\bfnm{G.}\binits{G.}}
(\byear{2013}).
\btitle{Fine {G}aussian fluctuations on the {P}oisson space {II}: rescaled
  kernels, marked processes and geometric {$U$}-statistics}.
\bjournal{Stochastic Process. Appl.}
\bvolume{123}
\bpages{4186--4218}.
\bdoi{10.1016/j.spa.2013.06.004}
\bmrnumber{3096352}
\end{barticle}
\endbibitem

\bibitem{L1}
\begin{barticle}[author]
\bauthor{\bsnm{Laurent},~\bfnm{B.}\binits{B.}}
(\byear{1996}).
\btitle{Efficient estimation of integral functionals of a density}.
\bjournal{Ann. Stat.}
\bvolume{24}
\bpages{659-681}.
\end{barticle}
\endbibitem

\bibitem{L2}
\begin{barticle}[author]
\bauthor{\bsnm{Laurent},~\bfnm{B.}\binits{B.}}
(\byear{1997}).
\btitle{Estimation of integral functionals of a density and its derivatives}.
\bjournal{Bernoulli}
\bvolume{3}
\bpages{181-211}.
\end{barticle}
\endbibitem

\bibitem{LM00}
\begin{barticle}[author]
\bauthor{\bsnm{Laurent},~\bfnm{B.}\binits{B.}} \AND
  \bauthor{\bsnm{Massart},~\bfnm{P.}\binits{P.}}
(\byear{2000}).
\btitle{Adaptive estimation of a quadratic functional by model selection}.
\bjournal{Ann. Statist.}
\bvolume{28}
\bpages{1302-1338}.
\end{barticle}
\endbibitem

\bibitem{Major}
\begin{bbook}[author]
\bauthor{\bsnm{Major},~\bfnm{P.}\binits{P.}}
(\byear{2013}).
\btitle{On the estimation of multiple random integrals and {$U$}-statistics}.
\bseries{Lecture Notes in Mathematics}
\bvolume{2079}.
\bpublisher{Springer, Heidelberg}.
\bdoi{10.1007/978-3-642-37617-7}
\bmrnumber{3087566}
\end{bbook}
\endbibitem

\bibitem{MT}
\begin{barticle}[author]
\bauthor{\bsnm{Mandelbaum},~\bfnm{A.}\binits{A.}} \AND
  \bauthor{\bsnm{Taqqu},~\bfnm{M.~S.}\binits{M.~S.}}
(\byear{1984}).
\btitle{{Invariance Principle for Symmetric Statistics}}.
\bjournal{Ann. Statist.}
\bvolume{12}
\bpages{483-496}.
\end{barticle}
\endbibitem

\bibitem{MS72}
\begin{barticle}[author]
\bauthor{\bsnm{Miller},~\bfnm{R.~G.}\binits{R.~G.}} \AND
  \bauthor{\bsnm{Sen},~\bfnm{P.~K.}\binits{P.~K.}}
(\byear{1972}).
\btitle{Weak convergence of $U$-statistics and von Mises' differentiable
  statistical functions}.
\bjournal{Ann. Math. Stat.}
\bvolume{43}
\bpages{31-41}.
\end{barticle}
\endbibitem

\bibitem{Neu}
\begin{barticle}[author]
\bauthor{\bsnm{Neuhaus},~\bfnm{G.}\binits{G.}}
(\byear{1977}).
\btitle{Functional limit theorems for {$U$}-statistics in the degenerate case}.
\bjournal{J. Multivariate Anal.}
\bvolume{7}
\bpages{424--439}.
\bdoi{10.1016/0047-259X(77)90083-5}
\bmrnumber{0455084}
\end{barticle}
\endbibitem

\bibitem{DP87}
\begin{barticle}[author]
\bauthor{\bsnm{Nolan},~\bfnm{D.}\binits{D.}} \AND
  \bauthor{\bsnm{Pollard},~\bfnm{D.}\binits{D.}}
(\byear{1987}).
\btitle{{Functional limit theorems for $U$-processes}}.
\bjournal{Ann. Statist.}
\bvolume{15}
\bpages{780-799}.
\end{barticle}
\endbibitem

\bibitem{NouNu}
\begin{barticle}[author]
\bauthor{\bsnm{Nourdin},~\bfnm{I.}\binits{I.}} \AND
  \bauthor{\bsnm{Nualart},~\bfnm{D.}\binits{D.}}
(\byear{2020}).
\btitle{{The functional Breuer-Major theorem}}.
\bjournal{Probab. Theory Relat. Fields}
\bvolume{176}
\bpages{203–218}.
\end{barticle}
\endbibitem

\bibitem{PR_book}
\begin{bbook}[author]
\beditor{\bsnm{Peccati},~\bfnm{G.}\binits{G.}} \AND
  \beditor{\bsnm{Reitzner},~\bfnm{M.}\binits{M.}}, eds.
(\byear{2016}).
\btitle{{Stochastic analysis for Poisson point processes}}.
\bpublisher{Springer-Verlag}.
\end{bbook}
\endbibitem

\bibitem{Penrose}
\begin{bbook}[author]
\bauthor{\bsnm{Penrose},~\bfnm{M.}\binits{M.}}
(\byear{2004}).
\btitle{{Geometric Random Graphs}}.
\bpublisher{Oxford}.
\end{bbook}
\endbibitem

\bibitem{RW19}
\begin{barticle}[author]
\bauthor{\bsnm{Ra{\v c}kauskas},~\bfnm{A.}\binits{A.}} \AND
  \bauthor{\bsnm{Wendler},~\bfnm{M.}\binits{M.}}
(\byear{2019}).
\btitle{{Convergence of U-Processes in H\"older Spaces with Application to
  Robust Detection of a Changed Segment}}.
\bjournal{Arxiv Preprint}.
\end{barticle}
\endbibitem

\bibitem{van_der_vaart}
\begin{barticle}[author]
\bauthor{\bsnm{Robins},~\bfnm{J.}\binits{J.}},
  \bauthor{\bsnm{Li},~\bfnm{L.}\binits{L.}},
  \bauthor{\bsnm{Tchetgen},~\bfnm{E.}\binits{E.}} \AND
  \bauthor{\bparticle{van~der} \bsnm{Vaart},~\bfnm{A.}\binits{A.}}
(\byear{2016}).
\btitle{{Asymptotic Normality of Quandratic Estimators}}.
\bjournal{Stoch. Processes Their Appl.}
\bvolume{126}
\bpages{3733-3759}.
\end{barticle}
\endbibitem

\bibitem{serfling}
\begin{bbook}[author]
\bauthor{\bsnm{Serfling},~\bfnm{R.~J.}\binits{R.~J.}}
(\byear{1980}).
\btitle{{Approximation Theorems in Mathematical Statistics}}.
\bpublisher{Wiley}.
\end{bbook}
\endbibitem

\bibitem{Vit92}
\begin{barticle}[author]
\bauthor{\bsnm{Vitale},~\bfnm{R.~A.}\binits{R.~A.}}
(\byear{1992}).
\btitle{Covariances of symmetric statistics}.
\bjournal{J. Multivariate Anal.}
\bvolume{41}
\bpages{14--26}.
\bdoi{10.1016/0047-259X(92)90054-J}
\bmrnumber{1156678}
\end{barticle}
\endbibitem

\end{thebibliography}

%
%
%

\end{document}